\documentclass[11pt]{amsart}
\usepackage{hyperref}
\usepackage{amsfonts}
\usepackage{amsmath}
\usepackage{amssymb}
\usepackage{amscd}
\usepackage{graphicx}
\usepackage{enumitem}
\usepackage{latexsym}
\usepackage{color}
\usepackage{comment}
\usepackage{epsfig}
\usepackage{amsopn}
\usepackage{xypic}
\usepackage{mathrsfs}
\usepackage{array}
\usepackage{rotating}
\usepackage[usenames,dvipsnames]{pstricks}
\usepackage{epsfig}
\usepackage{pst-grad} 
\usepackage{pst-plot} 
\usepackage[space]{grffile} 
\usepackage{etoolbox} 
\makeatletter 
\patchcmd\Gread@eps{\@inputcheck#1 }{\@inputcheck"#1"\relax}{}{}
\makeatother

\usepackage{booktabs}
\usepackage{longtable}
\theoremstyle{plain}
\newtheorem{lemma}{Lemma}[section]
\newtheorem*{theorem*}{Theorem}
\newtheorem*{lemma*}{Lemma}
\newtheorem*{proposition*}{Proposition}
\newtheorem*{conjecture*}{Conjecture}
\newtheorem*{corollary*}{Corollary}
\newtheorem*{problem*}{Problem}
\newtheorem{theorem}[lemma]{Theorem}
\newtheorem{conjecture}[lemma]{Conjecture}
\newtheorem{corollary}[lemma]{Corollary}
\newtheorem{proposition}[lemma]{Proposition}

\newtheorem{problem}[lemma]{Problem}

\newtheorem{propositionDefinition}[lemma]{Proposition-Definition}
\theoremstyle{definition}
\newtheorem{definition}[lemma]{Definition}
\newtheorem{example}[lemma]{Example}
\newtheorem{remark}[lemma]{Remark}

\newcommand{\F}[1]{\mathscr{#1}}

\newcommand{\Z}{\mathbb{Z}}

\newcommand{\FF}{\mathbb{F}}
\renewcommand{\F}{\mathbb{F}}
\newcommand{\CC}{\mathbb{C}}
\newcommand{\QQ}{\mathbb{Q}}
\newcommand{\RR}{\mathbb{R}}

\newcommand{\OO}{\mathcal{O}}
\newcommand{\te}{\otimes}
\newcommand{\sm}{\setminus}

\newcommand{\id}{\mathrm{id}}
\newcommand{\gen}{\mathrm{gen}}

\newcommand{\cF}{\mathcal F}
\newcommand{\cQ}{\mathcal Q}
\newcommand{\onto}{\twoheadrightarrow}
\newcommand{\cA}{\mathcal A}
\newcommand{\cM}{\mathcal M}
\newcommand{\cK}{\mathcal K}

\newcommand{\cP}{\mathcal P}

\newcommand{\cU}{\mathcal U}
\newcommand{\bv}{{\bf v}}
\newcommand{\be}{{\bf e}}

\newcommand{\bk}{{\bf k}}
\newcommand{\bl}{{\bf l}}
\newcommand{\bw}{{\bf w}}
\newcommand{\bgr}{{\bf gr}}

\newcommand{\cE}{\mathcal{E}}
\newcommand{\cV}{\mathcal{V}}
\newcommand{\cW}{\mathcal{W}}

\newcommand{\gr}{\mathrm{gr}}
\newcommand{\mus}{\mu\textrm{-}s}
\newcommand{\muss}{\mu\textrm{-}ss}

\newcommand{\ZZ}{\mathbb{Z}}

\renewcommand{\P}{\mathbb{P}}
\renewcommand{\cL}{\mathcal{L}}
\newcommand{\PP}{\mathbb{P}}

\DeclareMathOperator{\ch}{ch}
\DeclareMathOperator{\DLP}{DLP}

\DeclareMathOperator{\Hom}{Hom}

\DeclareMathOperator{\Pic}{Pic}

\DeclareMathOperator{\rk}{rk}
\DeclareMathOperator{\length}{length}

\DeclareMathOperator{\Ext}{Ext}
\DeclareMathOperator{\ext}{ext}

\DeclareMathOperator{\Coh}{Coh}
\DeclareMathOperator{\sHom}{\mathcal{H}\kern -.5pt\mathit{om}}
\DeclareMathOperator{\sTor}{\mathcal{T}\kern -1.5pt\mathit{or}}

\DeclareMathOperator{\Quot}{Quot}

\newcommand{\leqor}{\underset{{\scriptscriptstyle (}-{\scriptscriptstyle )}}{<}}

\begin{document}

\date{\today}
\author[I. Coskun]{Izzet Coskun}
\address{Department of Mathematics, Statistics and CS \\University of Illinois at Chicago, Chicago, IL 60607}
\email{coskun@math.uic.edu}

\author[J. Huizenga]{Jack Huizenga}
\address{Department of Mathematics, The Pennsylvania State University, University Park, PA 16802}
\email{huizenga@psu.edu}

\thanks{During the preparation of this article the first author was partially supported by the  NSF grant DMS-1500031 and NSF FRG grant  DMS 1664296 and  the second author was partially supported  by the NSA\ Young Investigator Grant H98230-16-1-0306 and NSF FRG grant DMS 1664303}

\title{Existence of semistable sheaves on Hirzebruch surfaces}

\begin{abstract}

Let $\FF_e$ denote the Hirzebruch surface $\PP(\OO_{\PP^1} \oplus \OO_{\PP^1}(e))$, and let $H$ be any ample divisor. In this paper, we algorithmically determine when the moduli space of semistable sheaves $M_{\FF_e,H}(r,c_1,c_2)$ is nonempty.  Our algorithm relies on certain stacks of prioritary sheaves. We first solve the existence problem for these stacks and then algorithmically determine the  Harder-Narasimhan filtration of the general sheaf in the stack. In particular, semistable sheaves exist if and only if the Harder-Narasimhan filtration has length one. 

We then study sharp Bogomolov inequalities $\Delta \geq \delta_H(c_1/r)$ for the discriminants of stable sheaves which take the polarization and slope into account; these inequalities essentially completely describe the characters of stable sheaves.  The function $\delta_H(c_1/r)$ can be computed to arbitrary precision by a limiting procedure.  In the case of an anticanonically polarized del Pezzo surface, exceptional bundles are always stable and $\delta_H(c_1/r)$ is computed by exceptional bundles.  More generally, we show that for an arbitrary polarization there are further necessary conditions for the existence of stable sheaves beyond those provided by stable exceptional bundles.  We compute $\delta_H(c_1/r)$ exactly in some of these cases.  Finally, solutions to the existence problem have immediate applications to the birational geometry of $M_{\F_e,H}(\bv)$.
\end{abstract}

\maketitle

\setcounter{tocdepth}{1}
\tableofcontents

\section{Introduction}\label{sec-Intro}
Let $X$ be a smooth, complex projective surface and let $H$ be an ample divisor on $X$.  The moduli spaces of sheaves $M_{X, H}({\bf v})$ parameterizing $S$-equivalence classes of Gieseker semistable sheaves on $X$ play a fundamental role in mathematics, especially in algebraic geometry, Donaldson's theory of four-manifolds \cite{Donaldson} and in mathematical physics \cite{Witten}. Despite being intensively studied, many basic questions concerning the geometry of $M_{X, H}({\bf v})$ remain open. Central among them is determining when $M_{X, H}({\bf v})$ is nonempty. 

We say that a Chern character ${\bf v}$ is {\em $H$-semistable} if there exists an  $H$-Gieseker semistable sheaf with Chern character ${\bf v}$; or equivalently, when $M_{X, H}({\bf v})$ is nonempty. Bogomolov's inequality says that if ${\bf v}$ is an $H$-semistable Chern character for any ample $H$, then the discriminant satisfies $\Delta({\bf v}) \geq 0$. On the other hand, O'Grady's theorem \cite{OGrady} guarantees the existence of $H$-semistable sheaves provided $\Delta(\bv) \gg 0$. However, a complete classification of $H$-semistable Chern characters is only known for a handful of surfaces such as  $\PP^2$, K3 surfaces and abelian surfaces (see \cite{CoskunHuizengaGokova, DLP, HuybrechtsLehn, LePotier} for detailed descriptions and references). 

There are many  results on classifying stable Chern characters on anticanonically polarized del Pezzo surfaces. Rudakov \cite{RudakovQuadric}, Gorodentsev \cite{Gorodentsev} and Kuleshov and Orlov \cite{KuleshovOrlov} study exceptional bundles on del Pezzo surfaces and prove many foundational results. Rudakov studies the stable Chern characters for the anticanonical polarization on $\PP^1 \times \PP^1$ and more generally del Pezzo surfaces \cite{Rudakov, Rudakov2} and shows that exceptional bundles control the classification. These results on del Pezzo surfaces form the starting point of our investigations. 

Let $\FF_e$ for $e \geq 0$ denote the Hirzebruch surface $\PP(\OO_{\PP^1} \oplus \OO_{\PP^1}(e))$. Let $F$ be the fiber of the natural projection to $\PP^1$ and let $E$ denote a section of self-intersection $-e$. Let $H$ be any ample divisor on $\FF_e$ and ${\bf v}$ be a positive rank Chern character.  In this paper, we algorithmically determine the Chern characters of $H$-Gieseker semistable sheaves on $\FF_e$; equivalently, we determine when $M_{\FF_e, H}({\bf v})$ is nonempty. For fixed rank $r$ and $c_1$, we obtain sharp Bogomolov inequalities on the discriminant $\Delta$ for the existence of semistable sheaves. These results have consequences for computing birational invariants such as ample and effective cones on $M_{\FF_e, H}({\bf v})$ (see \cite{CoskunHuizengaNef}). Our work is directly inspired by Dr\'{e}zet and Le Potier's classification of stable Chern characters on $\PP^2$; however, the classification is considerably more complicated. In particular, we will see that for an arbitrary polarization, exceptional bundles do not control the complete classification.

For any real number $m$, let $$H_m = E + (e+m)F.$$ The divisor  class $H_m$ is ample if $m>0$, and every ample divisor is a multiple of some $H_m$. Since scaling a divisor does not change semistability, to study semistable sheaves with respect to any ample divisor, it suffices to study $H_m$-semistable sheaves on $\FF_e$.  We will now explain our results in greater detail.

\subsection{Sharp Bogomolov inequalities}  
Recall that for a smooth surface $X$ the \emph{total slope} $\nu$ and \emph{discriminant} $\Delta$ of a Chern character $\bv \in K(X)$ are defined by $$ \nu = \frac{c_1}{r} \qquad \Delta = \frac{1}{2}\nu^2 - \frac{\ch_2}{r}.$$  We will often record a Chern character $\bv$ by the data $(r,\nu,\Delta)$.  The classical Bogomolov inequality says that if $\cV$ is a $\mu_H$-semistable sheaf, then $\Delta(\cV) \geq 0$.  By taking the polarization and the total slope into account, one can prove stronger lower bounds on the discriminant.

\begin{theorem}[See Theorem \ref{thm-deltaSurface}]
Let $m>0$.  There is a unique real-valued function $\delta_m^{\mus}(\nu)$ of the total slope $\nu$ with the following property.  Let $\bv = (r,\nu,\Delta)\in K(\F_e)$ be a character of positive rank.  
\begin{enumerate}
\item If $\Delta > \delta_m^{\mus}(\nu)$, then there are $\mu_{H_m}$-stable sheaves of character $\bv$.
\item If there is a non-exceptional $\mu_{H_m}$-stable sheaf of character $\bv$, then $\Delta \geq \delta_m^{\mus}(\nu)$.
\end{enumerate} 
\end{theorem}

Thus the inequality $\Delta \geq \delta_m^{\mus}(\nu)$ is a sharp Bogomolov inequality which is satisfied by all non-exceptional $\mu_{H_m}$-stable sheaves.    We show that the function $\delta_m^{\mus}(\nu)$ exists and we study its basic properties.  We also compute it in many cases.

\begin{enumerate}
\item For a fixed character $\bv\in K(\F_e)$, we give an explicit algorithm to determine if there are $\mu_{H_m}$-stable sheaves of character $\bv$.  This algorithm can be used to compute $\delta_m^{\mus}(\nu)$ as a limit and approximate it to arbitrary precision.  See \S\ref{ssec-introHN}.

\item If $e=0$ or $1$, then $\F_e$ is a del Pezzo surface, and we can consider the anticanonical polarization $H_{1-\frac{e}{2}} = -\frac{1}{2} K_{\F_e}$.  In this case, we expand on work of Rudakov to show that the function $\delta_{1-\frac{e}{2}}^{\mus}(\nu)$ can be computed by \emph{exceptional bundles}.    See \S \ref{ssec-introDLP}.

\item For certain polarizations $H_m$ different from the anticanonical polarization, the function $\delta^{\mus}_{m}(\nu)$ is not always computed by exceptional bundles.  We will show that $\delta_m^{\mus}(\nu)$ is sometimes computed by orthogonal pairs of \emph{Kronecker modules}.  See \S \ref{ssec-introKronecker}.
\end{enumerate}

\subsection{Prioritary sheaves} We now turn to the problem of constructing semistable sheaves.  In order to construct semistable sheaves it is convenient to have an irreducible family of sheaves that contains all the $H_m$-semistable sheaves.  Recall that for a divisor $D$ on $X$, a torsion-free coherent sheaf $\cV$ on a surface $X$ is called {\em $D$-prioritary} if $$\Ext^2(\cV, \cV(-D))=0.$$ For a character $\bv \in K(\F_e)$, let $\cP_D(\bv) \subset \Coh(\bv)$ be the open substack of $D$-prioritary sheaves.

The stack $\cP_F(\bv)$ of $F$-prioritary sheaves is irreducible by a theorem of Walter \cite{Walter}, and it is nonempty whenever the Bogomolov inequality $\Delta \geq 0$ holds.  Every $H_m$-semistable sheaf is $F$-prioritary.  Furthermore, every $H_m$-semistable sheaf is additionally $H_{\lceil m\rceil +1}$-prioritary, and $\cP_F(\bv)$ contains the stack $\cP_{H_{\lceil m\rceil+1 }}(\bv)$ as an open substack.  See \S\ref{sec-Prioritary} for details. Thus, if $M_{H_m}(\bv)$ is nonempty, we have a chain of open dense substacks $$\cM_{H_m}(\bv) \subset \cP_{H_{\lceil m\rceil+1 }}(\bv) \subset \cP_F(\bv).$$  The existence question for $M_{H_m}(\bv)$ can then be reduced to two separate questions.
\begin{enumerate} 
\item When is $\cP_{H_{\lceil m\rceil+1}}(\bv)$ nonempty?
\item If $\cP_{H_{\lceil m \rceil+1}}(\bv)$ is nonempty, is the general sheaf $H_m$-semistable?
\end{enumerate}
We study the first question in \S\ref{sec-existPrioritary}. The next theorem gives a complete answer.
\begin{theorem}[Proposition \ref{prop-triangle} and Corollary \ref{cor-prioritaryDelta}]
For a positive integer $n$, there is an explicitly computable function $\delta_n^p(\nu)$ with the following property.  Let $\bv = (r,\nu,\Delta) \in K(\F_e)$ be a character with $\Delta \geq 0$.  Then $\cP_{H_n}(\bv)$ is nonempty if and only if $\Delta \geq \delta_n^p(\nu)$.
\end{theorem}

There are two key previous results from \cite{CoskunHuizengaBN} used in the proof.  First, we know that the general sheaf $\cV$ in $\cP_F(\bv)$ admits a Gaeta-type resolution by line bundles.  We use this resolution to study $\Ext^2(\cV,\cV(-H_n))$.  We know  the cohomology of a general $F$-prioritary sheaf, and this gives us necessary numerical conditions on $\bv$ in order to have $\Ext^2(\cV,\cV(-H_n))=0$.  On the other hand, given a character $\bv$ that satisfies these conditions, we explicitly construct $H_n$-prioritary sheaves as direct sums of line bundles. 

\begin{remark}The theorem can also be rephrased in several other different ways.  For example, given a general $F$-prioritary sheaf $\cV$ of character $\bv$, what is the largest integer $n$ such that $\cV$ is $H_n$-prioritary? See Corollary \ref{cor-prioritaryRho}.
\end{remark}

\begin{remark}It follows that $\delta_n^p$ provides a strong Bogomolov inequality, in the sense that $$\delta_m^{\mus}(\nu) \geq \delta_{\lceil m\rceil +1}^p(\nu).$$ The right hand side has the advantage of being easily computable.
\end{remark}

\subsection{The generic Harder-Narasimhan filtration}\label{ssec-introHN}
To determine when the stack $\cP_{H_{\lceil m\rceil }}(\bv)$ contains $H_m$-semistable sheaves, we then study the $H_m$-Harder-Narasimhan filtration of the general sheaf.  In particular, there exists an $H_m$-semistable sheaf with Chern character ${\bf v}$ if and only if the generic $H_m$-Harder-Narasimhan filtration has length 1.

\begin{remark}
If the substack $\cP_{H_{\lceil m\rceil+1}}(\bv) \subset \cP_{H_{\lceil m\rceil}}(\bv)$ is empty, then there are definitely not $H_m$-semistable sheaves of character $\bv$, but we can still study the generic $H_m$-Harder-Narasimhan filtration of sheaves in $\cP_{H_{\lceil m\rceil}}(\bv)$.
\end{remark}

Suppose the $H_m$-Harder-Narasimhan factors of the general sheaf in $\cP_{H_{\lceil m\rceil}}(\bv)$ are $\bv_1,\ldots,\bv_\ell$.  They then must satisfy three key properties.  First, $M_{H_m}(\bv_i)$ is nonempty.  Second, prioritariness shows the general sheaf restricts to a curve of class $H_{\lceil m\rceil}$ (which is a $\P^1$) or $H_{\lfloor m\rfloor}$ as a balanced direct sum of line bundles.  This implies that the $H_m$-slopes of the factors are close: in particular, $$\mu_{H_m}(\bv_1) - \mu_{H_m}(\bv_\ell) \leq 1.$$  This then gives us $\Ext^2$-vanishings which allow us to compute the codimension of the Schatz stratum parameterizing sheaves with $H_m$-Harder-Narasimhan factors $\bv_1,\ldots,\bv_\ell$.  Since this codimension must be $0$, we get the orthogonality relations $$\chi(\bv_i,\bv_j) = 0 \qquad (i<j).$$
Conversely, we show that the Harder-Narasimhan factors are completely determined by these three properties.  

\begin{theorem}[Lemmas \ref{lem-HNclose} \& \ref{lem-HNorthogonal} and Theorem \ref{thm-HNcriterion}]\label{thm-HNintro}
Let $\bv \in K(\F_e)$, suppose there are $H_{\lceil m\rceil}$-prioritary sheaves of character $\bv$, and let $\bv = \bv_1 + \cdots + \bv_\ell$ be a decomposition with the following properties.
\begin{enumerate}
\item $q_1>\cdots > q_\ell$, where $q_i$ is the reduced $H_m$-Hilbert polynomial of $\bv_i$.
\item The moduli space $M_{H_m}(\bv_i)$ is nonempty.
\item $\mu_{H_m}(\bv_1) - \mu_{H_m}(\bv_\ell) \leq 1$.
\item $\chi(\bv_i,\bv_j)= 0$ for $i<j$.
\end{enumerate}
Then the general sheaf $\cV$ in $\cP_{H_{\lceil m\rceil}}(\bv)$ has an $H_m$-Harder-Narasimhan filtration with factors of characters $\bv_1,\ldots,\bv_\ell$.  Conversely, the characters of the factors of the $H_m$-Harder-Narasimhan filtration of $\cV$ have all the above properties.
\end{theorem}

Thus if there is no length $\ell \geq 2$ list of characters with these properties, then there are $H_m$-semistable sheaves.  We also show that the characters $\bv_i$ necessarily come from a bounded region in $K(\F_e)$, so there are only finitely many possibilities to consider.  The length $\ell$ of the filtration is also at most $4$.  We can then turn Theorem \ref{thm-HNintro} into an effective inductive algorithm to determine when the moduli space $M_{H_m}(\bv)$ is nonempty.  See Corollary \ref{cor-algorithm} for our most efficient procedure.

\begin{theorem}
For a fixed Chern character $\bv\in K(\F_e)$, there is an inductive algorithm to determine if the moduli space $M_{H_m}(\bv)$ is nonempty.  The number $\delta_m^{\mus}(\nu)$ can be computed to arbitrary precision as a limit.
\end{theorem}

\subsection{Exceptional bundles and the Dr\'ezet-Le Potier function}\label{ssec-introDLP}

 Recall that an exceptional bundle is a simple bundle $\cV$ with $\Ext^i(\cV,\cV) = 0$ for $i>0$.  Stable exceptional bundles give necessary conditions which the invariants of semistable sheaves must satisfy.  Suppose $\cV$ is a $\mu_H$-stable exceptional bundle, and suppose $\cW$ is a $\mu_H$-stable sheaf of character $\bw = (r,\nu,\Delta)$  with $$\frac{1}{2}K_{\F_e}\cdot H \leq \mu_H(\cW) - \mu_H(\cV) <0.$$ Then the only possibly nonzero group $\Ext^i(\cV,\cW)$ is $\Ext^1(\cV,\cW)$, so $\chi(\cV,\cW) \leq 0$.  When expanded with Riemann-Roch, this gives a lower bound on $\Delta$.  If the order of the $H$-slopes of $\cW$ and $\cV$ is reversed, an analogous discussion holds, and we conclude that there is an inequality of the form $$\Delta \geq \DLP_{H,\cV}(\nu)$$ which is satisfied by every $\mu_H$-stable sheaf $\cW$ with $0 < | \mu_H(\cW) - \mu_H(\cV)| \leq -\frac{1}{2}K_{\F_e}\cdot H.$ See \S \ref{ssec-DLP} for more details.

For simplicity, suppose $\cW$ does not have the same $H$-slope as any exceptional bundle.  If we take all the $\mu_H$-stable exceptional bundles $\cV$ which are sufficiently close to $\cW$ into account, we can define a \emph{Dr\'ezet-Le Potier} function $\DLP_H(\nu)$ which is the supremum of the functions $\DLP_{H,\cV}$. We will then have $$\Delta \geq \DLP_{H}(\nu),$$ and it follows that $$\delta_m^{\mus}(\nu) \geq \DLP_{H_m}(\nu);$$ see Corollary \ref{cor-deltaDLPe} in general.  Thus the Dr\'ezet-Le Potier function gives a stronger Bogomolov inequality, and can be used to bound the sharp Bogomolov inequality $\delta_m^{\mus}(\nu)$.  

If $e=0$ or $1$, then every exceptional bundle on $\F_e$ is $\mu_{-K_{\F_e}}$-stable by Gorodentsev \cite{Gorodentsev}.  Expanding on work of Rudakov \cite{Rudakov}, we see that  exceptional bundles precisely compute $\delta_{1-\frac{e}{2}}^{\mus}$.

\begin{theorem}[Corollary \ref{cor-deltaDLP}]
Let $e=0$ or $1$, and let $\nu \in \Pic(\F_e)\te \QQ$.  Then $$\delta_{1-\frac{e}{2}}^{\mus}(\nu) = \DLP_{-K_{\F_e}}(\nu).$$
\end{theorem}

\begin{example}
The analogous result holds for $\P^2$ by work of Dr\'ezet and Le Potier.  In that case it is customary to write the slope as $\mu$ and put $\delta(\mu) = \DLP_H(\mu)$.  Non-exceptional $\mu$-stable sheaves of character $(r,\mu,\Delta)$ exist if and only if $\Delta \geq \delta(\mu)$.  In Figure \ref{fig-P2}, the function $\delta(\mu)$ is the top fractal-like curve, bounding the region labeled ``S'' (for \emph{stable}).
\end{example}

\subsection{Stability of exceptional bundles and stability intervals} On the other hand,  computing $\DLP_{H}$ for $H$ different from the anticanonical polarization first requires that we determine the collection of $\mu_H$-stable exceptional bundles.  This can be done by induction on the rank.  First, we define a restricted Dr\'ezet-Le Potier function $\DLP_H^{<r}$ which only takes the $\mu_H$-stable exceptional bundles of rank less than $r$ into account. 
\begin{theorem}[Corollary \ref{cor-DLPExceptional}]\label{thm-exceptionalIntro}
Let $\bv = (r,\nu,\Delta) \in K(\F_e)$ be a character such that $\chi(\bv,\bv) = 1$, and let $H$ be an arbitrary polarization.  There is a $\mu_H$-stable exceptional bundle of character $\bv$ if and only if $$\Delta \geq \DLP_H^{<r}(\nu).$$
\end{theorem}

Given a bundle $\cV$ on $\F_e$, we can seek to compute the open interval $$I_\cV = \{m > 0 : \cV \textrm{ is $\mu_{H_m}$-stable}\} \subset \RR_{>0},$$ its \emph{stability interval}.  When $\cV$ is exceptional and $\mu_H$-stable for some polarization $H$ (e.g. $-K_{\F_e}$ in the del Pezzo case), we give an inductive algorithm to compute $I_\cV$ precisely; this is essentially equivalent to Theorem \ref{thm-exceptionalIntro}.  See Example \ref{ex-stabilityIntervals}.

For a \emph{general} sheaf $\cV$ it turns out that $\mu_{-K_{\F_e}}$-stability is the easiest stability to satisfy in the del Pezzo case (see Corollary \ref{cor-KstabilityEasy}), and in the non-del Pezzo case a general sheaf which is slope stable for some polarization is slope stable for polarizations arbitrarily close to $H_0$.  Consideration of stability intervals leads to the following result.

\begin{theorem}[Corollaries \ref{cor-deltaMonotone} and \ref{cor-deltaMonotoneHigher}]
Fix $\nu \in \Pic(\F_e)\te \QQ$.  As $m$ moves away from $1-\frac{e}{2}$ the number is $\delta_m^{\mus}(\nu)$ increasing.
\end{theorem}

\begin{remark}
The same statement is true for the function $\DLP_{H_m}(\nu)$, which will be an important technical tool in the second half of the paper.  See Proposition \ref{prop-DLPmonotone}.
\end{remark}

\subsection{Structure of the generic Harder-Narasimhan filtration}

When $\F_e$ is a del Pezzo surface, the fact that the function $\delta_{1-\frac{e}{2}}^{\mus}(\nu)$ is computed by exceptional bundles is closely related to the shape of the general Harder-Narasimhan filtration.  The following theorem generalizes an analysis of the Dr\'ezet and Le Potier classification of semistable sheaves on $\P^2$.  Recall that a {\em semiexceptional} bundle is a direct sum of copies of an exceptional bundle.

\begin{theorem}[Corollary \ref{cor-HNShape}]\label{thm-DPHNintro}
Let $e =0$ or $1$ and let $\bv=(r,\nu,\Delta)$ be a character with $\Delta \geq \frac{3}{8}$.  Let $H$ be a polarization sufficiently close to $-K_{\F_e}$ (depending on $\bv$). If there are no $H$-semistable sheaves of character $\bv$, then at most one of the $H$-Harder-Narasimhan factors of the general sheaf $\cV\in \cP_F(\bv)$ is not a semiexceptional bundle.
\end{theorem}

The obvious analogues of Theorems \ref{thm-HNintro} and \ref{thm-DPHNintro} also hold in the case of $\P^2$.  We make this story more explicit in this case, which motivated our study on $\F_e$.

\begin{example}\label{ex-HNP2}
On $\P^2$, let $H$ be the class of a line and consider the stack $\cP_H(\bv)$ of $H$-prioritary sheaves.  If $\bv=(r,\mu,\Delta)\in K(\P^2)$ is any character, then there are four possible shapes for the generic Harder-Narasimhan filtration.  See Figure \ref{fig-P2}.

\begin{figure}[t] 
\begin{center}
\includegraphics[scale=.6,bb=0 0 9.72in 6.81in]{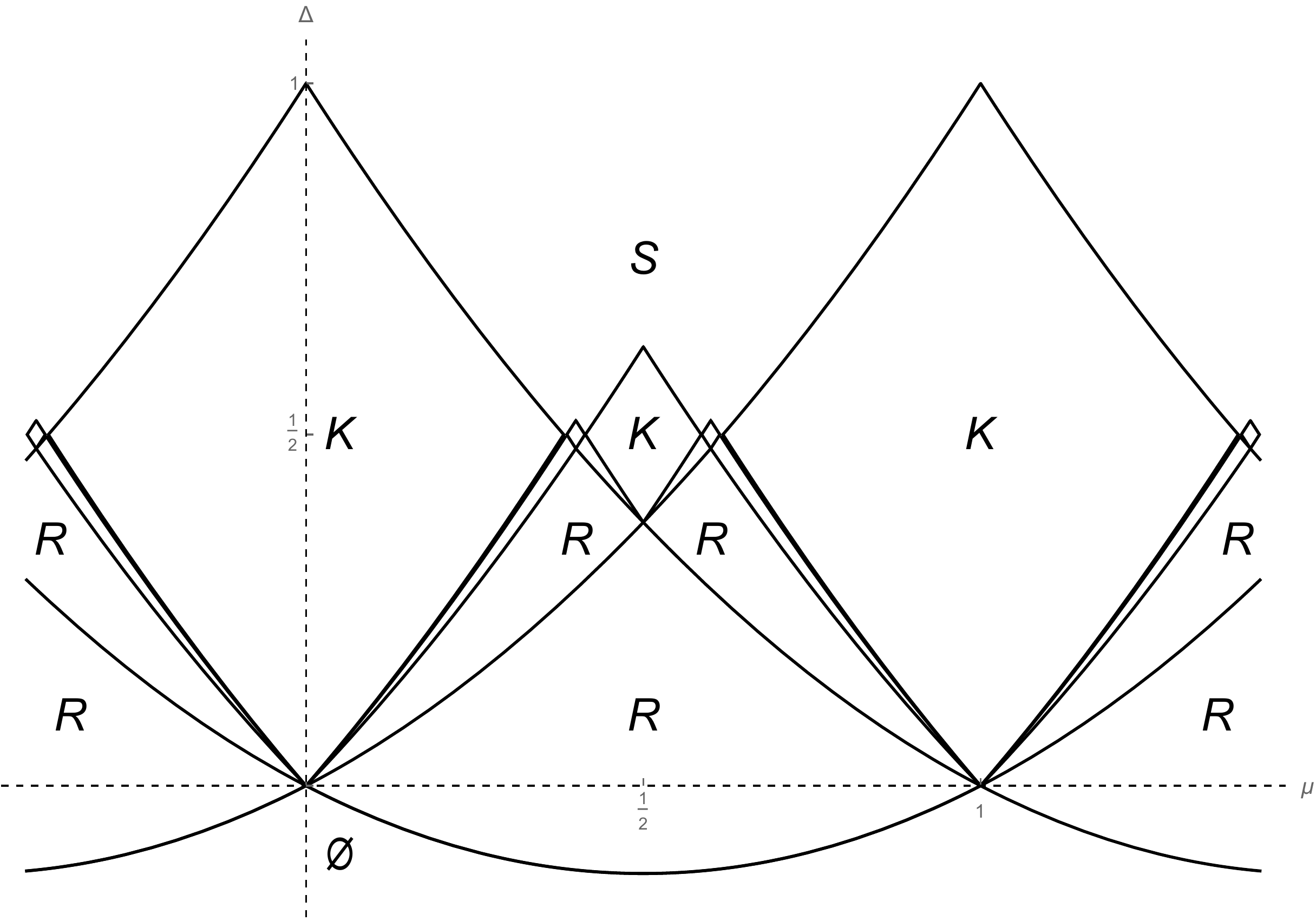}
\end{center}
\caption{On $\P^2$, the Dr\'ezet-Le Potier function $\delta(\mu)$ is the top curve bounding the region labeled ``S'' for stable.  Other regions in the $(\mu,\Delta)$-plane where the shape of the Harder-Narasimhan filtration of the general prioritary sheaf remains constant are also labeled.  See Example \ref{ex-HNP2}.}\label{fig-P2}
\end{figure}
\begin{enumerate}
\item In region ``S'' which lies on and above the Dr\'ezet-Le Potier curve $\Delta = \delta(\mu)$, there are semistable sheaves.

\item In the quadrilateral regions marked ``K,'' the Harder-Narasimhan filtration has length $2$.  One factor $\cE_3^{\oplus c}$ is a semiexceptional bundle with invariants at the bottom vertex of the quadrilateral.  The other factor is a \emph{Kronecker module} $\cK$, which has invariants lying on the portion of the curve $\Delta = \delta(\mu)$ lying over $\cE$.  For example, suppose $\mu(\cK) < \mu(\cE_3)$.  Then $\cK$ has a resolution $$0\to \cE_1^{\oplus a} \to \cE_2^{\oplus b} \to \cK\to 0$$ where $(\cE_1,\cE_2,\cE_3)$ is an exceptional collection.  This gives the orthogonality $\chi(\cE_3,\cK)=0$ needed for $\ch(\cE_3^{\oplus c}),\ch(\cK)$ to be the characters of the Harder-Narasimhan factors of the general sheaf.  If $\mu(\cK) > \mu(\cE_3)$ the discussion is dual.
\item In the triangular regions marked ``R,'' the general sheaf is \emph{rigid} and can be written as a direct sum of semiexceptional bundles with invariants at the vertices of the region.
\item Below the bottom curve, the stack $\cP_H(\bv)$ is empty.
\end{enumerate}
In every case where the generic Harder-Narasimhan filtration is nontrivial, we observe that the generic Harder-Narasimhan filtration arises from a full exceptional collection $\cE_1,\cE_2,\cE_3$ by grouping some of the adjacent terms together and constructing semistable bundles in the subcategories generated by the groups.  Case (2) corresponds to the groupings $\cE_1\cE_2|\cE_3$ and $\cE_1|\cE_2\cE_3$, while Case (3) corresponds to the grouping $\cE_1|\cE_2|\cE_3$ into three parts.  
\end{example}

\begin{example}\label{ex-HNDP}
Consider $\F_e$ with $e=0$ or $1$, and let $H$ and $\bv$ be as in Theorem \ref{thm-DPHNintro}.  Let $\cV_1,\ldots,\cV_\ell$ be the \emph{semiexceptional} factors appearing in the $H$-Harder-Narasimhan filtration of a general sheaf $\cV\in \cP_F(\bv)$, ordered as in the Harder-Narasimhan filtration.  So, there is at most one additional factor $\cW$, say inserted between $\cV_k$ and $\cV_{k+1}$.  

Let $\cE_\ell,\ldots,\cE_1$ be the exceptional bundles appearing in $\cV_1,\ldots,\cV_\ell$, indexed in the reverse order.  Then $\cE_1,\ldots,\cE_\ell$ is an exceptional collection.  Indeed, for $i>j$ we have $\Hom(\cE_i,\cE_j) = 0$ by stability, $\Ext^2(\cE_i,\cE_j) = 0$ since the slopes of the exceptional bundles have to be close (Theorem \ref{thm-HNintro} (3)), and therefore $\Ext^1(\cE_i,\cE_j) = 0$ by the orthogonality $\chi(\cE_i,\cE_j)=0$ (Theorem \ref{thm-HNintro} (4)).

Any exceptional collection on $\F_0$ or $\F_1$ can be completed to a full exceptional collection (of length $4$); see \cite{RudakovQuadric} and \cite{KuleshovOrlov}.  Furthermore, by a sequence of mutations we can ensure that the new exceptional sheaves which arise lie between the exceptionals corresponding to $\cV_k$ and $\cV_{k+1}$.  Since the remaining factor $\cW$ satisfies the orthogonalities $\chi(\cV_i,\cW) = 0$ for $i\leq k$ and $\chi(\cW,\cV_i) = 0$ for $i>k$, its Chern character must be a linear combination of the new exceptional sheaves.

We conclude that it is always possible to find a full exceptional collection $\cE_1,\cE_2,\cE_3,\cE_4$ and group the terms into adjacent blocks (at most one of size larger than $1$) so that the characters in the $H$-Harder-Narasimhan filtration correspond to the blocks in the reverse order, with each character being a linear combination of the exceptionals in the corresponding block.  The possible groupings are $$ \cE_1|\cE_2\cE_3\cE_4 \qquad \cE_1\cE_2\cE_3|\cE_4 \qquad \cE_1|\cE_2|\cE_3\cE_4 \qquad \cE_1 | \cE_2\cE_3 |\cE_4 \qquad \cE_1\cE_2 | \cE_3 | \cE_4 \qquad \cE_1|\cE_2|\cE_3|\cE_4.$$
\end{example}

\subsection{Harder-Narasimhan filtrations from orthogonal Kronecker modules}\label{ssec-introKronecker}
 Going in the other direction from Example \ref{ex-HNDP}, exceptional collections give a natural starting point to construct interesting generic Harder-Narasimhan filtrations. If we group the terms of a full exceptional collection $\cE_1,\cE_2,\cE_3,\cE_4$ into adjacent blocks and construct characters $\bv_1,\ldots,\bv_\ell$ as linear combinations of exceptionals in the blocks in reverse order, then the orthogonalities $\chi(\bv_i,\bv_j) = 0$ for $i<j$ are automatic.  In Example \ref{ex-HNDP} the grouping $$\cE_1\cE_2|\cE_3\cE_4$$ did not appear since it was forbidden by Theorem \ref{thm-DPHNintro}.  However, for polarizations other than the anticanonical, this grouping can be a source of length $2$ generic Harder-Narasimhan filtrations where no factor is semiexceptional. In \S\ref{sec-HNKronecker} we carry out this strategy to construct Chern characters ${\bf v}$ and polarizations $H_m$ such that the general $H_m$-Harder-Narasimhan filtration has length 2 and neither factor is semiexceptional. This will give examples of characters $\bv = (r,\nu,\Delta)$ such that the inequality $$\delta_m^{\mus}(\nu) > \DLP_{H_m}^{<r}(\nu)$$ is strict.  In these cases $\delta_m^{\mus}(\nu)$ is not computed by exceptional bundles, but we are able to exactly compute $\delta_m^{\mus}(\nu)$ (see Theorem \ref{thm-deltaKronecker}).  By contrast with the function $\DLP_{H_m}^{<r}(\nu)$, which varies in a locally constant fashion as $m$ changes, we find that $\delta_m^{\mus}(\nu)$ can increase continuously with $m$!

To construct these examples, we let $\ell\geq 3$ and start with the full exceptional collection $$\cE_1= \OO_{\F_e}(-E-\ell F),\quad\cE_2=\OO_{\F_e},\quad \cE_3=\OO_{\F_e}(F),\quad \cE_4=\OO_{\F_e}(E-(\ell-1-e)F).$$ 
We consider extensions of the form 
$$0 \rightarrow \cE_3^a \rightarrow \cK \rightarrow \cE_4^b \rightarrow 0$$ and cokernels of the form
$$0  \rightarrow \cE_1^c \rightarrow \cE_2^d \rightarrow \cL \rightarrow 0,$$ both of which are determined by \emph{Kronecker modules}.  We then have to determine the polarizations $H_m$ such that $\mu_{H_m}(\cK)$ is slightly larger than $\mu_{H_m}(\cL)$ and $\cK$ and $\cL$ are $\mu_{H_m}$-stable (the bulk of the methods in the paper come into play here).  Then the characters $\bk = \ch \cK$ and $\bl = \ch \cL$ will give the $H_m$-Harder-Narasimhan factors for the general sheaf in $\cP_F(\bk + \bl)$.  

For example, on $\FF_1$, if $\epsilon>0$ is small then the $H_{\frac{12}{7}+\epsilon}$-Harder-Narasimhan filtration of a general $F$-prioritary sheaf of character $$(r, \nu, \Delta) = \left(13, \frac{3}{13}E + \frac{6}{13} F,  \frac{98}{169} \right)$$  has two factors of characters $$\left(2, \frac{1}{2}E, \frac{5}{8} \right) \quad \mbox{and} \quad \left(11, \frac{2}{11} E + \frac{6}{11}F, \frac{65}{121} \right),$$ neither of which are exceptional.
Similarly, on $\FF_0$, the $H_{\frac{25}{9}+\epsilon}$-Harder-Narasimhan filtration of a general $F$-prioritary sheaf of character$$ \left(15,\frac{1}{5}{E}+\frac{1}{3}F,\frac{3}{5}\right)$$ has two factors given by $$\left(2,\frac{1}{2}E-\frac{1}{2}F,\frac{3}{4}\right) \quad \mbox{and} \quad 
\left(13,\frac{2}{13}E+\frac{6}{13}F,\frac{90}{169}\right).$$ 
See Examples \ref{ex-KroneckerF0} and \ref{ex-KroneckerF1}.

We conjecture that pairs of orthogonal Kronecker modules are the only additional source of generic Harder-Narasimhan filtrations.  We expect that an affirmative answer to the conjecture will allow an exact inductive computation of the function $\delta_m^{\mus}(\nu)$.

\begin{conjecture}
Let $\F_e$ be a Hirzebruch surface and let $H_m$ be an arbitrary polarization. Let $\bv\in K(\F_e)$ be a character such that there are $H_{\lceil m\rceil}$-prioritary sheaves of character $\bv$, and let $\bv_1,\ldots,\bv_\ell$ be the characters of the factors in the $H_m$-Harder-Narasimhan filtration of a general sheaf $\cV\in \cP_F(\bv)$.  Suppose that more than one of the $\bv_i$ is not semiexceptional.  Then $\ell=2$, and there is a full exceptional collection $\cE_1,\cE_2,\cE_3,\cE_4$ such that $\bv_1$ is a linear combination of $\ch \cE_3$ and $\ch \cE_4$ and $\bv_2$ is a linear combination of $\ch \cE_1$ and $\ch \cE_2$.  
\end{conjecture}

\subsection{Reduction to the del Pezzo case}  Much of the paper is written in the del Pezzo case $e=0$ or $1$, since the anticanonical polarization is often useful.  However, in the final section \S\ref{sec-reduction} we show that most of our results in these cases can be easily transported to the surfaces $\F_e$ with $e\geq 2$ by means of a simple linear map.

\subsection{The birational geometry of $M_{\FF_e, H_m}({\bf v})$} The results of this paper have immediate applications to the birational geometry of moduli spaces of sheaves on $\FF_e$. In \cite{CoskunHuizengaAmple} and \cite{CoskunHuizengaNef}, we showed how to compute the ample cone of the moduli spaces of sheaves on a surface using Bridgeland stability conditions provided $\Delta \gg 0$. A  $\mu_H$-stable sheaf is $(H,D)$-twisted Gieseker semistable for every twisting divisor $D$. Consequently, our results also classify when the Matsuki-Wentworth moduli spaces of $(H_m,D)$-twisted Gieseker semistable sheaves are nonempty for generic polarizations $H_m$ on $\FF_e$. By the main theorem of \cite{CoskunHuizengaNef}, if $\Delta \gg 0$, the problem of determining the Gieseker wall in the $(H,D)$-slice of the Bridgeland stability manifold is equivalent to knowing the classification of Chern characters of $(H, D)$-twisted Gieseker semistable sheaves. Hence, our results determine the Gieseker wall in the $(H_m,D)$-slice of the stability manifold of $\FF_e$. Furthermore, when the moduli space has no strictly semistable objects, by \cite{CoskunHuizengaNef}, the Bayer-Macr\`{i} divisor provides a nef divisor on the boundary of the nef cone. Hence, the calculations in \cite[\S 7.3]{CoskunHuizengaNef} extend from $\PP^1 \times \PP^1$ to all $\FF_e$ and from rank 2 to arbitrary rank.

The classification of stable Chern characters is also one of the main ingredients of solving the interpolation problem and constructing theta and Brill-Noether divisors on moduli spaces of sheaves (see \cite{ABCH, CoskunHuizengaGokova, CHW}). Our results make it possible to extend results of \cite{BertramCoskun} to higher rank and all Hirzebruch surfaces. Even in the case of $\PP^1 \times \PP^1$,  our classification of stable Chern characters allows constructions of effective divisors using the algorithm of \cite{Ryan} and eliminates assumptions in \cite{Ryan} regarding existence of stable Chern characters.

\subsection*{Organization of the paper} In \S \ref{sec-Prelim}, we recall the preliminary facts needed in the rest of the paper. In \S \ref{sec-Prioritary}, we study basic properties of prioritary sheaves on $\FF_e$ and show that $H_m$-semistable sheaves are $H_{\lceil m \rceil +1}$-prioritary. In \S \ref{sec-existPrioritary}, we classify Chern characters of $F$ and $H_k$-prioritary sheaves for every integer $k$.  In \S \ref{sec-genHN}, we study properties of generic Harder-Narasimhan filtrations and obtain strong restrictions on the Chern characters of the graded pieces.   

In sections \S 6-10, we primarily concentrate on the surfaces $\FF_0$ and $\FF_1$. In these cases, the anticanonical bundle is ample and the stability of sheaves with respect to the anticanonical polarization has been studied in great detail. In \S \ref{sec-exceptional}, we recall and reinterpret results of Gorodentsev, Kuleshov, Orlov and Rudakov concerning stability of sheaves with respect to $-K_{\FF_e}$. In \S \ref{sec-sufficient}, we show that the Dr\'{e}zet-Le Potier surface determines the stability of sheaves on $\FF_0$ or $\FF_1$ with respect to the anticanonical polarization. In \S \ref{sec-exceptional2}, we study exceptional bundles on $\FF_0$ and $\FF_1$ with respect to an arbitrary polarization and show that the Dr\'{e}zet-Le Potier surface determines them. We also show that the generic stability interval of an $\mu_{H_m}$-stable sheaf contains $1- \frac{e}{2}$ corresponding to the anticanonical polarization. In \S \ref{sec-sharpBogomolov}, we define a sharp Bogomolov function and determine its basic properties. In \S \ref{sec-HNKronecker}, we construct examples of spaces of prioritary sheaves and polarizations $H_m$ such that the generic $H_m$-Harder-Narasimhan filtration has two factors neither of which are semiexceptional. This shows that exceptional bundles do not control the entire story away from the anticanonical polarization. 

Finally, in \S \ref{sec-reduction}, we generalize many of the results from sections \S \ref{sec-exceptional}-\ref{sec-HNKronecker} to arbitrary Hirzebruch surfaces.

\subsection*{Acknowledgments} We would like to thank Daniel Levine and Dmitrii Pedchenko for valuable conversations.

\section{Preliminaries}\label{sec-Prelim}

In this section, we recall basic facts concerning moduli spaces of sheaves, prioritary sheaves and Hirzebruch surfaces that will be used in the rest of the paper. 

\subsection{Hirzebruch surfaces}\label{subsec-Hirzebruch} We refer the reader to \cite{Beauville, CoskunScroll} or \cite{Hartshorne} for more detailed discussions on Hirzebruch surfaces. For any integer $e \geq 0$, let $\FF_e$ denote the Hirzebruch surface $\PP( \OO_{\PP^1} \oplus \OO_{\PP^1}(e))$. The surface $\FF_e$ naturally fibers over $\PP^1$. Let $F$ denote the class of a fiber and let $E$ denote the class of a section with self-intersection $-e$. When $e=0$, $\FF_e \cong \PP^1 \times \PP^1$. When $e \geq 1$, then the section with negative self-intersection is unique. The Picard group of $\FF_e$ and the intersection pairing is given by 
$$\Pic(\FF_e) = \ZZ E \oplus \ZZ F, \quad E^2 = -e, \quad F^2 =0, \quad E \cdot F =1.$$
The canonical class of $\FF_e$ is $$K_{\FF_e} = -2E - (e+2) F.$$ The nef cone of $\FF_e$ is spanned $E+ eF$ and $F$ and the effective cone of $\FF_e$ is spanned by $E$ and $F$. In particular, the anti-canonical class $-K_{\FF_e}$ is effective, but when $e \geq 2$, it is not ample.  For  $m\in \QQ$, we consider the divisor class $H_m = E + (m+e)F$.  Then $H_0$ is nef and $H_m$ is ample for $m > 0$.  As $m$ tends to infinity, the divisor $H_m$ tends to the ray spanned by $F$, giving the other edge of the nef cone.  Every ample divisor on $\F_e$ is an integer multiple of some $H_m$ with $m>0$.

\subsection{Chern charcters and Riemann-Roch}  Given a  torsion-free sheaf $\cV$ on a surface $X$ and an ample divisor $H$, the {\em total slope $\nu$}, the {\em $H$-slope $\mu_H$} and the {\em discriminant $\Delta$} are defined as follows
$$\nu(\cV) = \frac{\ch_1(\cV)}{\ch_0(\cV)}, \quad \mu_H (\cV) = \frac{\ch_1(\cV) \cdot H}{\ch_0(\cV)}, \quad \Delta(\cV) = \frac{1}{2} \nu^2 - \frac{\ch_2(\cV)}{\ch_0(\cV)}.$$ These quantities depend only on the Chern character of $\cV$ and not on the particular sheaf. Given a Chern character ${\bf v}$, we define its total slope, $H$-slope and discriminant by the same formulae.  We will often record Chern characters by the rank, total slope and the discriminant. Note that one can  recover the Chern classes from this data. 

In terms of $\nu$ and $\Delta$, the Riemann-Roch Theorem says $$\chi(\cV) = \ch_0(\cV) (P(\nu(\cV)) - \Delta(\cV)),$$ where $$P(\nu) = \chi(\OO_X) + \frac{1}{2} (\nu^2 - \nu \cdot K_X)$$ is the Hilbert polynomial of $\OO_X$. Given two sheaves $\cV, \cW$ with invariants $(r(\cV), \nu(\cV), \Delta(\cV))$ and $(r(\cW), \nu(\cW), \Delta(\cW))$, let $\ext^i(\cV, \cW)$ denote the dimension of $\Ext^i(\cV, \cW)$. The Riemann-Roch Theorem  says that 
$$\chi(\cV, \cW) = \sum_{i=0}^2 (-1)^i \ext^i(\cV, \cW) = r(\cV) r(\cW) (P(\nu(\cW) - \nu(\cV)) - \Delta(\cV) - \Delta(\cW)).$$ In the case of the Hirzebruch surface $\FF_e$, for a sheaf with invariants ${\bf v} = (r, \nu, \Delta) = (r, aE + bF, \Delta)$ we have  $$P({\bf v}) = \left(a+1\right) \left(b+1 - \frac{ae}{2}\right) \quad  \mbox{and} \quad \chi({\bf v}) = r \left(\left(a+1\right) \left(b+1 - \frac{ae}{2}\right) - \Delta\right).$$ 

\subsection{Stability}\label{subsec-semistable} We refer the reader to \cite{CoskunHuizengaGokova, HuybrechtsLehn} and \cite{LePotier} for more detailed discussions.  A torsion-free, coherent sheaf $\cV$ is called $\mu_H$-(semi)stable (or \emph{slope} (semi)stable) if  every proper subsheaf $0 \not= \cW \subsetneq \cV$ of smaller rank satisfies $$\mu_H (\cW) \leqor \mu_H (\cV)$$Define the {\em Hilbert polynomial}  $P_{\cV}(m)$ and {\em reduced Hilbert polynomial}  $p_{\cV}(m)$ of a torsion-free sheaf $\cV$ on a surface $X$ with respect to an ample $H$ by 
$$P_{\cV}(m) = \chi(\cV(mH)), \quad p_{\cV}(m) = \frac{P_{\cV}(m)}{\rk(\cV)}.$$ A torsion-free sheaf $\cV$ is $H$-(semi)stable (or \emph{Gieseker} (semi)stable) if  for every proper subsheaf $\cW\subset \cV$, we have $$p_{\cW}(m) \leqor p_{\cV} (m) \ \mbox{for} \ m \gg 0.$$ Slope stability implies Gieseker stability and Gieseker semistability implies slope semistability. The Bogomolov inequality asserts that if $\cV$ is $\mu_H$-semistable, then $\Delta(\cV) \geq 0$. 

Every torsion-free sheaf $\cV$ admits a {\em Harder-Narasimhan} filtration with respect to both $\mu_H$- and $H$-semistability, that is there is a finite filtration 
$$0 \subset \cV_1 \subset \cV_2 \subset \cdots \subset \cV_n = \cV$$ such that the quotients $$\cW_i = \cV_i / \cV_{i-1}$$ are $\mu_H$ (respectively, $H$-Gieseker) semistable and $$\mu_H (\cW_i) > \mu_H (\cW_{i+1}) \quad \mbox{(respectively,} \  p_{\cW_i} (m) > p_{\cW_{i+1}} (m) \  \mbox{for}\  m \gg 0)$$ for $1 \leq i \leq n-1$.  The Harder-Narasimhan filtration is unique. Being $\mu_H$-(semi)stable or $H$-(semi)stable are open in flat families (see \cite[\S 2.3]{HuybrechtsLehn}). Furthermore, given a flat family of sheaves, there is a generic Harder-Narasimhan filtration (\cite[Theorem 2.3.2]{HuybrechtsLehn}). A semistable sheaf further admits a {\em Jordan-H\"{o}lder} filtration into stable sheaves. Two semistable sheaves are called {\em $S$-equivalent} if they have the same associated graded objects with respect to the Jordan-H\"{o}lder filtration.  

Gieseker \cite{Gieseker} and Maruyama \cite{Maruyama} constructed moduli spaces $M_{X, H}({\bf v})$ parameterizing $S$-equivalence classes of $H$-Gieseker semistable sheaves on $X$. In this paper, we will be concerned with the question of when $M_{\FF_e, H} ({\bf v})$ is nonempty.

 The notions of $\mu_H$-(semi)stability and $H$-(semi)stability depend on the polarization $H$. If we fix invariants of a sheaf on $X$, we obtain a locally finite wall-and-chamber decomposition of the ample cone of $X$, where within a chamber the set of sheaves that are $\mu_H$-stable remain constant (see \cite[\S 4.C]{HuybrechtsLehn} for more details). In particular, being $\mu_H$-stable is an open condition in the polarization $H$. Similarly, being $\mu_H$-semistable is a closed condition in the polarization. In contrast, Gieseker (semi)stability is not as well-behaved under change of polarization. The set of ample divisors $H$ for which a sheaf is $H$-Gieseker (semi)stable in general is neither open nor closed. However, if the sheaf is $\mu_H$-stable for some polarization, then the ample divisors $H$ for which it is not Gieseker (semi)stable differ only at the boundary of the corresponding chamber.

\subsection{Prioritary sheaves}\label{subsec-prioritary} It can be difficult to construct semistable sheaves. Prioritary sheaves provide an easier alternative on surfaces with negative canonical class.

\begin{definition}
Let $X$ be a smooth surface and $L$ a line bundle on $X$.  A torsion-free sheaf $\cV$ on $X$ is \emph{$L$-prioritary} if $\Ext^2(\cV,\cV(-L))=0$.  We write $\cP^X_L({\bf v})$ for the stack of prioritary sheaves of character ${\bf v}$.  We omit the surface $X$ from the notation if no confusion is possible.

We will also frequently consider sheaves that are simultaneously prioritary with respect to two different line bundles.  We write $\cP^X_{L_1,L_2}({\bf v}) = \cP_{L_1}^X({\bf v})\cap \cP_{L_2}^X({\bf v})$ for the stack of $L_1$- and $L_2$-prioritary sheaves of character ${\bf v}$.  The stack $\cP_{L_1,L_2}^X({\bf v}) \subset \cP_{L_1}^X({\bf v})$ is an open substack.
\end{definition}

If $H$ is a polarization on $X$ such that $H \cdot (K_X + L) < 0$, then every $\mu_H$-semistable sheaf is $L$-prioritary. If $\cV$ is $\mu_H$-semistable, we have 
$$\ext^2(\cV, \cV(-L))= \hom(\cV, \cV(K_X + L))=0$$ by semistability. In particular, when $K_X + L$ is anti-effective, then every semistable sheaf for any polarization $H$ is $L$-prioritary. 

Let $X$ be a birationally ruled surface with fiber class $F$ and let ${\bf v}$ be a positive rank Chern character. Then Walter's Theorem \cite{Walter} asserts that the stack $\cP_F({\bf v})$ is irreducible if nonempty. Furthermore, if  $r({\bf v}) \geq 2$, then the general member of $\cP_F({\bf v})$ is a vector bundle.

\subsection{Elementary modifications}\label{subsec-elementarymodification} Let $\cV$ be a torsion-free sheaf on a surface $X$. Let $p\in X$ be a general point  and  let $\phi: \cV \rightarrow \OO_p$ be a general surjection. Then the kernel $\cV'$ of $\phi$ 
$$ 0 \rightarrow \cV' \rightarrow \cV \stackrel{\phi}{\rightarrow} \OO_p \rightarrow 0$$
is a general {\em elementary modification} of $\cV$. The sheaves $\cV$ and $\cV'$ have the same rank $r$ and first Chern class and satisfy
$$\chi(\cV') = \chi(\cV) -1, \quad \Delta(\cV') = \Delta(\cV) + \frac{1}{r}.$$ By the long exact sequence of cohomology, $H^2(\cV) = H^2(\cV')$. If $h^0(\cV) > 0$, then  $$h^0(\cV')= h^0(\cV)-1.$$ If $h^0(\cV) = 0$, then $$h^1(\cV')= h^1(\cV)+1.$$
In particular, if $\cV$ has at most one nonzero cohomology group, so does $\cV'$. 

By \cite[Lemma 2.7]{CoskunHuizengaBN}, if $\cV$ is an $L$-prioritary sheaf, then $\cV'$ is also an $L$-prioritary sheaf. Consequently, in order to construct an $L$-prioritary sheaf of a given rank $r$, first Chern class $c_1$ and discriminant $\Delta$, it suffices to construct an $L$-prioritary sheaf with invariants $(r, c_1, \Delta')$ with $\Delta' \leq \Delta$.

\subsection{The cohomology of the general sheaf on a Hirzebruch surface} The cohomology of the general prioritary sheaf with $\Delta \geq 0$ was computed in \cite[Theorem 3.1]{CoskunHuizengaBN}. We include the statement for the reader's convenience.

\begin{theorem}[Betti numbers of a general sheaf]\label{thm-BN}
Let ${\bf v}\in K(\F_e)$ be a Chern character with positive rank $r=r({\bf v})$ and $\Delta({\bf v}) \geq 0$.  Then the stack $\cP_F({\bf v})$ is nonempty and irreducible. Let $\cV\in \cP_F({\bf v})$ be a general sheaf.
\begin{enumerate}
\item If we write $\nu({\bf v}) = \epsilon E + \varphi F$, then $$\chi({\bf v}) = r(P(\nu({\bf v}))-\Delta({\bf v})) = r \left(\left(\epsilon+1\right)\left(\varphi+1 - \frac{1}{2}e \epsilon \right)-\Delta({\bf v})\right)$$
\item If $\nu({\bf v})\cdot F \geq -1$, then $h^2(\F_e,\cV)=0$.
\item If $\nu({\bf v})\cdot F \leq -1$, then $h^0(\F_e,\cV)=0$.
\item In particular, if $\nu({\bf v})\cdot F = -1$, then $h^1(\F_e,\cV) = - \chi({\bf v})$ and all other cohomology vanishes.
\end{enumerate}
Now suppose $\nu({\bf v})\cdot F > -1$.  Then $H^2(\F_e,\cV)=0$, so either of the numbers $h^0(\F_e,\cV)$ or $h^1(\F_e,\cV)$ determine the Betti numbers of $\cV$.  These can be determined as follows.
\begin{enumerate}
\item[(5)] If $\nu({\bf v})\cdot E \geq -1$, then $\cV$ has at most one nonzero cohomology group.  Thus if $\chi({\bf v})\geq 0$, then $h^0(\F_e,\cV) = \chi({\bf v})$, and if $\chi({\bf v})\leq 0$, then $h^1(\F_e,\cV) = -\chi({\bf v})$.
\item[(6)] If $\nu({\bf v})\cdot E < -1$, then $H^0(\F_e,\cV) \cong H^0(\F_e,\cV(-E))$, and so the Betti numbers  of $\cV$ can be determined inductively using (3) and (5).
\end{enumerate}
(If $\nu({\bf v})\cdot F < -1$ and $r({\bf v})\geq 2$, then the cohomology of $\cV$ can be determined by Serre duality.)
\end{theorem}

\section{Prioritary sheaves and stability on Hirzebruch surfaces}\label{sec-Prioritary}

We begin by comparing $H_m$-semistability with the notion of a prioritary sheaf. The following lemma is useful for comparing notions of prioritary sheaves with respect to different line bundles.

\begin{lemma}\label{lem-priorCompare}
Let $X$ be a smooth surface and let $L$ and $M$ be line bundles on $X$.  If a torsion-free sheaf $\cV$ is $L$-prioritary and $L\te M^*$ is effective, then $\cV$ is $M$-prioritary.  
\end{lemma}
\begin{proof}
Use a section of $L\te M^*$ to consider an exact sequence $$0\to \OO_{X}(-L)\to \OO_X(-M)\to \OO_X(-M)|_D\to 0,$$ where $D$ is the zero scheme of the section. Applying $\cV\otimes -$, we have $\sTor^1(\cV,\OO_X(-M))=0$, so the sheaf $\sTor^1(\cV,\OO_X(-M)|_D)$ injects into the torsion-free sheaf $\cV(-L)$.  But $\sTor^1(\cV,\OO_X(-M)|_D)$ is torsion, supported on $D$, so it is zero and we have the exact sequence $$0\to \cV(-L)\to \cV(-M)\to \cV \otimes \OO_X(-M)|_D\to 0.$$ Applying $\Hom(\cV,-)$, we get an exact sequence $$\Ext^2(\cV,\cV(-L))\to \Ext^2(\cV,\cV(-M))\to \Ext^2(\cV,\cV \te \OO_X(-M)|_D)$$
We have $\Ext^2(\cV,\cV(-L))=0$ by assumption, and $$\Ext^2(\cV,\cV \te \OO_X(-M)|_D) \cong \Hom(\cV \te \OO_X(-M)|_D,\cV\te K_X)^* =0$$ by Serre duality since $\cV$ is torsion-free.  Therefore $\Ext^2(\cV,\cV(-M)) = 0$.
\end{proof}

In particular, on a Hirzebruch surface we can consider the notion of $H_m$-prioritary sheaves for integers $m$.  

\begin{propositionDefinition}\label{prop-index}
Let $\cV$ be a torsion-free sheaf on $\F_e$.  One of the following two possibilities holds.
\begin{enumerate}
\item There is an integer $\rho(\cV)$ such that $\cV$ is $H_m$-prioritary if and only if $m$ is an integer with $m\leq \rho(\cV)$.
\item The sheaf $\cV$ is $H_m$-prioritary for all $m\in \Z$.  In this case we declare $\rho({\cV})= \infty$.
\end{enumerate} 
The invariant $\rho({\cV})$ is called the \emph{prioritary index} of $\cV$.
\end{propositionDefinition}
\begin{proof}
By Lemma \ref{lem-priorCompare}, if $\cV$ is $H_m$-prioritary, then it is $H_{m-1}$-prioritary.  We must show that $\cV$ is $H_m$-prioritary for some $m$.  Consider the group $$\Ext^2(\cV,\cV(-H_m)) \cong \Hom(\cV(-H_m),\cV\te K_{\F_e})^* \cong \Hom(\cV,\cV(-E+(m-2)F)).$$ Pick any ample divisor $H$.  Then $\cV$ has an $H$-Harder-Narasimhan filtration.  If we pick $m \ll 0$, then we have that $$\mu_{\min,H}(\cV) > \mu_{\max,H}(\cV(-E+(m-2)F)).$$ This implies $\Ext^2(\cV,\cV(-H_m))=0$, so $\cV$ is $H_m$-prioritary for $m \ll 0$.
\end{proof}

\begin{example}
Since $H^2(\OO_{\F_e}(-H_m)) = 0$ for every integer $m$,  line bundles $L$ are $H_m$-prioritary for all $m$, and therefore $\rho(L) = \infty$.
\end{example}

\begin{example}\label{ex-prioritarySum}
Let $A>0$, $B>0$, $C\geq 0$ be integers, let $m,n\in \Z$, and consider the vector bundle $$\cV = \OO_{\F_e}(-E+(n-1)F)^A \oplus \OO_{\F_e}^B \oplus \OO_{\F_e}(-F)^C.$$ Then \begin{align*} \Ext^2(\cV,\cV(-H_m)) = 0 &\Leftrightarrow H^2(\OO_{\F_e}(-2E-(e+m-n+1)F))=0\\
&\Leftrightarrow H^0(\OO_{\F_e}((m-n-1)F))^*=0\\
&\Leftrightarrow m-n-1 \leq -1\\
&\Leftrightarrow m \leq n.\end{align*} Therefore $\rho(\cV) = n$.
\end{example}

\begin{example}
Taking $n = -e$ in Example \ref{ex-prioritarySum} gives a direct sum of line bundles which is $F$-prioritary and $E$-prioritary \cite[Lemma 3.3]{CoskunHuizengaBN}.  By choosing the exponents $a,b,c$ appropriately and by taking twists and/or duals of $\cV$, we can construct an $F$-prioritary and $E$-prioritary sheaf of any slope.  Furthermore, a calculation shows $\Delta(\cV) \leq 0$ \cite[Lemma 3.3]{CoskunHuizengaBN}, so by taking elementary modifications of $\cV$ (see \S \ref{subsec-elementarymodification}) we obtain the following result \cite[Corollary 3.6]{CoskunHuizengaBN}.
\end{example}

\begin{proposition}\label{prop-EFprioritary}
Let ${\bf v}\in K(\F_e)$ be a Chern character of positive rank satisfying the Bogomolov inequality $\Delta \geq 0$.  Then the stack $\cP_{F,E}({\bf v})$ is nonempty.
\end{proposition}

Recall that Walter's theorem \cite{Walter} says that the stacks $\cP_{F}({\bf v})$ of $F$-prioritary sheaves are irreducible if they are nonempty.  Therefore we have a chain of open substacks of the stack $\cP_F({\bf v})$:
$$\cdots \supset \cP_{F,H_{m-1}}({\bf v}) \supset \cP_{F,H_m({\bf v})} \supset \cP_{F,H_{m+1}}({\bf v})\supset\cdots$$  For fixed ${\bf v}$ it is an interesting question to determine the largest integer $m$ such that $\cP_{F,H_m}({\bf v})$ is nonempty.  Then the general $F$-prioritary sheaf is $H_m$-prioritary but not $H_{m+1}$-prioritary.

\begin{definition}
Let ${\bf v}\in K(\F_e)$ be a Chern character of positive rank such that the stack $\cP_F({\bf v})$ is nonempty.  The \emph{generic prioritary index} $\rho_{\gen}({\bf v})$ is the prioritary index $\rho(\cV)$ of a general $\cV\in \cP_F({\bf v})$.
\end{definition}

\begin{example} By Proposition \ref{prop-EFprioritary}, if $\Delta \geq 0$, then $\rho_{\gen}({\bf v}) \geq -e$.\end{example}

The next result shows the basic implication between semistability and prioritary sheaves.

\begin{proposition}\label{prop-ssPrior}
Let $m>0$ be a rational number, let $\epsilon\in \RR_{\geq 0}$, and let $\cV$ be a torsion-free sheaf with $$\mu_{\max,H_m}(\cV)-\mu_{\min,H_m}(\cV) \leq \epsilon.$$ If $n$ is an integer with $n < m+2-\epsilon$, then $\cV$ is $H_{n}$-prioritary.

In particular, if $\cV$ is $\mu_{H_m}$-semistable, then it is $H_{\lceil m\rceil +1}$-prioritary.
\end{proposition}
\begin{proof}
For an integer $n$, we compute $$\Ext^2(\cV,\cV(-H_{n})) \cong \Hom(\cV,\cV(K_{\F_e}+H_{n}))^*.$$ Then $\cV$ is $H_n$-prioritary if  $$\mu_{\min,H_m}(\cV) > \mu_{\max,H_m}(\cV)+(K_{\F_e}+H_n)\cdot H_m.$$ This inequality  holds if $$(K_{\F_e}+H_n)\cdot H_m<- \epsilon.$$ We compute \begin{align*}
(K_{\F_e}+H_n)\cdot H_m &= (-E+(n-2)F)\cdot (E+(m+e)F)\\
&= e-(m+e)+(n-2)\\
&= n-m-2,
\end{align*}
so if $n < m+2-\epsilon$, then $\cV$ is $H_n$-prioritary.
\end{proof}

\begin{remark}
Thus to study $H_m$-semistability, we will primarily be interested in sheaves that are prioritary with respect to a line bundle $H_m$ with $m\geq 1$.   An $H_1$-prioritary sheaf is automatically $F$-prioritary by Lemma \ref{lem-priorCompare}, so if $m\geq 1$, then the nonemptiness problems for the stacks $\cP_{F,H_m}({\bf v})$ and $\cP_{H_m}({\bf v})$ are equivalent.  On the other hand, the stacks $\cP_{H_m}({\bf v})$ must typically be badly behaved for $m \ll 0$, since their union is the stack of torsion-free sheaves of character ${\bf v}$.  The substack $\cP_{F,H_m}({\bf v})$ of the irreducible stack $\cP_F({\bf v})$ is a more reasonable object of study.
\end{remark}

\section{Existence of prioritary sheaves}\label{sec-existPrioritary}

Throughout this section we let ${\bf v}\in K(\F_e)$ be an integral Chern character of positive rank $r$, total slope $\nu  = \epsilon E + \varphi F$, and discriminant $\Delta\geq 0$ satisfying the Bogomolov inequality.  We fix an integer $m\in \Z$.  (Although we primarily care about the case where $m\geq 1$, we can handle all integers $m$ by the same argument.)  In this section we give a complete answer to the following question.
 
 \begin{problem}
 Is the stack $\cP_{F,H_m}({\bf v})$ of $F$- and $H_m$-prioritary sheaves of character ${\bf v}$ nonempty?
 \end{problem}
 
 We can phrase the answer to this question in two ways.  On the one hand, we explicitly compute the generic prioritary index $\rho_{\gen}({\bf v})$ of ${\bf v}$; then $\cP_{F,H_m}({\bf v})$ is nonempty if and only if $m\leq \rho_\gen({\bf v})$.
 On the other hand, we will give an explicit function $\delta_m^p({\nu})$ of $\nu$ such that $\cP_{F,H_m}({\bf v}) $ is nonempty if and only if $\Delta \geq \delta_m^p(\nu)$.

\subsection{Review of Gaeta resolutions}\label{ssec-Gaeta} Our key tool is results from \cite[\S4]{CoskunHuizengaBN} which show that the general sheaf in $\cP_F({\bf v})$ admits a particular Gaeta-type resolution.  Recall that for a line bundle $L$ an $L$-\emph{Gaeta resolution} of a sheaf $\cV$ of character ${\bf v}$ on $\F_e$ is a resolution of the form
$$0\to L(-E-(e+1)F)^\alpha \to L(-E-eF)^\beta \oplus L(-F)^\gamma \oplus L^\delta \to \cV \to 0$$ where $\alpha,\beta,\gamma,\delta$ are integers.  Here $L$ must be a line bundle such that the inequalities 
\begin{align}\label{lb-ineqs}\tag{$\ast$} \begin{split}\chi({\bf v} (-L))& \geq  0 \\ \chi({\bf v}  (-L-E)) &\leq  0 \\ \chi({\bf v}  (-L-F)) &\leq  0 \\ \chi({\bf v} (-L-E-F)) &\leq  0 \end{split}\end{align}
are satisfied, and $\alpha,\beta,\gamma,\delta$ must be the integers
\begin{align}\label{gaeta-exponents}\tag{$\ast\ast$} 
\begin{split}
\alpha &= -\chi({\bf v}(-L-E-F))\\
\beta &= -\chi({\bf v}(-L-E))\\
\gamma &= -\chi({\bf v}(-L-F))\\
\delta &= \chi({\bf v}(-L)).
\end{split}
\end{align}
Conversely, if we can find a line bundle $L$ such that the numerical inequalities (\ref{lb-ineqs}) are satisfied, then the stack $\cP_F({\bf v})$ is nonempty and the general $\cV\in \cP_F({\bf v})$ admits an $L$-Gaeta resolution.

To find an appropriate line bundle $L$, formally consider a variable line bundle $L_{a,b} = aE+bF$ with $a,b\in \RR$.  We consider the curve in the $(a,b)$-plane $\RR^2$ defined by the equation $\chi({\bf v}(-L_{a,b}))=0$, where the Euler characteristic is computed formally by Riemann-Roch.  Then $\chi({\bf v}(-L_{a,b}))=0$ gives $$\Delta = (\epsilon - a + 1)(\varphi - b+1-\frac{1}{2}e(\epsilon-a)).$$ If $\Delta >0$ (a similar discussion holds in the degenerate case $\Delta=0$), then this describes a hyperbola with asymptotes $$\ell_1 : a=\epsilon+1 \qquad \textrm{and} \qquad \ell_2 : b=\varphi + 1 - \frac{1}{2}e(\epsilon-a)$$ which are vertical and have slope $e/2$, respectively; the hyperbola has a branch $Q_1$ lying left of $\ell_1$ and below $\ell_2$, and a branch $Q_2$ lying right of $\ell_1$ and above $\ell_2$. See Example \ref{ex-hyperbola} and Figure \ref{fig-hyperbola} for an example of this hyperbola.

The function $\chi({\bf v}(-L_{a,b}))$ is positive below $Q_1$ and above $Q_2$, and negative on the region between the branches.  Thus a line bundle $L_{a,b}$ satisfies the inequalities (\ref{lb-ineqs}) if the lattice point $(a,b)$ lies below (or on) $Q_1$ and the points $(a+1,b)$, $(a,b+1)$, $(a+1,b+1)$ lie on or between $Q_1$ and $Q_2$.  There can be several possible line bundles with these properties, but there is one particular line bundle that typically works and is useful for our purposes.

\begin{definition}\label{def-L0}
Let $$ \psi := \varphi + \frac{1}{2}e(\lceil \epsilon \rceil-\epsilon)-\frac{\Delta}{1-(\lceil \epsilon\rceil - \epsilon)},$$ and define a line bundle $$L_0 := L_{\lceil \epsilon \rceil,\lceil \psi\rceil} = \lceil \epsilon\rceil E+\lceil \psi \rceil F.$$  Let $\alpha,\beta,\gamma,\delta$ be the integers defined by (\ref{gaeta-exponents}) when we take $L=L_0$.
\end{definition}

Next we analyze when the line bundle $L_0$ satisfies the inequalities (\ref{lb-ineqs}).  
  
\begin{proposition}\label{prop-Gaeta}
The line bundle $L_0$ has $\delta > 0$, $\beta\geq 0$, and $\gamma \geq 0$, so the inequalities (\ref{lb-ineqs}) hold if and only if $\alpha \geq 0$.
Thus, if $\alpha \geq 0$, then the general sheaf $\cV\in \cP_F({\bf v})$ admits a resolution of the form $$0\to L_0(-E-(e+1)F)^\alpha\to L_0(-E-eF)^\beta \oplus L_0(-F)^\gamma \oplus L_0^\delta \to \cV \to 0.$$
\end{proposition}
\begin{proof}
The integer $\lceil \epsilon\rceil$ is the largest integer which is strictly smaller than $\epsilon+1$.  Therefore any point $(\lceil \epsilon\rceil,b)$ lies left of the asymptote $\ell_1$, and if $b \ll 0$, then $(\lceil \epsilon \rceil,b)$ lies below $Q_1$.  A quick computation with Riemann-Roch shows that $$\chi({\bf v}(-L_{\lceil \epsilon\rceil, \psi +1})) = 0,$$ and therefore $\lceil \psi\rceil$ is the largest integer such that $\chi({\bf v}(-L_{\lceil \epsilon\rceil ,\lceil \psi \rceil}))> 0$.  Thus $\delta >0$.

Now $(\lceil \epsilon\rceil,\lceil \psi\rceil)$ lies strictly below $Q_1$. Considering the asymptotes $\ell_1$ and $\ell_2$, the point $(\lceil \epsilon \rceil +1 ,\lceil \psi \rceil)$ lies right of (or on) $\ell_1$ and below $\ell_2$, so it lies between $Q_1$ and $Q_2$.  Similarly, the point $(\lceil \epsilon\rceil,\lceil \psi\rceil +1)$ lies above (or on) $Q_1$ and left of $\ell_1$, so it lies between $Q_1$ and $Q_2$.  Therefore, if $\alpha \geq 0$, then all the inequalities (\ref{lb-ineqs}) are satisfied.
\end{proof}

\begin{remark}
The inequality $\alpha \geq 0$ nearly always holds.  Indeed, if $e\geq 2$, then the inequality $\alpha \geq 0$ automatically holds.  If $e = 0$ (resp. $e=1$), then it automatically holds if $\Delta \geq 1/4$ (resp. $\Delta \geq 1/8$).  See \cite[Lemma 4.5]{CoskunHuizengaBN} for details.
\end{remark}

Our problem is particularly easy to study in the relatively rare special case $\alpha <0$.

\begin{lemma}\label{lem-alpha<0}
If $\alpha < 0$, then the direct sum of line bundles $$\cW =  L(-E-(e+1)F)^{-\alpha} \oplus L(-E-eF)^\beta \oplus L(-F)^\gamma \oplus L^\delta$$ is rigid and $F$-prioritary of character ${\bf v}$.   Therefore, the general $\cV\in \cP_F({\bf v})$ is isomorphic to $\cW$.  Thus $\rho_{\gen}({\bf v}) = -e$ takes the smallest possible value.
\end{lemma}
\begin{proof}
The Chern class computation is elementary.  Since the line bundles in the direct sum are a strong exceptional collection, it is clear that $\cW$ is rigid.  It is also clearly $F$-prioritary.  Since $\delta >0$, we compute $\rho(\cW) = -e$ by a computation analogous to Example \ref{ex-prioritarySum}.
\end{proof}

\subsection{A necessary condition for existence} If $\alpha \geq 0$,  then by Proposition \ref{prop-Gaeta} we can let $\cV\in \cP_F({\bf v})$ be general and consider an $L_0$-Gaeta resolution of $\cV$:
$$0 \to L_0(-E-(e+1)F)^{\alpha} \to L_0(-E-eF)^{\beta} \oplus L_0(-F)^\gamma \oplus L_0^\delta \to \cV \to 0.$$
To study $\Ext^2(\cV,\cV(-H_m))$, we apply $\Ext(-,\cV(-H_m))$ and get an exact sequence
\begin{equation}\tag{$\clubsuit$}\label{extsequence}\Ext^2(\cV,\cV(-H_m)) \to \hspace{-.5em}\begin{array}{c}\Ext^2(L_0(-E-eF),\cV(-H_m))^\beta \\ \oplus \\ \Ext^2(L_0(-F),\cV(-H_m))^\gamma\\
\oplus\\
\Ext^2(L_0,\cV(-H_m))^\delta \end{array} \hspace{-.5em}\to \Ext^2(L_0(-E-(e+1)F),\cV(-H_m))^\alpha.\end{equation}
 From this sequence we deduce the following necessary inequality for the existence of an $F$- and $H_m$-prioritary sheaf.

\begin{theorem}\label{thm-prioritaryNecessary}

If there is an $F$- and $H_m$-prioritary sheaf of character ${\bf v}$, then $$\chi({\bf v}(-L_0-H_m)) \leq 0.$$
\end{theorem}
\begin{proof}
First notice by Lemma \ref{lem-alpha<0} that if $\alpha <0$, then $m\leq -e$  whenever there is an $H_m$-prioritary sheaf.  In this case we have $\chi({\bf v}(-L_0-H_{m})) \leq 0$ by the definition of $L_0$.

Assume $\alpha \geq 0$ for the rest of the proof, so we have an exact sequence (\ref{extsequence}).
First we compute $$\Ext^2(L_0(-E-(e+1)F),\cV(-H_m))=H^2(\cV(-L_0-(m-1)F)).$$ Since $$\nu(\cV(-L_0-(m-1)F))\cdot F=\epsilon - \lceil \epsilon\rceil >-1,$$ we conclude from Theorem \ref{thm-BN} that $H^2(\cV(-L_0-(m-1)F))=0$.

Therefore, if $\Ext^2(\cV,\cV(-H_m))=0$, then by sequence (\ref{extsequence}) we must have $\Ext^2(L_0,\cV(-H_m))=0$ (notice that $\delta > 0$ by the construction of $L_0$).  Now $$\Ext^2(L_0,\cV(-H_m)) = H^2(\cV(-L_0-H_m)),$$ and since $$\nu(\cV(-L_0-H_m))\cdot F = \epsilon - \lceil \epsilon\rceil -1 <-1$$ we have $H^0(\cV(-L_0-H_m))=0$ by Theorem \ref{thm-BN}.  Thus we must have $\chi(\cV(-L_0-H_m)) \leq 0$.
\end{proof}

\begin{remark}\label{rem-epInteger}
If $\epsilon$ is an integer, then by Riemann-Roch the inequality $\chi(\cV(-L_0-H_m)) \leq 0$ is always true by the Bogomolov inequality $\Delta \geq 0$.  So, there is no interesting restriction in this case.
\end{remark}

\begin{remark}\label{rmk-graphical} If $\epsilon$ is not an integer, the inequality $\chi(\cV(-L_0-H_m))\leq 0$ can be interpreted graphically in terms of the hyperbola $\chi({\bf v}(-L_{a,b}))=0$ in the $(a,b)$-plane.  Let $(a_0,b_0) = (\lceil \epsilon\rceil,\lceil \psi\rceil)$ be the lattice point which lies strictly below the branch $Q_1$ such that $(a_0+1,b_0)$ lies to the right of the vertical asymptote $\ell_1$ and $(a_0,b_0+1)$ lies on or above the branch $Q_1$.  Analogously, let $(a_1,b_1) = (a_0+1,b_1)$ be the lattice point which lies strictly above the branch $Q_2$ such that $(a_1-1,b_1)$ lies left of the vertical asymptote $\ell_1$ and $(a_1,b_1-1)$ lies on or below the branch $Q_2$.  Then the inequality $\chi(\cV(-L_0-H_m)) \leq 0$ means that the lattice point $(a_0+1,b_0+e+m)$ lies on or below the branch $Q_2$.  Equivalently, we must have 
$$b_0+e +m\leq b_1-1,$$ or $$m\leq b_1-b_0-e-1.$$
\end{remark}

\begin{example}\label{ex-hyperbola}
We illustrate the definitions and Remark \ref{rmk-graphical} in a particular case.  See Figure \ref{fig-hyperbola}.  We take $e=1$, $\nu = \frac{1}{2}E + \frac{1}{3}F$, and $\Delta = \frac{11}{10}$, and let $r$ be a rank such that the character $\bv = (r,\nu,\Delta)$ is integral ($r=120$ will do).  Then we compute $\psi = -\frac{97}{60}$ and $L_0 = \OO_{\F_1}(E-F)$. The points $(a_i,b_i)$ are $(a_0,b_0) = (1,-1)$ and $(a_1,b_1) = (2,5)$, respectively. Therefore $\chi(\cV(-L_0-H_m))\leq 0$ holds for integers $m \leq 4$.  Conversely, we will see that the general $\cV\in \cP_F(\bv)$ is $H_4$-prioritary in Corollary \ref{cor-prioritaryRho}.
\begin{figure}[t]
\begin{center}
\includegraphics[scale=.8,bb=0 0 6.25in 6.32in]{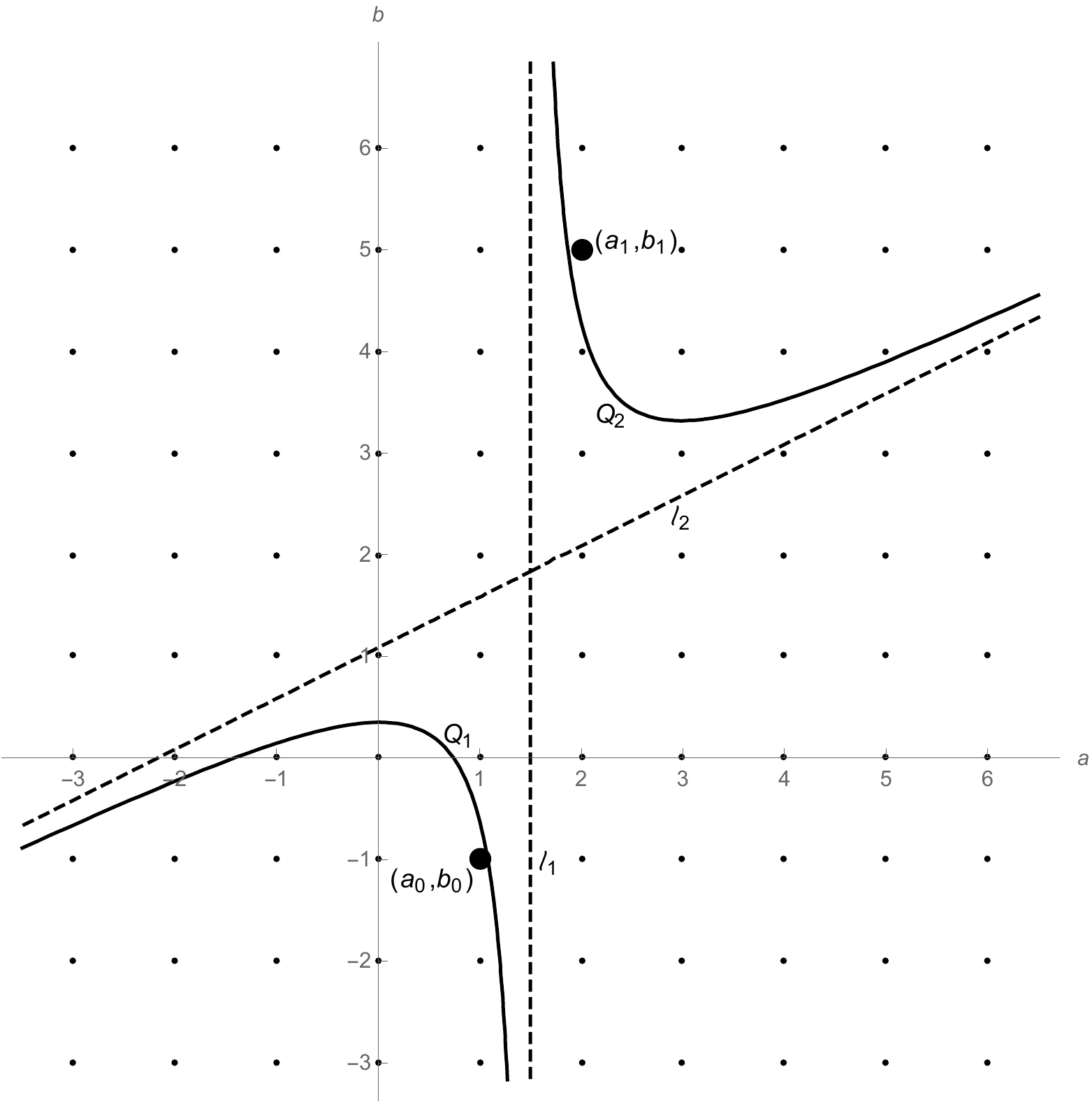}
\end{center}
\caption{For $e=1$, we show the hyperbola $\chi(\bv(-L_{a,b}))=0$ for a character $\bv$ with $\nu(\bv) = \frac{1}{2}E + \frac{1}{3}F$ and $\Delta(\bv)=\frac{11}{10}$.  Its two branches $Q_1$ and $Q_2$, asymptotes $\ell_1$ and $\ell_2$, and the points $(a_0,b_0) = (1,-1)$ and $(a_1,b_1) = (2,5)$ are shown.  See Example \ref{ex-hyperbola}.}\label{fig-hyperbola}
\end{figure}
\end{example}

\begin{remark}\label{rmk-deltamp}
Fix the integer $m$, and view ${\bf v}\in K(\F_e)$ as a Chern character with fixed rank $r > 0$ and slope $\nu = \epsilon E + \varphi F$, but variable discriminant $\Delta\geq 0$; then the corresponding line bundle $L_0 = L_0(\Delta)$ also varies with $\Delta$.  Suppose $\epsilon$ is not an integer.  As $\Delta$ increases, the branches of the hyperbola $\chi({\bf v}(-L_{a,b}))=0$ are pushed further away from the center $(\epsilon + 1, \varphi + \frac{1}{2}e+1)$, but the asymptotes remain fixed.  Then the lattice points $(a_0,b_0)= (a_0,b_0(\Delta))$ and $(a_1,b_1)=(a_1,b_1(\Delta))$  of Remark \ref{rmk-graphical} are pushed further apart from one another as $\Delta$ increases.  Thus the right hand side of the inequality $$m\leq b_1(\Delta)-b_0(\Delta)-e-1$$ is increasing in $\Delta$.  Furthermore, the functions $b_i(\Delta)$ are right-continuous in $\Delta$ (i.e. they remain constant when we increase $\Delta$ a little bit) and their discontinuities happen at rational $\Delta$.

\begin{definition}\label{def-deltaP}
If $\epsilon$ is not an integer, we let $\delta_m^p(\nu) \in \QQ_{\geq 0}$ be the smallest nonnegative number $\Delta$ such that $$m\leq b_1(\Delta) - b_0(\Delta)-e-1.$$  If $\epsilon$ is an integer, then because of Remark \ref{rem-epInteger} we define $\delta_m^p(\nu) = 0$.  Thus, if there is an $F$- and $H_m$- prioritary sheaf of character ${\bf v}=(r,\nu,\Delta)$ with $\Delta\geq 0$, then $\Delta \geq \delta_m^p(\nu)$.
\end{definition}

Suppose $\delta_m^p(\nu) > 0$.  If $\delta_m^p(\nu) < \delta_{m+1}^p(\nu)$, then $\delta_m^p(\nu)$ is characterized as the discriminant which makes $\chi(\cV(-L_0-H_m))=0$.  On the other hand  it is possible that $\delta_{m}^p(\nu) = \delta_{m+1}^p(\nu)$, since both $b_i(\Delta)$ can jump at the same value of $\Delta$.  In this case  $\delta_{m+1}^p(\nu) < \delta_{m+2}^p(\nu)$, so $\delta_{m}^p(\nu)$ is characterized as the discriminant which makes $\chi(\cV(-L_0-H_{m+1}))=0$; there is no discriminant which solves the equation $\chi(\cV(-L_0-H_m))=0$.
\end{remark}

\begin{example}
If $m \leq -e$, then since $b_1(\Delta) - b_0(\Delta) \geq 1$ for any $\Delta \geq 0$, the function $\delta_m^p(\nu)$ is identically $0$.  This corresponds to the fact that $\cP_{F,E}({\bf v})$ is nonempty by Proposition \ref{prop-EFprioritary}.
\end{example}

\begin{remark}\label{rem-twistDual}
Since twists of $H_m$-prioritary sheaves are $H_m$-prioritary and duals of locally free $H_m$-prioritary sheaves are $H_m$-prioritary, it is important to see that the necessary inequality of Theorem \ref{thm-prioritaryNecessary} is unchanged when we take a twist or dual.  
Correspondingly, the number $\delta_m^p(\nu)$ is unchanged under twisting or dualizing the slope $\nu$.

Let $M$ be a line bundle, and write $L_0({\bf v})$ for the line bundle $L_0$ of Definition \ref{def-L0} corresponding to the character ${\bf v}$.  Then $L_0({\bf v}\te M) = L_0({\bf v}) \te M$, so the inequality is unchanged.

Slightly less trivially, suppose the rank is at least $2$ (so a general sheaf in $\cP_F({\bf v})$ is locally free) and consider the Serre dual Chern character ${\bf v}^D$.  If $\epsilon$ is an integer, then the inequality always holds for both ${\bf v}$ and ${\bf v}^D$, so assume $\epsilon$ is not an integer.  Let $L_0({\bf v}) = a_0({\bf v})E+b_0({\bf v})F$ and $L_1({\bf v}) = a_1({\bf v})E+b_1({\bf v})F$ be the line bundles corresponding to the lattice points $(a_0({\bf v}),b_0({\bf v}))$ and $(a_1({\bf v}),b_1({\bf v}))$ defined for the character ${\bf v}$.  Then it is easy to check that $L_0({\bf v}^D) = L_1({\bf v})^*$ and $L_1({\bf v}^D) = L_0({\bf v})^*$.  Therefore $$b_1({\bf v}^D)-b_0({\bf v}^D)=-b_0({\bf v})+b_1({\bf v})$$ and the inequality is unchanged.
\end{remark}

We can also more explicitly rephrase the inequality $\chi(\cV(-L_0-H_m))\leq 0$ as an upper bound on $m$ in terms of $\nu$ and $\Delta$.

\begin{corollary}\label{cor-equivalentInequality}
If $\epsilon$ is not an integer, then the inequality $\chi(\cV(-L_0-H_m)) \leq 0$ holds if and only if
$$
m\leq \frac{\Delta}{(\lceil \epsilon \rceil-\epsilon)(\epsilon-\lfloor \epsilon\rfloor)}-\frac{e}{2}+1-(\lceil\psi\rceil - \psi).
$$
Thus if there is an $F$- and $H_m$-prioritary sheaf of character ${\bf v}$, the displayed inequality holds.
\end{corollary}

In particular, a (slightly weaker) necessary condition for the existence of an $H_m$-prioritary sheaf is the simpler inequality $$m \leq \frac{\Delta}{(\lceil \epsilon\rceil -\epsilon)(\epsilon-\lfloor \epsilon\rfloor)}-\frac{e}{2}+1$$ that depends only on $\epsilon$ and $\Delta$ (and not $\varphi$).

\begin{proof}
We use Riemann-Roch to compute
\begin{align*}\frac{\chi(\cV(-L_0-H_m))}{r}&=P(\nu(\cV(-L_0-H_m)))-\Delta \\&= (\epsilon-\lceil\epsilon\rceil)(\varphi-\lceil \psi\rceil-e-m+1-\frac{1}{2}e(\epsilon-\lceil\epsilon\rceil-1))-\Delta,\end{align*}  so the inequality $\chi(\cV(-L_0-H_m)) \leq 0$ is equivalent to 
\begin{align*}
m &\leq\frac{\Delta}{\lceil \epsilon\rceil -\epsilon}-\frac{\Delta}{\epsilon-\lfloor \epsilon\rfloor}-\frac{e}{2}+1+\left(\varphi +\frac{1}{2}e(\lceil \epsilon\rceil -\epsilon)+\frac{\Delta}{\epsilon-\lfloor \epsilon\rfloor}\right)-\lceil \psi \rceil\\
&= \frac{\Delta}{(\lceil \epsilon\rceil -\epsilon)(\epsilon-\lfloor \epsilon\rfloor)}-\frac{e}{2}+1-(\lceil \psi\rceil -\psi),
\end{align*}
as claimed.
\end{proof}

\subsection{Construction of prioritary sheaves}  Next we show that the inequality in Theorem \ref{thm-prioritaryNecessary} is sufficient for the existence of an $F$- and $H_m$-prioritary sheaf.   Note that if there is an $F$- and $H_m$-prioritary sheaf of invariants $(r,\nu,\Delta)$, then by considering elementary modifications there is also an $F$- and $H_m$-prioritary sheaf of invariants $(r,\nu,\Delta')$ for any $\Delta' > \Delta$ such that the character is integral.  Thus to complete the classification it suffices to construct $F$- and $H_m$-prioritary sheaves of rank $r$, slope $\nu$, and discriminant at most $\delta_m^p(\nu)$.

\begin{proposition}\label{prop-triangle}
Fix a rank $r$, total slope $\nu=\epsilon E + \varphi F$, and integer $m$.  Suppose  the point $(\epsilon,\varphi)\in \QQ^2$ is a convex combination $$(\epsilon,\varphi) = \lambda_1 (-1,m-1) + \lambda_2(0,0) + \lambda_3(0,-1).$$ Then there are uniquely determined integers $A,B,C$ such that the direct sum of line bundles $$\cV = \OO_{\F_e}(-E+(m-1)F)^A \oplus \OO_{\F_e}^B \oplus \OO_{\F_e}(-F)^C$$ has rank $r$ and total slope $\nu$. The bundle $\cV$ is $F$- and $H_m$-prioritary, and $$\delta_m^p(\nu) = \max\{\Delta(\cV),0\}.$$
\end{proposition}

The proposition can also be viewed as giving a more explicit description of the function $\delta_m^p(\nu)$ on the triangular region $T$ with vertices $(-1,m-1)$, $(0,0)$, $(0,-1)$.  By taking twists and duals, we can cover the entire plane $\QQ^2$ with regions of this shape.  Since $\delta_m^p(\nu)$ is invariant under twists and duals, we can compute the function for any $\nu$.

\begin{proof}
We must have $\lambda_1 = -\epsilon$, $\lambda_3 = -((m-1)\epsilon + \varphi)$, and $\lambda_2 = 1-\lambda_1-\lambda_3$.  These are all nonnegative rational numbers with denominators dividing $r$, so the desired nonnegative integers are $A=r\lambda_1$, $B=r\lambda_2$, $C=r\lambda_3$.  By Example \ref{ex-prioritarySum}, the bundle $\cV$ is $F$- and $H_m$-prioritary.

Next we compute $\Delta(\cV)$. We have
$$
\ch(\cV) = (A+B+C,-A E+(A(m-1)-C)F,\frac{1}{2}A(-e-2(m-1)))
$$
so
\begin{align*}\Delta(\cV) &= \frac{-A^2 e-2A(A(m-1)-C)}{2(A+B+C)^2}+\frac{A(e+2m-2)}{2(A+B+C)}
\\&= \frac{A}{2r^2}(B(e+2m-2)+C(e+2m)).
\end{align*} 

If $\Delta(\cV) = -\ell/r^2 \leq 0$, then by taking $\ell$ elementary modifications of  $\cV^{\oplus r}$ we can construct an $F$- and $H_m$-prioritary sheaf of rank $r^2$, slope $\nu$, and discriminant $0$.  This implies $\delta_m^p(\nu) = 0$.

Finally suppose $\Delta(\cV) > 0$; then in particular $\epsilon$ is not an integer.  We compute the line bundle $L_0$ for the character $\ch \cV$.  We have $\lceil \epsilon \rceil = 0$, and
\begin{align*}
\psi &= \varphi+\frac{1}{2}e(\lceil \epsilon\rceil -\epsilon) - \frac{\Delta}{1-(\lceil \epsilon\rceil-\epsilon)}\\
&= \frac{A(m-1)-C}{r}+\frac{eA}{2r}-\frac{A(B(e+2m-2)+C(e+2m))}{2r^2(-\frac{A}{r}+1)}\\
&= -\frac{C}{B+C}.
\end{align*}
There are two cases to consider.

\emph{Case 1: $B>0$.}  If $B>0$, then $\lceil \psi \rceil =0$ and $L_0 = \OO_{\F_e}$.  Then $$\cV(-H_m) = \OO_{\F_e}(-2E - (e+1)F)^A \oplus \OO_{\F_e}(-E-(e+m)F)^B \oplus \OO_{\F_e}(-E-(e+m+1)F))^C$$ has $\chi(\cV(-H_m)) = 0$, and therefore $\Delta(\cV) = \delta_m^p(\nu)$ by Remark \ref{rmk-deltamp}.

\emph{Case 2: $B=0$.}  In this case $\cV$ is actually $H_{m+1}$-prioritary.  Here $\lceil \psi\rceil = -1$, so $L_0 = \OO_{\FF_e}(-F)$.  Notice that $\chi(\cV(-L_0-F)) = \chi(\cV) = 0$ and $\chi(\cV(-L_0-H_{m+1}))=\chi(\cV(-H_m))=0$.  This implies that $\Delta(\cV) = \delta_{m+1}^p (\nu)$, but also $\delta_m^p(\nu) = \delta_{m+1}^p(\nu)$: in the notation of Remark \ref{rmk-deltamp}, the lattice points $(a_0,b_0+1)$ and $(a_1,b_1-1)$ both lie on the hyperbola, so $b_1(\Delta)-b_0(\Delta)-e-1$ jumps from $m-1$ to $m+1$ at $\Delta(\cV)$.
\end{proof}

This construction allows us to complete the classification of $F$- and $H_m$-prioritary sheaves which satisfy the Bogomolov inequality.

\begin{theorem}\label{thm-prioritaryChi}
Let ${\bf v}\in K(\F_e)$ be an integral Chern character of positive rank which satisfies the Bogomolov inequality $\Delta \geq 0$, and let $m\in \Z$.  Let $L_0$ be the line bundle of Definition \ref{def-L0}.  Then the stack $\cP_{F,H_m}({\bf v})$ is nonempty if and only if $$\chi({\bf v}(-L_0 - H_m)) \leq 0.$$  
\end{theorem}
\begin{proof}
By Theorem \ref{thm-prioritaryNecessary}, we need to show that if the inequality holds then the stack is nonempty.  If $r=1$, then the inequality automatically holds and the stack is nonempty, so we may as well assume $r\geq 2$.  By taking twists and/or duals, any slope $(\epsilon,\varphi)\in \QQ^2$ can be moved into the triangular region $T$ with vertices $(-1,-m-1)$, $(0,0)$, $(0,-1)$.  Since both the nonemptiness of the stack and the validity of the inequality are unchanged under taking twists and duals (see Remark \ref{rem-twistDual}), we may as well assume $(\epsilon,\varphi)$ lies in $T$.  The inequality $\chi({\bf v}(-L_0-H_m)) \leq 0$ implies $\Delta \geq \delta_m^p(\nu)$, and then taking elementary modifications of the bundle $\cV$ in Proposition \ref{prop-triangle} of discriminant $\Delta(\cV) \leq \delta_m^p(\nu)$ produces an $F$- and $H_m$-prioritary sheaf of character ${\bf v}$. 
\end{proof}

The next two corollaries are equivalent reformulations of Theorem \ref{thm-prioritaryChi}.

\begin{corollary}\label{cor-prioritaryDelta}
For ${\bf v}$ as in Theorem \ref{thm-prioritaryChi} with total slope $\nu$, the stack $\cP_{F,H_m}({\bf v})$ is nonempty if and only if $$\Delta \geq \delta_m^p(\nu),$$ where $\delta_m^p(\nu)$ is the function of Definition \ref{def-deltaP}.
\end{corollary}

\begin{corollary}\label{cor-prioritaryRho}
For ${\bf v}$ as in Theorem \ref{thm-prioritaryChi} with slope $\nu = \epsilon E + \varphi F$, if $\epsilon$ is not an integer, then the generic prioritary index satisfies
$$\rho_{\gen}({\bf v}) = \left\lfloor\frac{\Delta}{(\lceil\epsilon\rceil-\epsilon)(\epsilon-\lfloor \epsilon\rfloor)}-\frac{e}{2}+1 - (\lceil \psi\rceil -\psi)\right\rfloor$$ where $$\psi = \varphi + \frac{1}{2}e(\lceil \epsilon\rceil - \epsilon) - \frac{\Delta}{\epsilon - \lfloor \epsilon\rfloor}.$$ If $\epsilon$ is an integer, then $\rho_\gen({\bf v}) = \infty$.
\end{corollary}
\begin{proof}
This is immediate from Theorem \ref{thm-prioritaryChi} and Corollary \ref{cor-equivalentInequality}.
\end{proof}

\section{The generic Harder-Narasimhan filtration} \label{sec-genHN}
Throughout this section, we let $m\in \QQ_{>0}$ and suppose ${\bf v}\in K(\F_e)$ is a Chern character such that there are $H_{\lceil m\rceil}$-prioritary sheaves of character $\bv$.  Then the general $F$-prioritary sheaf of character $\bv$ is $H_{\lceil m\rceil}$-prioritary. For a general $\cV\in \cP_{F}({\bf v})$, the numerical invariants of the factors in the $H_m$-Harder-Narasimhan filtration of $\cV$ are fixed.  The goal of this section is to describe a finite computational procedure for determining these numerical invariants.  In particular, we can computationally determine whether the moduli space $M_{H_m}({\bf v})$ is nonempty.

Let $\cV_s/S$ be a complete family of $F$-prioritary sheaves of character ${\bf v}=(r,\nu,\Delta)$ parameterized by a smooth and irreducible variety $S$.  By passing to an open dense subset of $S$, we may assume that every sheaf $\cV_s$ is $H_{\lceil m\rceil}$-prioritary and has an $H_m$-Harder-Narasimhan filtration where the quotients have the same numerical invariants.  Suppose this Harder-Narasimhan filtration has length $\ell$, and the $H_m$-semistable quotients $\gr_{i,s}$ have corresponding $H_m$-Hilbert polynomials $P_i,$ reduced $H_m$-Hilbert polynomials $p_1>\cdots > p_\ell$, and Chern characters ${\bf gr}_i = (r_i,\nu_i,\Delta_i).$  
 
Note that these assumptions require the moduli spaces $M_{H_m}({\bf gr}_i)$ to be nonempty.  Additionally, since the stack $\cP_{F}({\bf v})$ is irreducible, the invariants ${\bf gr}_i$ depend only on $m$ and ${\bf v}$, and not on the particular choice of complete family.

 First we show the total slopes of the terms in the $H_m$-Harder-Narasimhan filtration lie in a narrow strip centered on the slope $\nu$.  This result depends on the existence of $H_{\lceil m\rceil}$-prioritary sheaves.

\begin{lemma}\label{lem-HNclose}
With the above notation, we have $$0\leq (\nu_1-\nu_\ell) \cdot H_m \leq 1,$$ and therefore $$|(\nu_i - \nu) \cdot H_m|< 1$$ for each $i$.
\end{lemma}
\begin{proof}
First suppose $C\subset \F_e$ is a smooth rational curve.  The general $\cV_s$ is either locally free (if $r\geq 2$) or a twist of an ideal sheaf supported at general points of $\F_e$, so $\cV_s|_C$ is locally free.  Recall that if $\cV_s/S$ is a complete family of $\OO_{\F_e}(C)$-prioritary sheaves which are locally free along $C$, then the general $\cV_s$ has restriction $\cV_s|_C$ which is a balanced vector bundle \cite[Proposition 2.6]{CoskunHuizengaBN}, so that $$\mu_{\max}(\cV_s|_C)-\mu_{\min}(\cV_s|_C) \leq 1.$$  Observe that $\mu_{{\max},\OO_{\F_e}(C)}(\cV_s) \leq \mu_{\max}(\cV_s|_C)$. Indeed, suppose $\cF\subset \cV_s$ is a subsheaf.  Then $$\mu_{\OO_{\F_e}(C)}(\cF) = \mu(\cF|_C) \leq \mu_{\max}(\cV_s|_C).$$ Analogously we have $\mu_{{\min},\OO_{\F_e}(C)} (\cV_s) \geq \mu_{\min}(\cV_s|_C)$, and we conclude that $$\mu_{\max,\OO_{\F_e}(C)} (\cV_s) - \mu_{\min,\OO_{\F_e}(C)} (\cV_s) \leq 1$$ holds for a general $s\in S$.  (Even if $L$ is not ample, we write for example $\mu_{\max,L}(\cV)$ for the maximum $L$-slope of a subsheaf of $\cV$, if it exists.  For $L=\OO_{\F_e}(C)$, the above restriction argument shows the maximum exists.)

Now observe that if $\cV_s/S$ is a complete family of $H_{\lceil m\rceil }$-prioritary sheaves, then it is also a complete family of $H_{\lfloor m\rfloor}$-prioritary sheaves.  Therefore by the previous paragraph, for a general $s$ we have 
\begin{align*} \mu_{\max,H_{\lceil m\rceil}}(\cV_s) - \mu_{\min,H_{\lceil m\rceil }}(\cV_s) \leq 1\\
\mu_{\max,H_{\lfloor m\rfloor}}(\cV_s) - \mu_{\min,H_{\lfloor m\rfloor}}(\cV_s) \leq 1.
\end{align*}
Then we have an inequality $\mu_{H_{\lceil m\rceil}}(\gr_{1,s}) \leq \mu_{\max,H_{\lceil m\rceil}}(\cV_s)$, as well as three analogous inequalities obtained by switching $\gr_{1,s}$ with $\gr_{\ell,s}$ or $\lceil m\rceil$ with $\lfloor m\rfloor$.  Write $m$ as a convex combination $m = \lambda \lfloor m\rfloor +(1-\lambda) \lceil m \rceil$.  Then $H_m = \lambda H_{\lfloor m \rfloor} + (1-\lambda) H_{\lceil m\rceil}$, so taking $s$ to be general we find\begin{align*}
 (\nu_1-\nu_\ell)\cdot H_m &= \mu_{\max,H_m} (\cV_s) - \mu_{\min,H_m}(\cV_s) \\& = \mu_{H_m}(\gr_{1,s}) - \mu_{H_m}(\gr_{\ell,s})\\
 &= \lambda(\mu_{H_{\lfloor m\rfloor}}(\gr_{1,s}) - \mu_{H_{\lfloor m\rfloor}}(\gr_{\ell,s}))  + (1-\lambda)(\mu_{H_{\lceil m\rceil}}(\gr_{1,s}) - \mu_{H_{\lceil m\rceil}}(\gr_{\ell,s}))\\
&\leq \lambda(\mu_{\max,H_{\lfloor m\rfloor}}(\cV_s) - \mu_{\min,H_{\lfloor m\rfloor}}(\cV_s)) + (1-\lambda)(\mu_{\max,H_{\lceil m\rceil}}(\cV_s)-\mu_{\min,H_{\lceil m \rceil}}(\cV_s))\\
&\leq \lambda + (1-\lambda)\\
& = 1.
 \end{align*}  
Both $\nu_i\cdot H_m$ and $\nu\cdot H_m$ lie in the interval $[\nu_{\ell}\cdot H_m,\nu_1\cdot H_m]$ of length $\leq 1$ (and $\nu\cdot H_m$ is in the interior unless the interval is a point), so the second statement follows immediately.
\end{proof}

Next we show that the characters ${\bf gr}_i$ must satisfy strong orthogonality properties.

\begin{lemma}\label{lem-HNorthogonal}
With the above notation, we must have $\chi({\bf gr}_i,{\bf gr}_j) = 0$ for all $i<j$.
\end{lemma}
\begin{proof}
Let $S_{H_m}(P_1,\ldots,P_\ell)= S$ be the Schatz stratum parameterizing sheaves $\cV_s$ with an $H_m$-Harder-Narasimhan filtration having quotients with  Hilbert polynomials $P_1,\ldots,P_\ell$.  Therefore the codimension of the stratum in $S$ is $0$.

On the other hand, we can use a standard argument to compute the codimension of the stratum in a different way.  Fix some $s\in S$, and equip $\cV_s$ with its Harder-Narasimhan filtration with quotients $\gr_{1,s},\ldots,\gr_{\ell,s}$.    Then by \cite[\S 15.3]{LePotier} there is a spectral sequence with $E_1$-term given by $$E_1^{p,q} = \begin{cases} \bigoplus_i \Ext^{p+q}(\gr_{i,s},\gr_{i-p,s}) &\textrm{if $p<0$}\\ 0 & \textrm{if $p\geq 0$}\end{cases}$$ which abuts on $\Ext_+^{p+q}(\cV_s,\cV_s)$ in degree $p+q$.  Similarly, there is a spectral sequence with $E_1$-term given by $$E_1^{p,q} = \begin{cases} 0 &\textrm{if $p<0$}\\ \bigoplus_i \Ext^{p+q}(\gr_{i,s},\gr_{i-p,s}) & \textrm{if $p\geq 0$}\end{cases}$$ which abuts on $\Ext_-^{p+q}(\cV_s,\cV_s)$ in degree $p+q$.  

Since $\Hom(\gr_{i,s},\gr_{j,s}) = 0$ for $i<j$ by semistability, we see that $\Ext_+^0(\cV_s,\cV_s)=0$.  Likewise, Lemma \ref{lem-HNclose} gives $$\mu_{\max,H_m}(\cV_s)-\mu_{\min,H_m}(\cV_s) \leq 1 < -K_{\F_e}\cdot H_m,$$ so $$\ext^2(\gr_i(\cV_s),\gr_j(\cV_s))=\hom(\gr_j(\cV_s),\gr_i(\cV_s)\te K_{\F_e}) = 0$$ for any $i,j,$ and therefore both $\Ext^2_+(\cV_s,\cV_s)=0$ and $\Ext^2_-(\cV_s,\cV_s)=0$.  So, the only nonzero terms in the spectral sequence for $\Ext_+^{p+q}(\cV_s,\cV_s)$ have $p+q = 1$, and we conclude $$\ext_+^1(\cV_s,\cV_s) = \sum_{i<j} \ext^1(\gr_i(\cV_s),\gr_j(\cV_s)) = -\sum_{i<j} \chi({\bf gr}_i,{\bf gr}_j).$$ By using the scheme of relative flags we can then show that $S_{H_m}(P_1,\ldots,P_\ell)\subset S$ is smooth at $s$ and the normal space has dimension $\ext^1_+(\cV_s,\cV_s)$; see \cite[\S 15.4]{LePotier} for details.  Since $\chi({\bf gr}_i,{\bf gr}_j)\leq 0$ for all $i<j$, the result follows.
\end{proof}

Conversely, the restrictions in Lemmas \ref{lem-HNclose} and \ref{lem-HNorthogonal} \emph{determine} the characters ${\bf gr}_i$.

\begin{theorem}\label{thm-HNcriterion}
Suppose ${\bf w}_1,\ldots, {\bf w}_k\in K(\F_e)$ are characters of positive rank with the following properties.
\begin{enumerate}
\item ${\bf w}_1+\cdots + {\bf w}_k = {\bf v}$.
\item $q_1 > \ldots > q_k,$ where $q_i$ is the reduced $H_m$-Hilbert polynomial corresponding to ${\bf w}_i$.
\item $\mu_{H_m}({\bf w}_1) - \mu_{H_m}({\bf w}_k) \leq 1$.
\item $\chi({\bf w}_i,{\bf w}_j) = 0$ for $i<j$.
\item The moduli space $M_{H_m}({\bf w}_i)$ is nonempty for each $i$.
\end{enumerate}
Then $k = \ell$ and ${\bf gr}_i = {\bf w}_i$ for each $i$.
\end{theorem}
\begin{proof}
Pick $H_m$-semistable sheaves $\cW_i \in M_{H_m}({\bf w}_i)$ for each $i$, and consider the sheaf $$\cU := \bigoplus_i \cW_i,$$ so that the Harder-Narasimhan filtration of $\cU$ has factors $\cW_1,\ldots, \cW_k$.
Then by assumption $$\mu_{\max,H_m}(\cU) - \mu_{\min,H_m}(\cU) = \mu_{H_m}(\cW_1) - \mu_{H_m}(\cW_k) \leq 1,$$ so by Proposition \ref{prop-ssPrior} we see that $\cU$ is $H_{\lceil m\rceil}$-prioritary of character ${\bf v}$.  

We can now construct a complete family $\cU_t/\Sigma$ parameterized by a smooth, irreducible variety $\Sigma$ such that $\cU = \cU_{t_0}$ for some $t_0\in T$.  Indeed, let $d \gg 0$ be sufficiently large and divisible, let $\chi = \chi(\OO_{\F_e}(-dH_m),\cU)$, and consider the universal family of quotients $\cU_t/\Sigma$ on $\Sigma = \Quot(\OO_{\F_e}(-dH_m)^{\chi},\ch \cU)$ parameterizing quotients \begin{equation}\label{eqn-sequence}\tag{$\spadesuit$} 0\to \cK_t \to \OO_{\F_e}(-dH_m)^{\chi} \to \cU_t\to 0.\end{equation} Let $t_0\in \Sigma$ be the point corresponding to the canonical evaluation $$\OO_{\F_e}(-dH_m) \te \Hom(\OO_{\F_e}(-dH_m),\cU)\to \cU.$$  The tangent space to $\Sigma$ at a point $t$ corresponding to an exact sequence (\ref{eqn-sequence}) is $\Hom(\cK_t,\cU_t)$, and $\Sigma$ is smooth at $t$ if $\Ext^1(\cK_t,\cU_t) = 0$.  Applying $\Hom(-,\cU_t)$ to (\ref{eqn-sequence}), we get the following portion of the long exact sequence in cohomology:
$$\Hom( \cK_t,\cU_t) \to \Ext^1(\cU_t,\cU_t) \to \Ext^1(\OO_{\F_e}(-dH_m)^\chi, \cU_t) \to \Ext^1(\cK_t,\cU_t)\to \Ext^2(\cU_t,\cU_t).$$ By passing to the open subset parameterizing locally-free sheaves if necessary, we have $\ext^2(\cU_t, \cU_t) = \hom(\cU_t, \cU_t(K_{\FF_e}))=0$ by our assumptions on the slopes.  
Since $d\gg 0$, we have $\Ext^1(\OO_{\F_e}(-dH_m)^\chi,\cU_{t})=0$ by Serre vanishing and boundedness of the Quot scheme. Therefore $\Ext^1(\cK_{t},\cU_{t})=0$ and $\Sigma$ is smooth at $t$, including at $t=t_0$.  Furthermore, the Kodaira-Spencer map at $t$ is  the natural map $$T_{t}\Sigma = \Hom(\cK_{t},\cU_{t}) \to \Ext^1(\cU_{t},\cU_{t}),$$ so the universal family on $\Sigma$ is complete at $t$, including at $t=t_0$.  We have thus constructed the required complete family $\cU_t/\Sigma$.

Let $Q_i$ be the $H_m$-Hilbert polynomial corresponding to ${\bf w}_i$.  Then by the same computation as in the proof of Lemma \ref{lem-HNorthogonal}, the Schatz stratum $S_{H_m}(Q_1,\ldots,Q_k) \subset \Sigma$ is smooth at $t_0$ of codimension $0$.  Therefore the stratum is dense in $\Sigma$, and the general sheaf $\cU_t$ has an $H_m$-Harder-Narasimhan filtration with quotients of character ${\bf w}_i$.  Thus ${\bf gr}_i = {\bf w}_i$.\end{proof}

It follows that if we know (say by induction) when moduli spaces $M_{H_m}({\bf w})$ are nonempty for characters ${\bf w}$ with rank smaller than ${\bf v}$, then we can determine the characters ${\bf gr}_i$ by searching for the unique list of characters ${\bf w}_1,\ldots, {\bf w}_k$ satisfying the assumptions of Theorem \ref{thm-HNcriterion}.  If there is no such list with length $k\geq 2$, then the only possibility is that $k=1$ and the moduli space $M_{H_{m}}({\bf v})$ is nonempty.

We can make a couple quick observations to place further bounds on the characters ${\bf gr}_i$ which reduce the search for a list of characters ${\bf w}_1,\ldots,{\bf w}_k$ satisfying the assumptions of Theorem \ref{thm-HNcriterion} to a finite computation.  First, we bound the total slopes $\nu_i$ of the ${\bf gr}_i$ to a parallelogram region centered on $\nu$.

\begin{lemma}\label{lem-slopeQuad}
For each $i$, we have \begin{align*} |(\nu_i - \nu)\cdot F| &< \max\{1,\frac{2}{e+2m}\}\\ |(\nu_i-\nu)\cdot H_m| &< 1.\end{align*}
\end{lemma}
\begin{proof}
The second inequality was Lemma \ref{lem-HNclose}.  For any $i<j$, the orthogonality $\chi({\bf gr}_i,{\bf gr}_j)=0$, Riemann-Roch, and the Bogomolov inequality give $$P(\nu_j-\nu_i) = \Delta_i + \Delta_j \geq 0.$$ We also have the inequalities $$ -1 \leq (\nu_j-\nu_i)\cdot H_m \leq 0$$ from Lemma \ref{lem-HNclose}.
Write $\nu_j-\nu_i = aE+bF$.  Then these inequalities give $$P(aE+bF) = (a+1)(b+1-\frac{1}{2}ea) \geq 0$$ and $$-1 \leq am+b  \leq 0.$$
These inequalities easily imply $$ - 1 \leq a \leq \frac{2}{e+2m},$$ or $$|(\nu_j-\nu_i)\cdot F| \leq \max\{1,\frac{2}{e+2m}\}.$$  The result follows since $\nu = \alpha_1\nu_1+\cdots +\alpha_\ell \nu_\ell$ is a weighted mean of the $\nu_i$'s, so that $$
|(\nu_i-\nu)\cdot F| = \left| \sum_j \alpha_j(\nu_i-\nu_j)\cdot F\right| \leq \sum_j \alpha_j |(\nu_i-\nu_j)\cdot F| < \max\{1,\frac{2}{e+2m}\},$$
with the last inequality being strict because of the $0$ term $\alpha_i|(\nu_i-\nu_i)\cdot F|=0$ in the sum.
\end{proof}

The discriminants $\Delta_i$ of the ${\bf gr}_i$ are also easy to describe.

\begin{lemma}\label{lem-discBound}
Suppose $\ell \geq 2$, so that there are no semistable sheaves of character ${\bf v}$.
\begin{enumerate}
\item If $\nu_1 \cdot H_m > \nu_\ell \cdot H_m$, then each discriminant $\Delta_i$ is the smallest discriminant of an $H_m$-semistable sheaf with arbitrary rank and slope $\nu_i$.
\item If $\nu = \nu_1  = \cdots  = \nu_\ell$, then $\ell  =2$.  Then $\Delta_1$ (resp. $\Delta_2$) is the smallest (resp. second smallest) discriminant of an $H_m$-semistable sheaf with arbitrary rank and slope $\nu$.
\end{enumerate}
\end{lemma}
If $m$ is generic and $\nu_1\cdot H_m = \nu_\ell \cdot H_m$, then $\nu_1=\cdots = \nu_\ell$, so one of the two cases of Lemma \ref{lem-discBound} always applies when $m$ is generic.
\begin{proof}
(1) Given an index $1\leq i\leq \ell$, either $\nu_1\cdot H_m > \nu_i \cdot H_m$ or $\nu_i \cdot H_m > \nu_\ell\cdot H_m$.  Assume $\nu_1\cdot H_m > \nu_i \cdot H_m$; the other case is similar.  Suppose there is an $H_m$-semistable sheaf $\cV$ of rank $r'$, slope $\nu_i$, and discriminant $\Delta' < \Delta_i$.  We have $\chi({\bf gr}_1,{\bf gr}_i) = 0$, so Riemann Roch gives $$\frac{\chi({\bf gr}_1,\cV)}{r_1r'}=P(\nu_i-\nu_1) - \Delta_1-\Delta'  > P(\nu_i-\nu_1)-\Delta_1-\Delta_i =\frac{\chi({\bf gr}_1,{\bf gr}_i)}{r_1r_i}=0,$$ and therefore $\chi({\bf gr}_1,\cV)> 0$.  If $\gr_1$ is an $H_m$-semistable sheaf of character ${\bf gr}_1$, then $$\hom( \gr_1,\cV) = \ext^2(\gr_1,\cV) = 0$$ by semistability and Serre duality since $(\nu_1-\nu_i)\cdot H_m \leq 1$.  Therefore $\chi({\bf gr}_1,\cV) \leq 0$, a contradiction.

(2) First note that the discriminants satisfy $\Delta_1< \Delta_2 < \cdots < \Delta_\ell$.  Then since $\chi({\bf gr}_1,{\bf gr}_2)=0$, Riemann-Roch shows $\chi({\bf gr}_1,{\bf gr}_\ell) < 0$ if $\ell > 2$.  Therefore $\ell = 2$.

The rest of the proof is similar to (1).  We have $\chi({\bf gr}_1,{\bf gr}_2) = 0$.  Suppose we can find an $H_m$-semistable sheaf $\cV$ with slope $\nu$ and discriminant $\Delta'<\Delta_1$.  Then $\chi(\cV,{\bf gr}_2) >0$, even though $\hom(\cV,\gr_2) = \ext^2(\cV,\gr_2)=0$ for any $H_m$-semistable sheaf $\gr_2$ of character ${\bf gr}_2$.

Similarly, suppose we can find an $H_m$-semistable sheaf $\cV$ with slope $\nu$ and discriminant $\Delta'$ satisfying $\Delta_1 < \Delta' < \Delta_2$. Then $\chi({\bf gr}_1,\cV) > 0$, but $\hom(\gr_1,\cV) = \ext^2(\gr_1,\cV) = 0$ for any $H_m$-semistable sheaf $\gr_1$ of character ${\bf gr}_1$.
\end{proof}

Finally, the length $\ell$ of the generic Harder-Narasimhan filtration is easy to bound.

\begin{lemma}
The characters ${\bf gr}_1,\ldots,{\bf gr}_\ell\in K(\F_e)\te \QQ$ are linearly independent.  Therefore $\ell \leq 4$.
\end{lemma}
\begin{proof}
Suppose the characters are dependent.  Then we can partition $\{1,\ldots,\ell\} = A\coprod B$ into two sets $A$ and $B$ and find nonnegative integers $\{m_a\}_{a\in A}$ and $\{n_b\}_{b\in B}$ (not all $0$) such that $$ \sum_{a\in A} m_a {\bf gr}_a = \sum_{b\in B} n_b {\bf gr}_b=:{\bf w}.$$ The stack $\cP_{H_{\lceil m\rceil}} ({\bf w})$ is nonempty since it contains the sheaf $\bigoplus_{a\in A} \gr_a^{\oplus m_a}$, as in the proof of Theorem \ref{thm-HNcriterion}. Then Theorem \ref{thm-HNcriterion} shows that both $\{m_a {\bf gr}_a\}_{a\in A}$ and $\{n_b {\bf gr}_b\}_{b\in B}$ are the characters of the quotients in the $H_m$-Harder-Narasimhan filtration of a general sheaf $\cW\in \cP_{H_{\lceil m\rceil}}({\bf w})$.  However, the ${\bf gr}_i$ are all distinct and the $H_m$-Harder-Narasimhan filtration is unique, so this is a contradiction.
\end{proof}

If we have computed the generic Harder-Narasimhan filtration for all characters of rank $<r$, then it becomes easier to find the generic Harder-Narasimhan filtration for a character of rank $r$.  Instead of searching over all lists $\bgr_1,\ldots,\bgr_\ell$ for the generic Harder-Narasimhan filtration, we can instead just search for the character $\bgr_1$ of the maximal destabilizing subsheaf.  Then under the assumption that $\bgr_1$ is the first character in the generic Harder-Narasimhan filtration, the rest of the characters $\bgr_2,\ldots,\bgr_\ell$ are determined by induction.

\begin{lemma}
Let ${\bf u} = \bgr_2 + \cdots + \bgr_\ell$.  Then there are $H_{\lceil m\rceil}$-prioritary sheaves of character ${\bf u}$, and the generic $\cU \in \cP_{H_{\lceil m\rceil}}({\bf u})$ has an $H_m$-Harder-Narasimhan filtration with quotients of characters $\bgr_2,\ldots,\bgr_\ell$.
\end{lemma}
\begin{proof}
Considering slopes, a direct sum of sheaves of characters $\bgr_2,\ldots,\bgr_\ell$ is $H_{\lceil m\rceil}$-prioritary, so such sheaves exist.  Since the characters $\bgr_1,\ldots,\bgr_\ell$ are the characters of the $H_m$-Harder-Narasimhan filtration of a general sheaf in $\cP_{H_{\lceil m\rceil}}(\bv)$, they satisfy all the properties listed in Theorem \ref{thm-HNcriterion}.  The theorem also shows that $\bgr_2,\ldots,\bgr_\ell$ are the characters of the $H_m$-Harder-Narasimhan filtration of a general sheaf in $\cP_{H_{\lceil m\rceil}}({\bf u})$.
\end{proof}

This gives us the following inductive version of Theorem \ref{thm-HNcriterion}.

\begin{corollary}\label{cor-algorithm}
Let $\bw_1$ be a character and let ${\bf u} = \bv - \bw_1$.  Then $\bw_1 = \bgr_1$ if and only if the following conditions are satisfied.
\begin{enumerate}
\item There are $H_{\lceil m\rceil}$-prioritary sheaves of character ${\bf u}$.  In this case, let $\bw_2,\ldots,\bw_k$ be the characters of the $H_m$-Harder-Narasimhan filtration of a general sheaf in $\cP_{H_{\lceil m\rceil}}({\bf u})$.
\item $q_1 > q_2$, where $q_i$ is the reduced $H_m$-Hilbert polynomial corresponding to $\bw_i$.
\item $\mu_{H_m} (\bw_1) - \mu_{H_m}(\bw_k) \leq 1$.
\item $\chi(\bw_1,\bw_i) = 0$ for $i\geq 2$.
\item The moduli space $M_{H_m}(\bw_1)$ is nonempty.
\end{enumerate}
Furthermore, if these conditions are satisfied, then $\bw_i = \bgr_i$ for all $i$.
\end{corollary}
\begin{proof}
Everything follows immediately from Theorem \ref{thm-HNcriterion} and the fact that $\bw_2,\ldots,\bw_\ell$ are the characters of a generic Harder-Narasimhan filtration.
\end{proof}

\section{Exceptional bundles and necessary conditions for stability}\label{sec-exceptional}

In this section we study exceptional bundles on Hirzebruch surfaces and their stability.  First we recall known results from the case of a del Pezzo surface with anticanonical polarization to show that the exceptional bundles on $\F_0$ and $\F_1$ can be explicitly determined by induction on the rank.  For a given polarization, the exceptional bundles define a Dr\'ezet-Le-Potier type surface which restricts the numerical invariants of semistable bundles.  

\subsection{Exceptional bundles on del Pezzo and Hirzebruch surfaces} Here we recall some previous results on exceptional bundles, mostly for anticanonically polarized del Pezzo surfaces and Hirzebruch surfaces.

\begin{definition}
A sheaf $\cV$ on a smooth surface $X$ is
\begin{enumerate}
\item \emph{simple}, if $\Hom(\cV,\cV) = \CC$;
\item \emph{rigid}, if $\Ext^1(\cV,\cV) = 0$;
\item \emph{exceptional}, if it is simple, rigid, and $\Ext^2(\cV,\cV) = 0$;
\item \emph{semiexceptional}, if it is a direct sum of copies of an exceptional sheaf.
\end{enumerate} 
We call a character $\bv\in K(X)$ of positive rank \emph{potentially exceptional} if $\chi(\bv,\bv)=1$, and \emph{exceptional} if there is an exceptional bundle of character $\bv$.
\end{definition}

Simplicity automatically implies strong results about prioritariness when $-K_X$ is effective.

\begin{lemma}\label{lem-simple}
Let $\cV$ be a simple bundle on a smooth surface $X$.  If $D\in \Pic(X)$ is a divisor such that $-(K_X+D)$ is nontrivial and effective, then $\cV$ is $D$-prioritary.

In particular, any simple bundle on $\F_e$ is $H_2$-prioritary.
\end{lemma}
\begin{proof}
By Serre duality, $$\ext^2(\cV,\cV(-D)) = \hom(\cV(-D),\cV(K_X))=\hom(\cV,\cV(K_X+D)).$$ Suppose there is a nonzero map $\cV\to \cV(K_X+D)$.  Since $-(K_X+D)$ is effective, we can compose this with an injective map $\cV(K_X+D)\to \cV$ to get a map $\cV\to \cV$.  Since $-(K_X+D)$ is nontrivial, this map is not an isomorphism, but it is also not zero since $\cV(K_X+D)\to \cV$ is injective.  This contradicts the simplicity of $\cV$, so $\hom(\cV,\cV(K_X+D))=0$.
 
 In the Hirzebruch case, we have $-(K_{\F_e}+ H_2) = E$.
\end{proof}

As a consequence, exceptional bundles on Hirzebruch surfaces are determined by their Chern characters.

\begin{proposition}\label{prop-excPrior}
Let $\cV$ be an exceptional bundle on $\F_e$.
\begin{enumerate}
\item If $D\in \Pic(\F_e)$ is any divisor such that $-(K_{\F_e}+D)$ is effective, then $\cV$ is $D$-prioritary.  In particular, $\cV$ is $F$-prioritary and $H_2$-prioritary.
\item Any exceptional bundle with the same invariants as $\cV$ is isomorphic to $\cV$.
\end{enumerate}
\end{proposition}
\begin{proof}
(1) This is Lemma \ref{lem-simple}.

(2) The stack $\cP_F({\bf v})$ is irreducible.  Therefore, in any complete family of $F$-prioritary sheaves of character ${\bf v}$, the general sheaf is isomorphic to $\cV$.  The same argument applies to any other exceptional bundle.
\end{proof}

The next result was first proved by Mukai for K3 surfaces, then restated in a way that holds more generally in \cite{Gorodentsev}.

\begin{proposition}[\cite{Mukai,Gorodentsev}]\label{prop-mukai}
Let $X$ be a smooth surface.  
\begin{enumerate}
\item If $\cV$ is a torsion-free sheaf on $X$, then $$\ext^1(\cV,\cV) \geq \ext^1(\cV^{**},\cV^{**}) + 2 \length(\cV^{**}/\cV).$$ In particular, if $\cV$ is rigid, then it is locally free.

\item Suppose $$0\to \cW \to \cV \to \cU \to 0$$ is a short exact sequence with $\Hom(\cW,\cU) = \Ext^2(\cU,\cW) = 0$.  Then $$\ext^1(\cV,\cV) \geq \ext^1(\cW,\cW) + \ext^1(\cU,\cU).$$ Thus if $\cV$ is rigid, then so are $\cW$ and $\cU$.
\end{enumerate}
\end{proposition}

Rigid bundles on del Pezzo surfaces decompose into exceptional bundles.

\begin{theorem}[{\cite[Theorem 5.2]{KuleshovOrlov}}]\label{thm-rigidSplit}
Let $X$ be a del Pezzo surface.  Then any rigid bundle $\cV$ splits as a direct sum of exceptional bundles.
\end{theorem}

The exceptional bundles on an anticanonically polarized del Pezzo surface are automatically slope-stable.

\begin{theorem}[{\cite{Gorodentsev}}]\label{thm-delPezzoExc}
Let $X$ be a del Pezzo surface.  Then any exceptional bundle is $\mu_{-K_X}$-stable.
\end{theorem}

The next lemma collects several useful facts about exceptional bundles on Hirzebruch surfaces.

\begin{lemma}\label{lem-excFacts} Let $\bv\in K(\F_e)$ be a potentially exceptional character of rank $r$ with $c_1(\bv) = aE + bF$. 
\begin{enumerate}
\item The discriminant of $\bv$ is $$\Delta = \frac{1}{2} - \frac{1}{2r^2}.$$
\item The integers $r$ and $a$ are coprime, and $r$ is odd if $e$ is even.
The integer $b$ satisfies the congruence $$2ab \equiv a^2e+aer-r^2-1 \pmod{2r}.$$ Conversely, the residue class of $b \pmod r$ is uniquely determined by $r$ and $a$.

\item The character $\bv$ is primitive.  

\item If $\cV$ is an $H_m$-stable sheaf of discriminant $\Delta(\cV)<\frac{1}{2}$, then $\cV$ is exceptional.

\item If $m$ is generic and $\cV$ is a $\mu_{H_m}$-semistable sheaf of character $\bv$, then it is $\mu_{H_m}$-stable and exceptional.

 \item If $m$ is generic and $\cV$ is an $H_m$-semistable sheaf of discriminant $\Delta(\cV)< \frac{1}{2}$, then $\cV$ is semiexceptional.
\end{enumerate}
\end{lemma}
\begin{proof}

(1) We have $\chi(\bv,\bv) = 1$, so solving the Riemann-Roch formula $$1 = \chi(\bv,\bv) = r^2(P(0)-2\Delta)$$ for $\Delta$ proves the equality.

(2) The Euler characteristic $\chi(\bv)$ must be an integer.  By Riemann-Roch,
\begin{align*}\chi(\bv) &= r\left(P\left(\frac{a}{r}E+\frac{b}{r}F\right)-\frac{1}{2}+\frac{1}{2r^2}\right)\\
&= r\left(\left(\frac{a}{r}+1\right)\left(\frac{b}{r}+1-\frac{ea}{2r}\right)-\frac{1}{2}+\frac{1}{2r^2}\right)\\
&= \frac{1}{2r}\left(1+2ab+2r(a+b)-a^2e-aer+r^2\right).
\end{align*}
Therefore, $2r$ divides $1+2ab-a^2e-aer+r^2$, giving the congruence $$2ab \equiv a^2e+aer-r^2-1 \pmod{2r}.$$  Then $e$ and $r$ cannot both be even.  Furthermore, $1$ is a $\Z$-linear combination of $r$ and $a$.

Note that $r$ is odd if either $a$ or $e$ are even.  Therefore $a^2e+aer-r^2-1$ is always even, and the congruence is equivalent to $$ab \equiv \frac{1}{2}(a^2e+aer-r^2-1) \pmod r.$$ Since $a$ and $r$ are coprime, this uniquely determines $b \pmod r$.

(3) Clearly $\bv$ is primitive by (2).  

(4)
We have $\hom(\cV,\cV) = 1$ and $\ext^2(\cV,\cV) = 0$, so $$\chi(\cV,\cV) = 1- \ext^1(\cV,\cV) = r^2(1-2\Delta) > 0.$$ Therefore $\ext^1(\cV,\cV) = 0$.

(5) Since $\bv$ is primitive and $m$ is generic, $\cV$ has no subsheaf of smaller rank with the same $H_m$-slope.  Thus $\cV$ is $\mu_{H_m}$-stable, and $\cV$ is exceptional by (4).

(6) Since $m$ is generic, the Jordan-H\"older factors $\gr_1,\ldots,\gr_\ell$ of $\cV$ all have the same total slope and discriminant.  They are also exceptional bundles, by (1), so their Chern characters are primitive, hence have the same rank, and they are the same.  By Proposition \ref{prop-excPrior} (2), the factors are all isomorphic.  Then an easy induction using $\ext^1(\gr_1,\gr_1) = 0$ shows $\cV \cong \gr_1^{\oplus \ell}$.
\end{proof}

\subsection{The Dr\'ezet-Le Potier surface}\label{ssec-DLP} Consider a Hirzebruch surface $\F_e$ polarized by an ample divisor $H$.  Suppose $\cV$ is a $\mu_H$-stable sheaf on $\F_e$.
The existence of $\cV$ restricts the possible numerical invariants of $\mu_H$-semistable sheaves.  In particular, if $\cW$ is a $\mu_H$-semistable sheaf with $$\frac{1}{2}K_{\F_e}\cdot H\leq \mu_{H} (\cW)-\mu_H(\cV) < 0,$$ then $\Hom(\cV,\cW) = 0$ and $$\ext^2(\cV,\cW) = \hom(\cW,\cV(K_{\F_e}))=0$$ by stability and Serre duality.  Therefore, $\chi(\cV,\cW)\leq 0$.  By the Riemann-Roch formula, this inequality can be viewed as a lower bound on $\Delta(\cW):$  $$\Delta(\cW) \geq P(\nu(\cW)-\nu(\cV))-\Delta(\cV).$$ Likewise, if instead $$0 < \mu_{H}(\cW)-\mu_H(\cV) \leq -\frac{1}{2}K_{\F_e}\cdot H,$$ then the inequality $\chi(\cW,\cV)\leq 0$ provides a lower bound $$\Delta(\cW) \geq P(\nu(\cV)-\nu(\cW))-\Delta(\cV)$$ on $\Delta(\cW)$.

Heuristically, $\mu_H$-stable exceptional bundles $\cV$ often give strong bounds for $\mu_H$-semistability since their discriminants are small.  For a $\mu_H$-stable exceptional bundle $\cV$, we define a function $$\DLP_{H,\cV}(\nu)  = \begin{cases} P(\nu-\nu(\cV)) - \Delta(\cV) & \textrm{if } \frac{1}{2}K_{\F_e}\cdot H\leq (\nu-\nu(\cV)) \cdot H < 0\\
P(\nu(\cV)-\nu) - \Delta(\cV) & \textrm{if } 0 <  (\nu-\nu(\cV)) \cdot H\leq  -\frac{1}{2}K_{\F_e}\cdot H\\
\max \{P(\pm(\nu-\nu(\cV)))-\Delta(\cV)\} &\textrm{if } (\nu-\nu(\cV))\cdot H = 0.
 \end{cases}$$ on the strip of slopes $\nu = \frac{a}{r}E + \frac{b}{r}F = (\frac{a}{r},\frac{b}{r})\in \QQ^2$ satisfying $$|(\nu-\nu(\cV))\cdot H| \leq -\frac{1}{2} K_{\F_e}\cdot H.$$
Our previous discussion shows that if there is a $\mu_H$-semistable sheaf of total slope $\nu$ and discriminant $\Delta$ such that $0 < |(\nu-\nu(\cV))\cdot H| \leq -\frac{1}{2} K_{\F_e}\cdot H$, then $\Delta \geq \DLP_{H,\cV}(\nu)$.  
 
\begin{remark}
The definition of $\DLP_{H,\cV}(\nu)$ in the third case $(\nu-\nu(\cV))\cdot H=0$ is somewhat arbitrary.  In particular, we don't necessarily know that $\Delta \geq \DLP_{H,\cV}(\nu)$ whenever there is a $\mu_H$-semistable sheaf of total slope $\nu$ and discriminant $\Delta$ such that $(\nu-\nu(\cV))\cdot H = 0$.  

However, if $H$ is  generic, then $(\nu - \nu(\cV))\cdot H = 0$ will only happen when $\nu = \nu(\cV)$.  Suppose $\cW$ is $H$-semistable of total slope $\nu(\cW) = \nu(\cV)$.   If $\Delta(\cW) = \Delta(\cV)$, then $\cW$ is semiexceptional by Lemma \ref{lem-excFacts} (6).  On the other hand if $\Delta(\cW) \neq \Delta(\cV)$, then either $\Hom(\cW,\cV)=0$ or $\Hom(\cV,\cW) = 0$ by $H$-semistability, and in either case Riemann-Roch implies $$\Delta(\cW) \geq \DLP_{H,\cV}(\nu) = \frac{1}{2}+\frac{1}{2r(\cV)^2}.$$  Thus if $H$ is generic, then $\Delta \geq \DLP_{H,\cV}(\nu)$ whenever there is an $H$-semistable sheaf of total slope $\nu$ and discriminant $\Delta$ satisfying $|(\nu-\nu(\cV))\cdot H| \leq -\frac{1}{2}K_{\F_e}\cdot H$.
\end{remark}

\begin{remark}
In the previous remark, if $H$ is generic and $\cW$ is only $\mu_H$-semistable, we cannot conclude $\Delta(\cW) \geq \DLP_{H,\cV}(\nu)$.  For example, the sheaf $\OO_{\F_e} \oplus I_p$  on $\F_e$ is only $\mu_H$-semistable, but has discriminant $\frac{1}{2} < \DLP_{H,\OO_{\F_e}}(0)=1.$
\end{remark}

 \begin{remark}\label{rem-DLPK}
It is useful to expand the two branches defining the function $\DLP_{H,\cV}(\nu)$.  We have \begin{align*} P(\nu-\nu(\cV))- \Delta(\cV) &= \frac{1}{2}(\nu-\nu(\cV))^2 + 1 - \Delta(\cV)-  \frac{1}{2}K_{\F_e}\cdot (\nu-\nu(\cV)) \\ 
 P(\nu(\cV)-\nu)-\Delta(\cV)&= \frac{1}{2}(\nu-\nu(\cV))^2+1-\Delta(\cV)+\frac{1}{2}K_{\F_e}\cdot (\nu-\nu(\cV)). \end{align*}  Thus the two branches have very similar formulas; the only difference is that depending on the sign of $(\nu-\nu(\cV))\cdot H$, we either subtract or add $\frac{1}{2}K_{\F_e}\cdot (\nu-\nu(\cV)).$

In particular, in the special case where $e=0$ or $1$ and $H  = -K_{\F_e}$, if $(\nu-\nu(\cV))\cdot (-K_{\F_e}) = 0$, then both branches always produce the same result.  Thus the two numbers being maximized in the third part of the definition of $\DLP_{-K_{\F_e},\cV}(\nu)$ are the same, and $\DLP_{-K_{\F_e},\cV}(\nu)$ is continuous on its domain in this case.

For all polarizations that are not a multiple of $-K_{\F_e}$, the function $\DLP_{H,\cV}(\nu)$ is discontinuous along the line $(\nu-\nu(\cV))\cdot H=0$, although the two branches do produce the same result at the point $\nu = \nu(\cV)$. 
\end{remark}

\begin{example}\label{ex-Crescents}
In Figure \ref{fig-Crescents}, we let $e=0$ or $1$ and plot the functions $\DLP_{-K_{\F_e},\OO_{\F_e}}(\epsilon E + \varphi F)$ for $(\epsilon,\varphi)$ in the square $[-2,2]\times [-2,2]$.  We cut off the graph below $\Delta = 0$, since the Bogomolov inequality will apply in that case.  The black line indicates the line of slopes where the branch used to compute the function changes.  Observe that the surface is highest for total slopes close to $\nu(\OO_{\F_e})$.  For other exceptional bundles, the graph is translated by the total slope and shifted downward by the discriminant.

\begin{figure}[t]
\begin{center}
\includegraphics[bb=0 0 5.86in 2.75in]{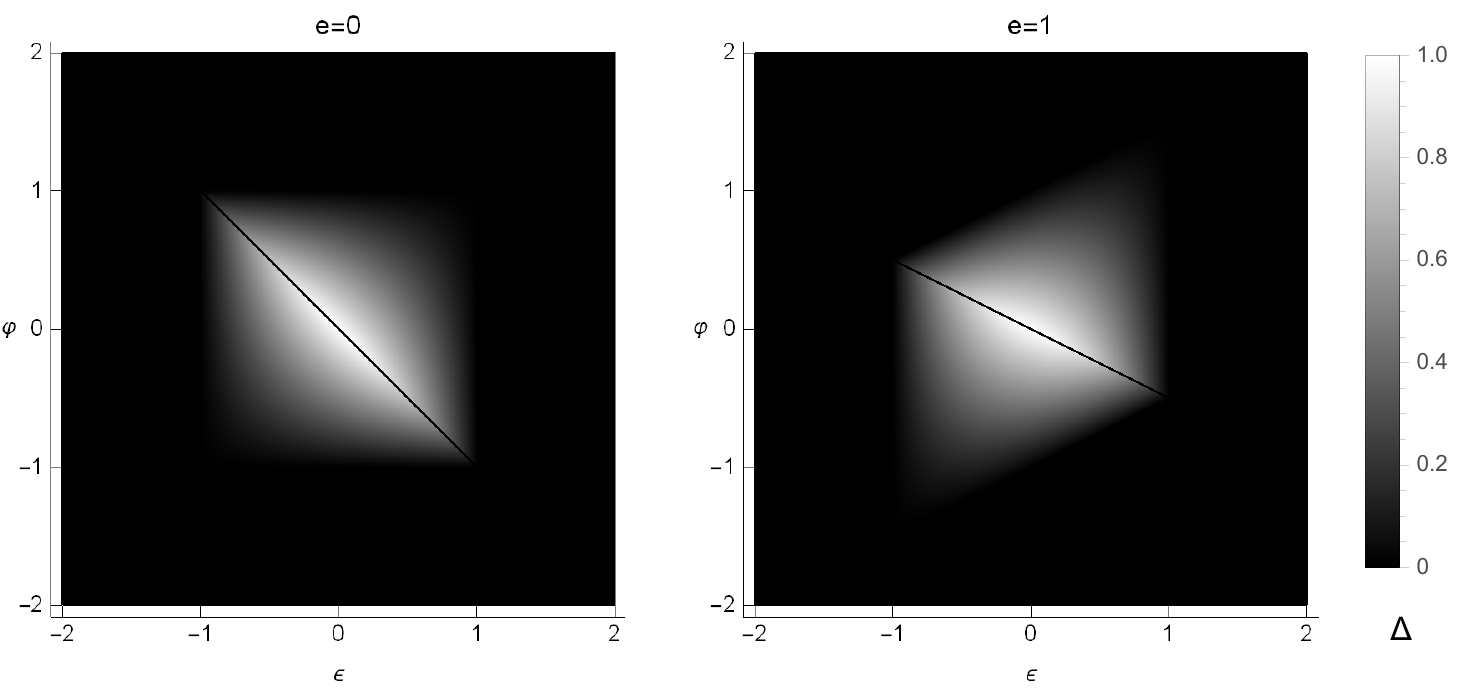}
\end{center}
\caption{For $e=0$ or $1$, we graph the functions $\DLP_{-K_{\F_e},\OO_{\F_e}}(\epsilon E + \varphi F)$ on the square $[-2,2]\times[-2,2]$ in the $(\epsilon,\varphi)$-plane.  See Example \ref{ex-Crescents}.}\label{fig-Crescents}
\end{figure}
\end{example}

\begin{example}\label{ex-CrescentsH}
In Figure \ref{fig-CrescentsH}, we show how the function $\DLP_{H,\cV}$ changes with the polarization.  For $e=0$ and $0\leq t \leq 8$, we graph the function $\DLP_{H_{1+\frac{t}{8}},\OO_{\F_0}}(\epsilon E + \varphi F)$ on the square $[-2,2]\times [-2,2]$.  Similar pictures hold for other exceptional bundles and other Hirzebruch surfaces.
\begin{figure}[t]
\begin{center}
\includegraphics[bb=0 0 5.86in 5.86in]{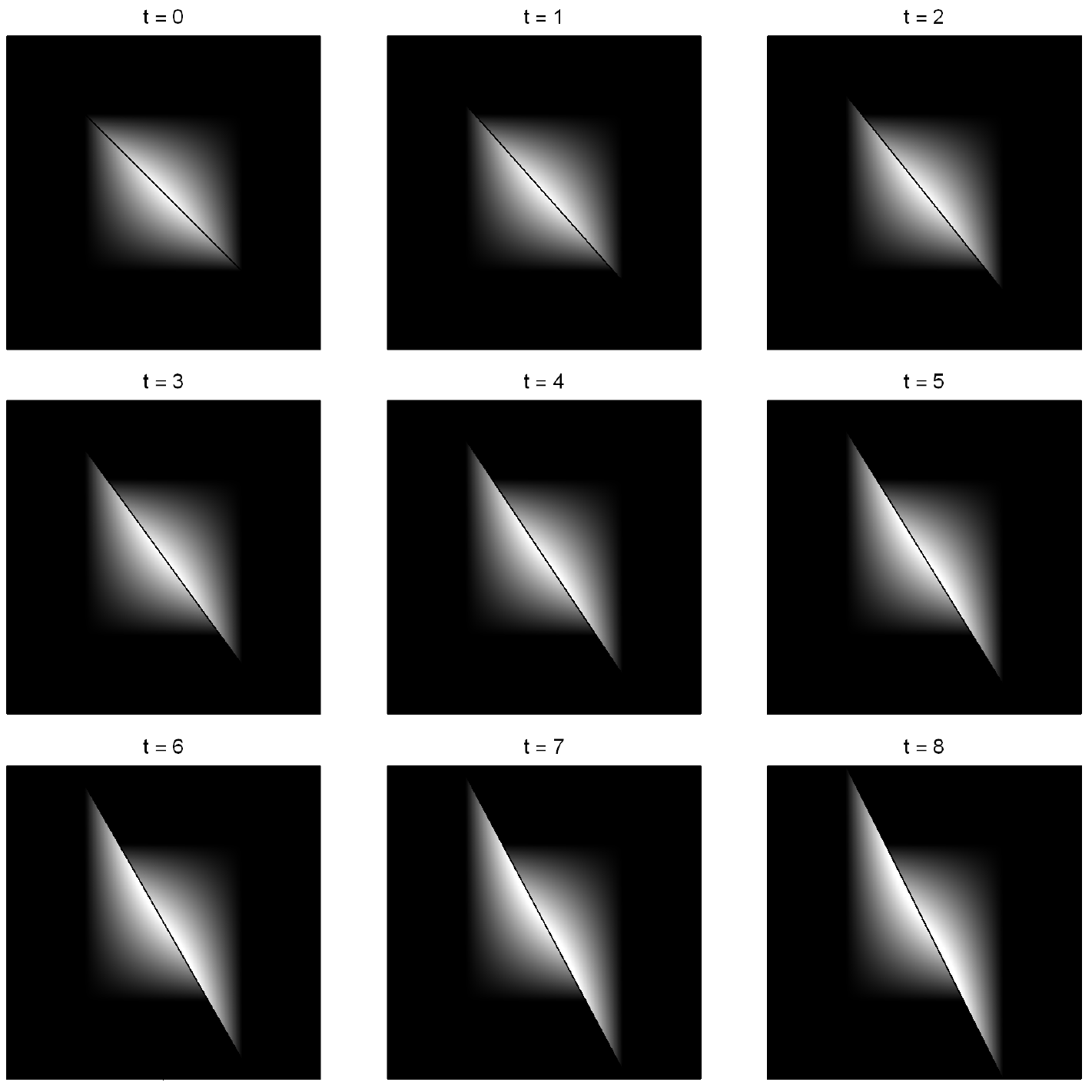}
\end{center}
\caption{For $e=0$ and $0\leq t\leq 8$, we graph the functions $\DLP_{H_{1+\frac{t}{8}},\OO_{\F_e}}(\epsilon E + \varphi F)$ on the square $[-2,2]\times[-2,2]$ in the $(\epsilon,\varphi)$-plane.  See Example \ref{ex-CrescentsH}.}\label{fig-CrescentsH}
\end{figure}
\end{example}
 
Taking the inequalities coming from all the $\mu_H$-stable exceptional bundles $\cV$ produces a Dr\'ezet-Le Potier type surface that restricts the numerical invariants of semistable sheaves.

\begin{definition}
Let $\mathbb{E}_{H}$ be the set of $\mu_H$-stable exceptional bundles on $\F_e$.  Define a function
$$\DLP_{H}(\nu) = \sup_{\substack{\cV\in \mathbb{E}_{H} \\ |(\nu-\nu(\cV))\cdot H|\leq -\frac{1}{2}K_{\F_e}\cdot H}} \DLP_{H,\cV}(\nu).$$
Similarly, it is useful to define analogous functions where the ranks of the exceptional bundles are bounded.  For an integer $r> 1$ we define $$\DLP_{H}^{<r} (\nu) = \sup_{\substack{\cV\in \mathbb{E}_{H} \\ |(\nu-\nu(\cV))\cdot H| \leq -\frac{1}{2}K_{\F_e}\cdot H \\ r(\cV) < r}} \DLP_{H,\cV}(\nu).$$ In the special case $r=2$, we get a function $\DLP_H^{<2}(\nu)$ controlled by line bundles.  Since this function is important in what follows, we write it as $$\DLP^1_H(\nu) = \sup_{L\in \Pic(\F_e) \atop |(\nu-L)\cdot H| < -\frac{1}{2}K_{F_e}\cdot H} \DLP_{H,L}(\nu).$$  Note that there is no question of stability for line bundles.
\end{definition}

\begin{remark}\label{rem-max}
For any exceptional bundle $\cV$, polarization $H$, and $c\in \RR$, the set $$(\DLP_{H,\cV})^{-1}([c,\infty)) = \{ \nu : |(\nu-\nu(\cV))\cdot H| \leq -\frac{1}{2}K_{\F_e} \cdot H \textrm{ and } \DLP_{H,\cV}(\nu) \geq c\}$$ is bounded.  It follows that the supremum in the definition of $\DLP_H^{<r}(\nu)$ is actually a maximum.

On the other hand, at least if irrational slopes $\nu$ are allowed, then the supremum in the definition of $\DLP_H(\nu)$ may not be achieved by any exceptional bundle.  Similar phenomena  occur as in the case of $\P^2$: the analogous function $\delta(\mu)$ on $\P^2$ takes a value of $\frac{1}{2}$ at some transcendental slopes $\mu$, and this value is not computed by any particular exceptional bundle. See \cite[\S 4]{CHW}.  On $\F_0$, Abe \cite{Abe} shows that there are balanced slopes $\mu F_1 +\mu F_2$ with the same property.
\end{remark}

Our discussion in this section shows the following.

\begin{proposition}\label{prop-DLPss}
Let $H$ be generic. 

\begin{enumerate}\item If $\cW$ is an $H$-semistable exceptional bundle on $\F_e$ of rank $r$, then $$\Delta(\cW) \geq \DLP_{H}^{<r}(\nu(\cW)).$$
\item If $\cW$ is an $H$-semistable non-semiexceptional sheaf on $\F_e$, then
$$\Delta(\cW) \geq \DLP_{H}(\nu(\cW)).$$
\end{enumerate}
\end{proposition}

\begin{example}\label{ex-Kgraphs}
To build intuition, we graph some of the functions $\DLP_{-K_{\F_e}}^{<r}(\nu)$, which will be studied heavily in the next section.    These pictures will be better justified after Theorem \ref{thm-stabilityInterval}, which will give us a quick algorithm to determine all the exceptional bundles up to a given rank.

In Figure \ref{fig-F0K}, we take $e=0$ and graph the function $\DLP_{-K_{\F_0}}^{<8}(\nu)$ over the unit square $[0,1]\times [0,1]$ of slopes $\epsilon E + \varphi F = (\epsilon,\varphi)$.  There are contributions to the surface from exceptional bundles of ranks $1,3,5,7$ (see Example \ref{ex-stabilityIntervals} and Table \ref{table-stabilityInterval0}).

In Figure \ref{fig-F1K}, we take $e=1$ and graph the function $\DLP_{-K_{\F_0}}^{<7}(\nu)$ over the unit square $[0,1]\times [0,1]$ of slopes $\epsilon E + \varphi F = (\epsilon,\varphi)$.  There are contributions to the surface from exceptional bundles of ranks $1,2,4,5,6.$ (see Example \ref{ex-stabilityIntervals} and Table \ref{table-stabilityInterval1}).

In each case, both functions take values slightly below $1/2$ on certain small regions where higher rank exceptional bundles have not been included; on the other hand, we will see that when all exceptionals are accounted for we have $\DLP_{-K_{\F_e}}(\nu) \geq \frac{1}{2}$.  See Corollary \ref{cor-K1/2}.

\begin{figure}[t]
\begin{center}
\includegraphics[bb=0 0 5.58in 5.06in]{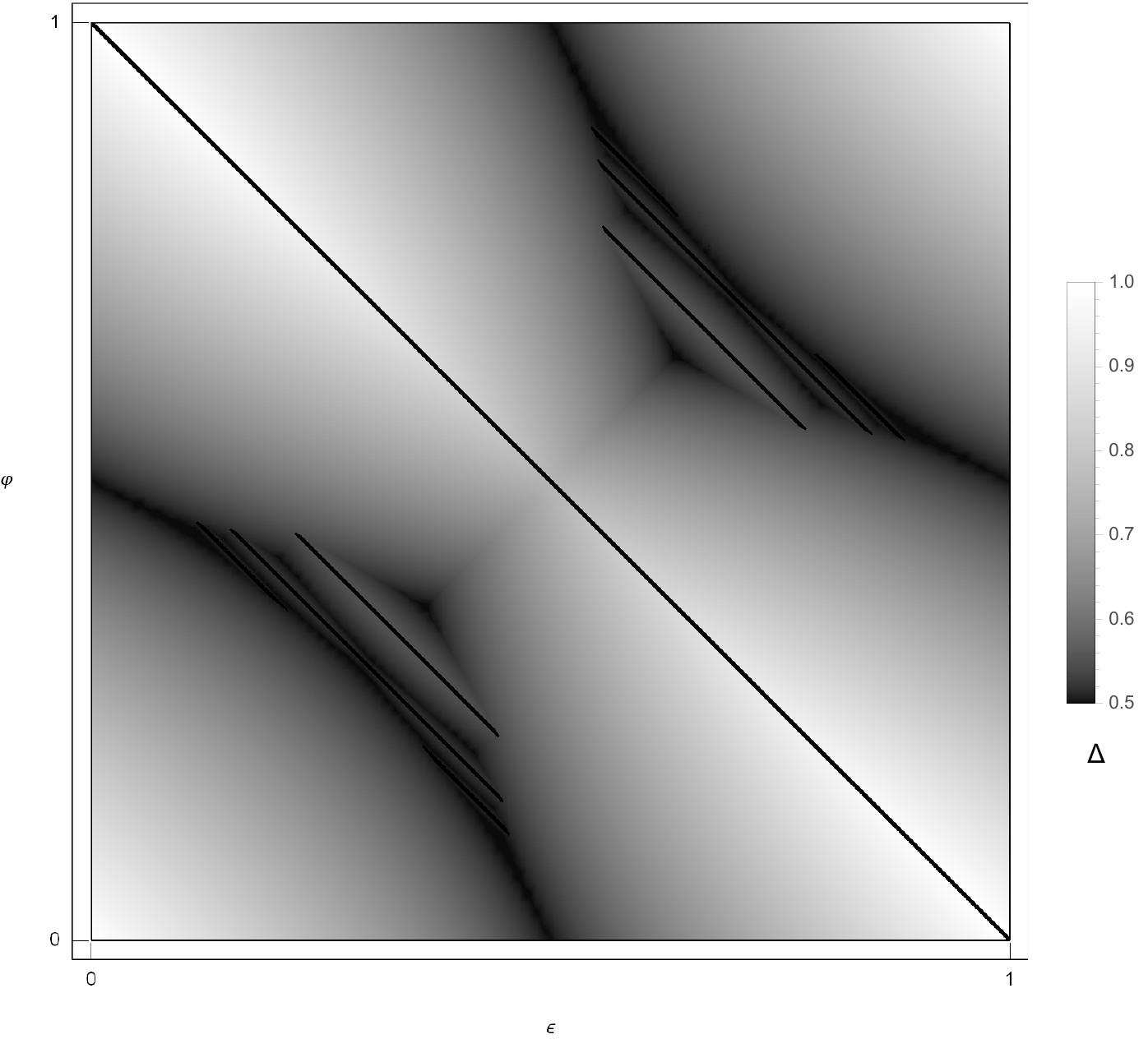}
\end{center}
\caption{For $e=0$, we graph the function $\DLP_{-K_{\F_0}}^{<8}(\epsilon E + \varphi F)$ on the unit square in the $(\epsilon,\varphi)$-plane.  See Example \ref{ex-Kgraphs}.}\label{fig-F0K}
\end{figure}

\begin{figure}[t]
\begin{center}
\includegraphics[bb=0 0 5.58in 5.06in]{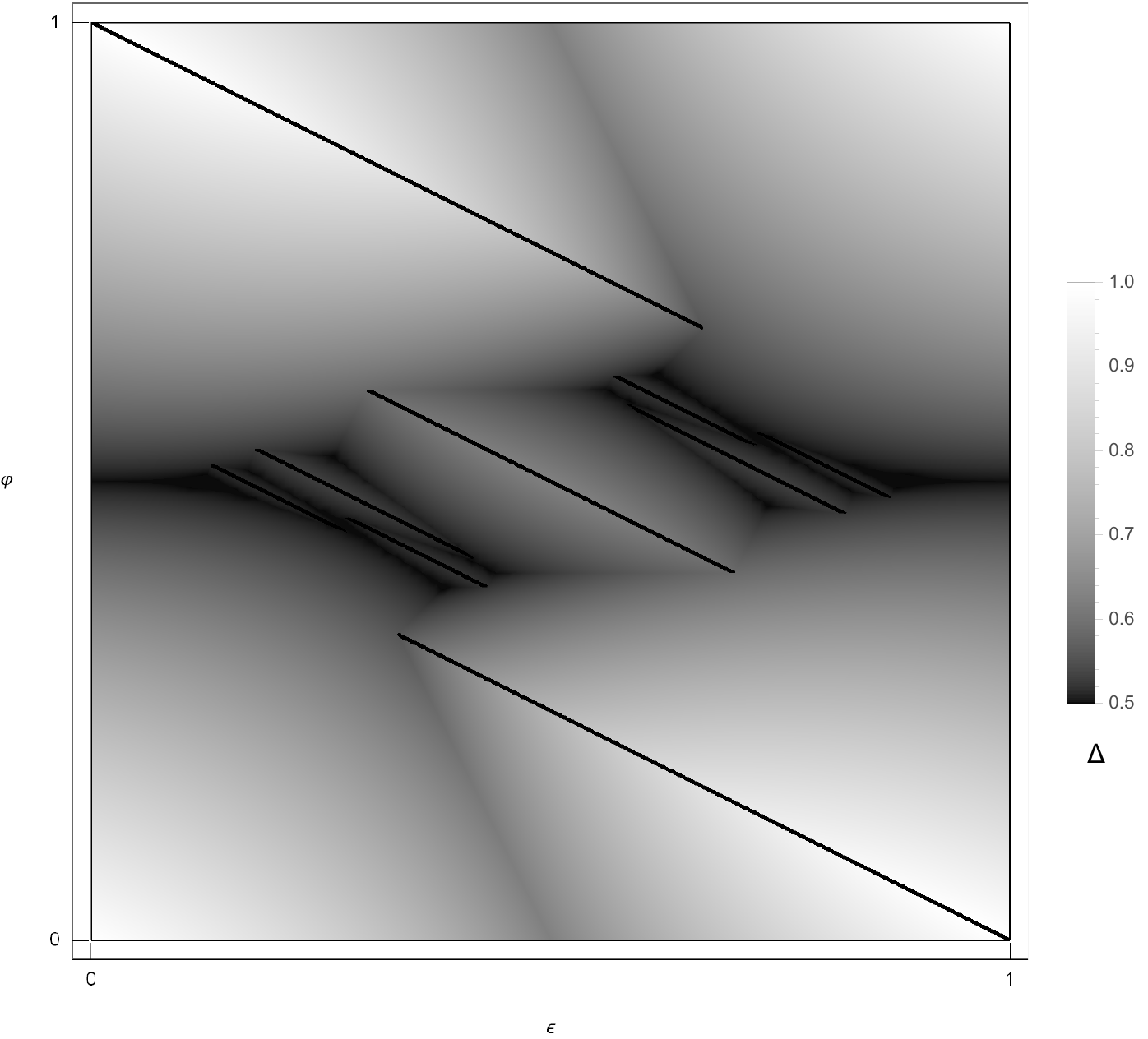}
\end{center}
\caption{For $e=1$, we graph the function $\DLP_{-K_{\F_1}}^{<7}(\epsilon E + \varphi F)$ on the unit square in the $(\epsilon,\varphi)$-plane.  See Example \ref{ex-Kgraphs}.}\label{fig-F1K}
\end{figure}
\end{example}

\section{Sufficient conditions for stability on a del Pezzo Hirzebruch surface}\label{sec-sufficient}

In this section, we study the existence problem for semistable sheaves on an anticanonically polarized del Pezzo Hirzebruch surface $\F_0$ or $\F_1$.  Similar results were previously obtained by Rudakov in \cite{Rudakov} and \cite{Rudakov2}.  

\subsection{Line bundles and $\DLP^{1}_H$ on a del Pezzo surface}  Before proceeding, we need to study the function $\DLP^{1}_H(\nu)$ which uses line bundles to restrict the numerical invariants of semistable sheaves on $\F_0$ and $\F_1$.  We begin by investigating the anticanonical case.

\begin{lemma}\label{lem-38}
Let $e=0$ or $1$, and let $\nu\in \Pic(\F_e)\te \QQ$ be a total slope.  Then $$\DLP_{-K_{\F_e}}^{1} (\nu) \geq \frac{3}{8}.$$
\end{lemma}
\begin{proof}
We may as well assume $\nu = \epsilon E + \varphi F = (\epsilon,\varphi)\in \QQ^2$ lies in the unit square $[0,1]\times [0,1]$.  Then a straightforward computation shows that one of the four line bundles $\OO$, $\OO(E)$, $\OO(F)$, or $\OO(E+F)$ always provides the required inequality.
\end{proof}

\begin{remark}
When $e=1$, in fact $\DLP_{-K_{\F_e}}^{1}(\frac{1}{2}E+\frac{1}{2}F) = \frac{3}{8}$, and there is a rank $2$ exceptional bundle of this slope.  When $e=0$, the inequality can be further improved to $\DLP_{-K_{\F_e}}^{1}(\nu) \geq \frac{4}{9}$, which is achieved at the total slope $\frac{1}{3}F_1 + \frac{1}{3}F_2$, although we won't need this.
\end{remark}

The next result is the line bundle version of a more general result we will prove later.

\begin{lemma}[Monotonicity in the polarization---line bundle case]\label{lem-monotonicity2}
Let $e=0$ or $1$, and let $\nu\in \Pic(\F_e)\te \QQ$ be a total slope.  Consider the polarizations $A_m = -\frac{1}{2}K_{\F_e}+mF$.  If $0 \leq m \leq m'$ or if $\frac{e}{2}-1<m' \leq m \leq 0$, then $$\DLP_{A_m}^{1}(\nu) \leq \DLP_{A_{m'}}^{1}(\nu).$$ Informally, as a function of the polarization $H$, $\DLP_H^{1}(\nu)$ gets larger as (the ray spanned by) $H$ moves away from (the ray spanned by) $-K_{\F_e}$.
\end{lemma}
\begin{proof}
Let $L$ be a line bundle.  Notice that in the definition of $\DLP_{H,L}$ the two formulas $$P(\pm(\nu-L))-\Delta(L)$$ defining the function don't depend on $H$; the polarization $H$ only matters to select the appropriate branch depending on the sign of $(\nu-L)\cdot H$.  In the special case when $H = -K_{\F_e}$, we furthermore have $$\DLP_{-K_{\F_e},L}(\nu) = \min\{ P(\pm(\nu-L))-\Delta(L)\}$$ by the formulas in Remark \ref{rem-DLPK}.  Thus, the smaller of the two branches is always selected when $H = -K_{\F_e}$.  As $H$ moves away from $-K_{\F_e}$, the sign of $(\nu-L)\cdot H$ possibly changes, and the larger branch will be selected to compute $\DLP_{H,L}(\nu)$.  Once the larger branch is selected it will continue to be selected as $H$ moves further away from $-K_{\F_e}$, so long as $\nu$ remains in the domain of definition of $\DLP_{H,L}(\nu)$.  With $m,m'$ as in the statement, it follows that $$\DLP_{A_m,L}(\nu) \leq \DLP_{A_{m'},L}(\nu)$$ if $\nu$ is in the domain of definition of both functions.

As $H$ moves away from $-K_{\F_e}$, it is also possible that $\nu$ either leaves or enters the domain of definition of $\DLP_{H,L}(\nu)$; this transition happens when $|(\nu-L)\cdot H| = -\frac{1}{2}K_{\F_e}\cdot H$.  There are four cases to consider, but they can all be handled similarly.  Suppose $(\nu-L)\cdot H = -\frac{1}{2}K_{\F_e}\cdot H$.  Let $H_-$ be a polarization slightly closer to $-K_{\F_e}$ and let $H_+$ be a polarization slightly farther from $-K_{\F_e}$.  Suppose that 
\begin{align*}(\nu-L)\cdot H_+ &> -\frac{1}{2}K_{\F_e}\cdot H_+ \\ (\nu-L)\cdot H_- &< -\frac{1}{2}K_{\F_e}\cdot H_{-},\end{align*} so that $\nu$ is in the domain of definition of $\DLP_{H_{-},L}$ but not in the domain of definition of $\DLP_{H_+,L}$.  Consider the line bundle $L' = L \te (-K_{\F_e})$.  Then $(\nu-L')\cdot H = \frac{1}{2}K_{\F_e}\cdot H$, and $$(\nu-L')\cdot H_+>\frac{1}{2}K_{\F_e}\cdot H_+,$$ so that $\nu$ is in the domain of definition of $\DLP_{H_+,L'}$.  But the function $P$ satisfies $P(\nu + K_{\F_e}) = P(-\nu)$ by Serre duality, so $$\DLP_{H_+,L'}(\nu) = P(\nu-L') - \Delta(L') = P(L-\nu)-\Delta(L) = \DLP_{H_{-},\cV}(\nu).$$ Thus as $\nu$ leaves the domain of definition of $\DLP_{H_{-},L}$, it enters the domain of definition of $\DLP_{H_+,L'}$, and the formula used to compute each function at $\nu$ is the same.  Therefore, the computation of the maximum in the definition of $\DLP^{1}_H(\nu)$ is unaffected as $\nu$ enters or leaves the domain of definition of a function $\DLP_{H,L}$.  Thus we have shown that every term in the supremum defining $\DLP_{A_m}^{1}(\nu)$ is bounded by some term in the supremum defining $\DLP_{A_{m'}}^{1}(\nu)$.  
\end{proof}

Combining Lemmas \ref{lem-38} and \ref{lem-monotonicity2}, we get the following inequality that will be needed in the next subsection.

\begin{corollary}\label{cor-38}
Let $e=0$ or $1$, and let $\nu\in \Pic(\F_e)\te \QQ$.  For any polarization $H$, we have $$\DLP_H^{1}(\nu) \geq \frac{3}{8}.$$
\end{corollary}

\subsection{Exceptional bundles and stability for del Pezzo surfaces}

In the anticanonically polarized del Pezzo case, Proposition \ref{prop-DLPss} actually gives necessary conditions for the existence of semistable (resp. exceptional) sheaves.  Since the anticanonical polarization can lie on various walls it is more convenient to first work with polarizations arbitrarily close to it.

\begin{theorem}\label{thm-existence}
Let $e=0$ or $1$ and let $\bv = (r,\nu,\Delta)\in K(\F_e)$ have positive rank.  Let $\epsilon >0$ be sufficiently small (depending on $r$) and let $H:=H_{\pm} = -\frac{1}{2}K_{\F_e}\pm \epsilon F$.  If $$\Delta \geq \DLP_H^{<r}(\nu)$$ then there is an $H$-semistable sheaf of character $\bv$.
\end{theorem}
\begin{proof}
The result is clear if $r = 1$, so assume $r\geq 2$.  If $e=0$, then by symmetry we only have to consider the polarization $H=H_-$.  By Corollary \ref{cor-38} we have $\DLP_{H}^{1}(\nu) \geq \frac{3}{8},$ so we may as well assume $\Delta\geq \frac{3}{8}$.  Then there are $H_1$-prioritary sheaves of character $\bv$, since in the notation of Section \ref{sec-existPrioritary} if $\nu = \alpha E + \beta F$ then $$\rho_{\gen}({\bf v}) = \left\lfloor\frac{\Delta}{(\lceil\alpha\rceil-\alpha)(\alpha-\lfloor \alpha\rfloor)}-\frac{e}{2}+1 - (\lceil \psi\rceil -\psi)\right\rfloor\geq \left\lfloor \frac{\frac{3}{8}}{\frac{1}{4}} -\frac{1}{2}+1-1\right\rfloor=1$$ by Corollary \ref{cor-prioritaryRho}.  Since $$H=-\frac{1}{2}K_{\F_e} \pm \epsilon F =  H_{1-\frac{e}{2}\pm \epsilon}$$ and $\lceil 1-\frac{e}{2}\pm \epsilon\rceil =1$ (recall we are taking the $-$ sign if $e=0$) we can use the methods of Section \ref{sec-genHN} to study the $H$-Harder-Narasimhan filtration of a general sheaf $\cV\in \cP_{F}(\bv)$.  

Suppose that there is no $H$-semistable sheaf of character $\bv$, so that the $H$-Harder-Narasimhan filtration of $\cV$ has length at least two.  Let $\bgr_1,\ldots,\bgr_\ell$ be the characters of the factors, with total slopes $\nu_1,\ldots,\nu_\ell$ and discriminants $\Delta_1,\ldots,\Delta_\ell$.
 Then in particular $\chi(\bgr_1,\bgr_\ell) = 0$, and so by Riemann-Roch $$P(\nu_\ell - \nu_1) = \Delta_1+\Delta_\ell.$$
Here we have $-1\leq H\cdot(\nu_\ell - \nu_1) \leq 0.$ We want to show that either $\Delta_1$ or $\Delta_\ell$ is less than $\frac{1}{2}$, so that the corresponding factor is semiexceptional by Lemma \ref{lem-excFacts} (6).  To do this, we can show that $P(\nu_\ell-\nu_1)\leq 1$ with equality if and only if $\nu_1 = \nu_\ell$; when equality holds, the fact that the filtration is a Harder-Narasimhan filtration forces $\Delta_1 < \Delta_\ell$.

We want to choose $\epsilon>0$ small enough that the Harder-Narasimhan factors of $\cV$ remain constant as $\epsilon$ becomes arbitrarily small.  To do this, note that if $\epsilon$ is bounded to an interval (say $(0,1))$, then by Lemma \ref{lem-slopeQuad} the total slopes $\nu_i$ will come from a fixed bounded region that does not depend on $\epsilon$.  Since the ranks are bounded, there are then only finitely many possibilities for the slopes $\nu_i$.   Thus there are only finitely many values $m$ such that two distinct total slopes $\nu_i$ have the same $H_m$-slopes.  If $\epsilon$ is chosen small enough, it follows that the ordering of the numbers $H_m\cdot \nu_i$ does not change as $\epsilon$ shrinks to $0$.  By an induction on the rank, if we further shrink $\epsilon$, then for a polarization $H_m \in [H,-\frac{1}{2}K_{\F_e})$ the set of $H_m$-semistable sheaves of rank less than $r$ will not depend on $m$.  Hence, for this choice of $\epsilon$ the $H_m$-Harder-Narasimhan filtration of $\cV$ is independent of the polarization $H_m\in [H,-\frac{1}{2}K_{\F_e})$.

The zero locus of the equation $P(\nu)=1$ in the $\nu$-plane is a hyperbola; the inequality $P(\nu)<1$ holds between the branches of the hyperbola.  The tangent line to $P(\nu)=1$ at $\nu = 0$ (corresponding to $\OO_{\F_e}$) is given by the equation $\nu\cdot (-K_{\F_e}) = 0$, and the tangent line to $P(\nu)=1$ at $\nu = K_{\F_e}$ is the parallel line $\nu\cdot (-K_{\F_e} )= -K_{\F_e}^2$.  Then if $\nu$ satisfies $$ -K_{\F_e}^2\leq\nu\cdot (-K_{\F_e}) \leq 0, $$ it follows that $P(\nu) \leq 1$, with equality if and only if $\nu$ is either $0$ or $K_{\F_e}$.

 By our choice of $\epsilon$, the inequalities $-1\leq H\cdot (\nu_\ell-\nu_1) \leq 0$ give $$-1\mp\epsilon F\cdot (\nu_\ell-\nu_1) \leq -\frac{1}{2}K_{\F_e} \cdot (\nu_\ell-\nu_1)\leq 0;$$  notice that if $H\cdot (\nu_\ell-\nu_1)=0$, then $\nu_\ell=\nu_1$.  Since the $H$-Harder-Narasimhan filtration of $\cV$ does not change as $\epsilon$ further shrinks, we can conclude that $$-\frac{1}{2}K_{\F_e}^2<-1\leq -\frac{1}{2}K_{\F_e}\cdot(\nu_\ell-\nu_1) \leq 0.$$ Then the previous paragraph shows $P(\nu_\ell-\nu_1) \leq 1$, with equality if and only if $\nu_1=\nu_\ell$.

Thus we have shown that either $\bgr_1$ or $\bgr_\ell$ is semiexceptional.  Without loss of generality, suppose $\bgr_1$ is semiexceptional.  Then $$\chi(\bgr_1,\bv) = \sum_{i=1}^\ell \chi(\bgr_1,\bgr_i) = \chi(\bgr_1,\bgr_1)>0.$$ We conclude that $$\Delta< \DLP_{H,\bgr_1}(\nu),$$ so  $\Delta<\DLP_H^{<r}(\nu)$ holds.
\end{proof}

Combining the theorem with Proposition \ref{prop-DLPss} gives the next more complete picture.

\begin{corollary}\label{cor-DLPexist}
Let $e=0$ or $1$ and let $\bv = (r,\nu,\Delta)\in K(\F_e)$ have positive rank.  Let $\epsilon >0$ be sufficiently small (depending on $r$) and let $H:=H_{\pm} = -\frac{1}{2}K_{\F_e}\pm \epsilon F$.  
\begin{enumerate}

\item If $\bv$ is potentially exceptional, then it is exceptional if and only if $$\Delta \geq \DLP_{H}^{<r}(\nu).$$

\item If $\bv$ is not semiexceptional, there is an $H$-semistable sheaf of character $\bv$ if and only if $$\Delta \geq \DLP_{H}(\nu).$$
\end{enumerate}
\end{corollary}
\begin{proof}
(1 \& 2) ($\Rightarrow$) This is Proposition \ref{prop-DLPss} (1) and (2).

(2) ($\Leftarrow$) By assumption $\Delta  \geq \DLP_H(\nu) \geq \DLP_H^{<r}(\nu)$, so there are $H$-semistable sheaves of character $\bv$ by Theorem \ref{thm-existence}.

(1) ($\Leftarrow$) By Theorem \ref{thm-existence} there are $H$-semistable sheaves of character $\bv$, and Lemma \ref{lem-excFacts} (6) completes the proof.
\end{proof} 

Analyzing the proof of Theorem \ref{thm-existence} gives the following fact about the generic Harder-Narismhan filtration in the anticanonically polarized case.

\begin{corollary}\label{cor-HNShape}
Let $e=0$ or $1$ and let $\bv = (r,\nu,\Delta)\in K(\F_e)$ have positive rank.  Let $\epsilon >0$ be sufficiently small (depending on $r$) and let $H:=H_{\pm} = -\frac{1}{2}K_{\F_e}\pm \epsilon F$.  
If $\Delta \geq \frac{3}{8}$ and there are not $H$-semistable sheaves of character $\bv$, then at most one of the $H$-Harder-Narasimhan factors of the general sheaf $\cV\in \cP_F(\bv)$ is not a semiexceptional bundle.
\end{corollary}
\begin{proof}
The assumption $\Delta\geq \frac{3}{8}$ allows us to carry out the argument in the proof of Theorem \ref{thm-existence}, but instead of analyzing the factors $\bgr_1$ and $\bgr_\ell$ we can analyze $\bgr_i$ and $\bgr_j$ with $i<j$.  It follows that at least one of $\bgr_i$ and $\bgr_j$ is semiexceptional, so at most one factor $\bgr_i$ is not semiexceptional.  
\end{proof}

The next technical lemma allows us to deduce results about the polarization $-K_{\F_e}$, instead of polarizations arbitrarily close to it.

\begin{lemma}\label{lem-limitK}
Let $e=0$ or $1$, let $r\geq 2$, and let $\nu\in \Pic(\F_e) \te \QQ$.  If $\epsilon > 0$ is sufficiently small (depending on $r$) and $H := H_{\pm} = -\frac{1}{2}K_{\F_e} \pm \epsilon F$ then 
$$\DLP_{-K_{\F_e}}^{<r} (\nu) = \DLP_{H}^{<r}(\nu).$$
\end{lemma}
\begin{proof}
 If $\epsilon$ is sufficiently small, then every exceptional bundle of rank less than $r$ is $\mu_H$-stable, so the exceptionals used to compute $\DLP_{-K_{\F_e}}^{<r}$ and $\DLP_{H}^{<r}$ are the same.  By Remark \ref{rem-max} there are only finitely many exceptionals which are relevant to the computation.  Let $\cW$ be any exceptional used to compute either function.  Then if $\epsilon$ is sufficiently small we actually have $$\DLP_{-K_{\F_e},\cW}(\nu) = \DLP_{H,\cW}(\nu).$$
Indeed, either 
\begin{enumerate}
\item potentially shrinking $\epsilon$ further, $-K_{\F_e}\cdot (\nu - \nu(\cW))$ and $H\cdot (\nu-\nu(\cW))$ have the same sign, in which case both sides are computed in the same way, or \item $-K_{\F_e}\cdot (\nu - \nu(\cW)) = 0$, in which case both the expressions $P(\nu - \nu(\cW)) - \Delta(\cW)$ and $P(\nu(\cW)-\nu) - \Delta(\cW)$ which could compute $\DLP_{-K_{\F_e},\cW}(\nu)$ and $\DLP_{H,\cW}(\nu)$ are equal.  (See Remark \ref{rem-DLPK}.) 
\end{enumerate}
Therefore the supremums over the relevant exceptional bundles are equal.
\end{proof}

We deduce the following cleaner statement for exceptional bundles.  This allows us to quickly compute the exceptional bundles by induction on the rank.

\begin{corollary}\label{cor-delPezzoExceptional}
If $e=0$ or $1$ and $\bv=(r,\nu,\Delta)\in K(\F_e)$ is potentially exceptional, then it is exceptional if and only if $$\Delta \geq \DLP_{-K_{\F_e}}^{<r}(\nu).$$
\end{corollary}
\begin{proof}
($\Rightarrow$) Let $\cV$ be an exceptional bundle of character $\bv$, and let $\cW$ be an exceptional bundle which computes $\DLP_{-K_{\F_e}}^{<r}(\nu)$.   Without loss of generality, suppose $$\mu_{-K_{\F_e}}(\cW)-\frac{1}{2} K_{\F_e}^2\leq \mu_{-K_{\F_e}}(\cV) \leq \mu_{-K_{\F_e}}(\cW).$$  Since $\cV$ is $\mu_{-K_{\F_e}}$-slope stable and $r(\cW) < r(\cV)$, we have $\Hom(\cW,\cV)=\Ext^2(\cW,\cV)=0$, so $\chi(\cW,\cV)\leq 0$.  It follows that $\Delta\geq \DLP_{-K_{\F_e},\cW}(\nu) = \DLP_{-K_{\F_e}}^{<r}(\nu)$.  

($\Leftarrow$) Let $\epsilon>0$ be small and let $H = -\frac{1}{2}K_{\F_e}-\epsilon F$.  By Lemma \ref{lem-limitK}, $$\Delta \geq \DLP_{-K_{\F_e}}^{<r}(\nu) =\DLP_{H}^{<r}(\nu),$$ and $\bv$ is exceptional by Corollary \ref{cor-DLPexist} (1).
\end{proof}

Analogously there is a statement for semistability which only refers to the anticanonical polarization; however, since sheaves with different total slopes can have the same $-K_{\F_e}$-slope, it is not a complete classification. 

\begin{corollary}\label{cor-delPezzoKss}
Let $e=0$ or $1$ and let $\bv =(r,\nu,\Delta)\in K(\F_e)$ have positive rank.  If $$\Delta \geq \DLP^{<r}_{-K_{\F_e}}(\nu),$$ then there are $-K_{\F_e}$-semistable sheaves of character $\bv$.
\end{corollary}
\begin{proof}
The proof is similar to the proof of Corollary \ref{cor-delPezzoExceptional}.  If $\epsilon>0$ is sufficiently small and $H_{\pm} = -\frac{1}{2}K_{\F_e} \pm \epsilon F$, then we have $\DLP_{-K_{\F_e}}^{<r}(\nu) = \DLP_{H_{\pm}}^{<r}(\nu)$ and there are $H_{\pm}$-semistable sheaves of character $\bv$.  By the irreducibility of the stack of prioritary sheaves, there are sheaves of character $\bv$ which are simultaneously both $H_+$- and $H_-$-semistable, and they will also be $-K_{\F_e}$-semistable.
\end{proof}

\begin{example}
On $\F_0 = \P^1\times \P^1$ with $\Pic(\F_0) = \Z F_1\oplus \Z F_2$, the bundle $\OO(F_1) \oplus \OO(F_2)$ is $-K_{\F_0}$-semistable of discriminant $\frac{1}{4}$.  Its invariants satisfy $\Delta < \DLP_{-K_{\F_e}}(\nu)$, even though it is not semiexceptional.
\end{example}

We can now also prove the following qualitative fact about the Dr\'ezet-Le Potier surface.

\begin{corollary}\label{cor-K1/2}
Let $e=0$ or $1$ and let $\nu\in \Pic(\F_e)\te \QQ$.  Then $$\DLP_{-K_{\F_e}}(\nu) \geq \frac{1}{2}.$$  
\end{corollary}
\begin{proof}
To get a contradiction, let $\bv = (r,\nu,\Delta)\in K(\F_e)$ be a Chern character of smallest rank such that $$\DLP_{-K_{\F_e}}(\nu) \leq \Delta < \frac{1}{2}.$$ Then $\DLP_{-K_{\F_e}}^{<r} (\nu) \leq \Delta$.  By Lemma \ref{lem-limitK}, if $\epsilon>0$ is sufficiently small and $H = -\frac{1}{2}K_{\F_e}-\epsilon F$, then $\DLP_{-K_{\F_e}}^{<r}(\nu) = \DLP_H^{<r} (\nu)$, so $\DLP_H^{<r}(\nu) \leq \Delta$.  Theorem \ref{thm-existence} then gives an $H$-semistable sheaf $\cV$ of character $\bv$.  The character $\bv$ is primitive by our minimality assumption, and since $H$ is generic $\cV$ is $H$-stable. By Lemma \ref{lem-excFacts} (4), $\cV$ is exceptional. Since $\cV$ is $\mu_{-K_{\F_e}}$-stable by Theorem \ref{thm-delPezzoExc}, we find $$\DLP_{-K_{\F_e}}(\nu) \geq \DLP_{-K_{\F_e},\cV}(\nu) = \frac{1}{2}+\frac{1}{2r^2} > \frac{1}{2},$$ contradicting our assumption.
\end{proof}

\section{Stability intervals and the stability of exceptional bundles}\label{sec-exceptional2}

For a sheaf $\cV$, we can study the collection of polarizations such that $\cV$ is stable.  In this section, we both study this question in general and discuss how to compute this set for exceptional bundles on del Pezzo Hirzebruch surfaces.

\subsection{Stability intervals in general} First we make a general definition for an arbitrary Hirzebruch surface $\F_e$.

\begin{definition}
Let $\cV$ be a coherent sheaf on $\FF_e$.  The \emph{stability interval} of $\cV$ is $$I_\cV = \{m> 0: \cV\textrm{ is $\mu_{H_m}$-stable}\}\subset \RR_{>0}.$$
\end{definition}

\begin{proposition}
The stability interval $I_\cV$ is either empty or it is an open interval $(m_0,m_1)$, possibly with $m_0 = 0$ or $m_1=\infty$.  If $m_0>0$, then $\cV$ is strictly $\mu_{H_{m_0}}$-semistable, and if $m_1<\infty$, then $\cV$ is strictly $\mu_{H_{m_1}}$-semistable.
\end{proposition}
\begin{proof}
Suppose $I_\cV$ is nonempty.  Slope stability is open in the polarization, so if $\cV$ is $\mu_H$-stable then it is also $\mu_{H'}$-stable for $H'$ sufficiently close to $H$.  If $H,H'$ are two polarizations such that $\cV$ is both $\mu_H$- and $\mu_{H'}$-stable, then $\cV$ is $\mu_{H''}$-stable for any convex combination $H''$ of $H$ and $H'$.  Therefore, $I_\cV$ is a nonempty open interval.

Without loss of generality, suppose $m_1<\infty$; then $\cV$ is not $\mu_{H_{m_1}}$-stable.  On the other hand, if it is not $\mu_{H_{m_1}}$-semistable, then a destabilizing subsheaf would show that $\cV$ is not $\mu_{H_{m_1-\epsilon}}$-stable for $\epsilon>0$ small.  Therefore $\cV$ is strictly $\mu_{H_{m_1}}$-semistable.
\end{proof}

\begin{remark}
We could analogously define intervals using other notions of stability, e.g. slope-semistabity or Gieseker (semi)-stability.  If $\cV$ is $\mu_H$-stable for some polarization $H$, then the closure of $I_{\cV}$ is the interval given by slope-semistability, and the other intervals potentially differ only at the endpoints.
\end{remark}

Also of interest is the stability interval of the general sheaf.  The next result shows that this notion makes sense.

\begin{proposition}
Suppose there is a $\mu_H$-stable sheaf of character $\bv$ for some polarization $H$. There is an open dense substack of $\cP_F(\bv)$ parameterizing sheaves $\cV$ such that the stability interval $I_{\bv}:=I_\cV$ is as large as possible.
\end{proposition}

We call $I_\bv$ the \emph{generic stability interval}.  

\begin{proof}
Let $I_{\bv} = (m_0,m_1)$ be the union of the stability intervals $I_\cV$ of all $F$-prioritary sheaves $\cV$ of character $\bv$.  Since $\cP_F(\bv)$ is irreducible and slope-stability is open in the polarization, $I_{\bv}$ is a nonempty open interval.  If $m_1$ is finite, then for $m<m_1$ sufficiently close to $m_1$ the notion of $\mu_{H_{m}}$-stability for sheaves of character $\bv$ is independent of $m$, since the walls for stability are locally finite near $H_{m_1}$.  If $m_1= \infty$, then by Yoshioka \cite[Lemma 1.2]{YoshiokaRuled} the notions of $\mu_{H_m}$-stability of sheaves of character $\bv$ stabilize for large enough $m$.  In either case, there is an open substack of $\cP_F(\bv)$ parameterizing sheaves which are $\mu_{H_m}$-stable for all $m<m_1$ sufficiently close to $m_1$.

If $e=0$, then by symmetry the notion of $\mu_{H_m}$-stability stabilizes for sheaves of character $\bv$ as $m>m_0$ approaches $m_0$.  If $e \geq 1$ then $H_{m_0}$ is a positive class in the sense of Huybrechts and Lehn \cite[\S 4.C]{HuybrechtsLehn}, and again the walls for stability are locally finite near $H_{m_0}$ (even when $m_0 = 0$).  So again, the notion of $\mu_{H_m}$-stability stabilizes for $m > m_0$ close to $m_0$.

If $m_0'>m_0$ and $m_1'<m_1$ are sufficiently close to $m_0$ and $m_1$, then the open substack of simultaneously $\mu_{H_{m_0'}}$- and $\mu_{H_{m_1'}}$-stable sheaves is the substack parameterizing sheaves $\cV$ with $I_\cV = I_{\bv}$.\end{proof}

\subsection{Stability intervals of exceptionals on $\F_0$ and $\F_1$} Let $e=0$ or $1$ and let $\cV$ be an exceptional bundle on $\F_e$ of rank at least $2$.  Then $\cV$ is $\mu_{-K_{\F_e}}$-stable, but as we vary the polarization it will eventually cease to be stable.  Intuitively, Proposition \ref{prop-mukai} says that the destabilizing subsheaf will have to be an exceptional sheaf of lower rank.  This makes it possible to compute the stability interval $I_\cV$.

\begin{remark}\label{rem-notStableForever}
Let $\cV$ be an exceptional bundle of rank at least $2$, and let $\nu(\cV) = \epsilon E+\varphi F$.  By Lemma \ref{lem-excFacts} (2), $\epsilon$ is not an integer.  The prioritary index $\rho(\cV)$ is then finite by Corollary \ref{cor-prioritaryRho}, and $\cV$ is not $\mu_{H_m}$-semistable if $m$ is sufficiently large. 

Note that $\varphi$ is also not an integer.  Indeed, a simple computation as in Lemma \ref{lem-38} shows that if $\varphi$ is an integer then $\DLP_{-K_{\F_e}}^1(\nu)\geq \frac{1}{2}$.  This contradicts Corollary \ref{cor-delPezzoExceptional}.
\end{remark}

The remark proves the following.

\begin{proposition}\label{prop-interval} Let $e=0$ or $1$ and let $\cV$ be an exceptional bundle of rank at least $2$ on $\F_e$.  The stability interval $I_\cV$ is an open interval $(m_0,m_1)$ that contains $1-\frac{e}{2}$ (corresponding to $-K_{\F_e}$), and $m_1<\infty$.  \qed
\end{proposition}

\begin{remark}
If $e=0$ and $\cV$ is an exceptional bundle on $\F_0$ of rank at least $2$, then by symmetry the stability interval $I_\cV = (m_0,m_1)$ has $m_0>0$.  If $c_1(\cV) = aF_1 + aF_2$ is symmetric, then furthermore $I_\cV$ is of the form $(1/m,m)$, but in general this need not be the case.  See Example \ref{ex-stabilityIntervals} for a rank $5$ example.

On $\F_1$, there are many exceptional bundles with a stability interval of the form $(0,m_1)$.  If $\pi : \F_1\to \P^2$ is the blowdown map, $\pi^*T_{\P^2}$ gives such an example.
\end{remark}

If $\cV$ is an exceptional bundle of rank $r$ and the intervals $I_\cW$ have already been computed for all exceptional bundles $\cW$ of rank less than $r$, then we can compute the interval $I_\cV$.

\begin{theorem}\label{thm-stabilityInterval}
Let $e=0$ or $1$ and let $\cV$ be an exceptional bundle of rank $r\geq 2$ on $\F_e$.  For an exceptional bundle $\cW\neq \cV$, let $m_{\cV,\cW}\in \QQ_{>0}$ be the number $m>0$ such that $\cV$ and $\cW$ have the same $H_m$-slope, if it exists.  Let $$S_\cV = \left\{m_{\cV,\cW}: {{\textrm{$\cW$ is an exceptional bundle with $r(\cW) < r$,}} \atop {\textrm{$\chi(\cW,\cV)>0$, and $m_{\cV,\cW}\in I_\cW$}}}\right\}\subset \QQ_{>0}.$$ Then $I_\cV$ is the connected component of $\RR_{>0}\sm S_\cV$ that contains $1-\frac{e}{2}$.
\end{theorem}
\begin{proof}
First let us show that $I_\cV \cap S_\cV = \emptyset$.  Suppose $m:=m_{\cV,\cW}\in S_\cV$, where $\cW$ is an exceptional bundle that proves $m\in S_{\cV}$.  Let us show that $\cV$ is not $\mu_{H_m}$-stable; suppose instead that $\cV$ is $\mu_{H_m}$-stable.  We have $\mu_{H_m}(\cV) = \mu_{H_m}(\cW)$, and $\cW$ is $\mu_{H_m}$-stable since $m\in I_\cW$.  Then $\Ext^2(\cW,\cV) = 0$ by stability, and since $\chi(\cW,\cV) > 0$ we have $\Hom(\cW,\cV)\neq 0$.  Stability forces $\cV\cong \cW$, which contradicts $r(\cW)<r(\cV)$.

Therefore, $I_\cV$ is contained in $\RR_{>0} \sm S_\cV$.   On the other hand, say $I_\cV = (m_0,m_1)$.  By symmetry, it is enough to show that $m_1\in S_\cV$.  Let $\epsilon > 0$ be small.  Then $\cV$ is $\mu_{H_{m_1}}$-semistable but not $H_{m_1+\epsilon}$-semistable; let $\cF \subset \cV$ be a maximal destabilizing subsheaf for $H_{m_1+\epsilon}$-semistability, and assume $\epsilon$ was chosen small enough that $\cF$ does not depend on $\epsilon$.  Note that since $\cV$ is $\mu_{H_{m_1}}$-semistable, it is $H_{\lceil m_1\rceil +1}$-prioritary and therefore it is $H_{\lceil m_1+\epsilon \rceil}$-prioritary.  Since $\cV$ is rigid, the $H_{m_1+\epsilon}$-Harder-Narasimhan filtration of $\cV$ is the Harder-Narasimhan filtration of a general $H_{\lceil m_1+\epsilon\rceil}$-prioritary sheaf of character $\ch \cV$.  If we consider the exact sequence $$0\to \cF\to \cV\to \cQ\to 0,$$ then we have $\Hom(\cF,\cQ) = 0$ (since $\cF$ is the maximal destabilizing subsheaf) and $\Ext^2(\cQ,\cF) = 0$ by Lemma \ref{lem-HNclose}.  From Proposition \ref{prop-mukai} we see that $\cF$ is rigid, and since it is semistable for a generic polarization $H_{m_1+\epsilon}$ it must be semiexceptional by Theorem \ref{thm-rigidSplit}.  Say $\cF = \cW^{\oplus k}$ for an exceptional bundle $\cW$.  Then $\chi(\cF,\cQ) = 0$ by Lemma \ref{lem-HNorthogonal}, so $\chi(\cW,\cQ) = 0$ and $\chi(\cW,\cV) = k >0$.  By our choice of $\epsilon$, we have $\mu_{H_{m_1}}(\cW) = \mu_{H_{m_1}}(\cV)$, so $m_{\cV,\cW} = m_1$.  Finally, since $\cW$ is $H_{m_1+\epsilon}$-stable  it is $\mu_{H_{m_1}}$-stable by Proposition \ref{prop-interval}.  Therefore $m_1\in I_\cW$.  We conclude that $m_1\in S_\cV$.
\end{proof}

\begin{remark}\label{rem-stabilityIntervalCompute}
The set $S_\cV$ can be efficiently computed near $1-\frac{e}{2}$.  Suppose $\cW$ is an exceptional bundle that shows that $m=m_{\cV,\cW}\in S_\cV$, and let $$\nu = \nu(\cV)-\nu(\cW) = a E+b F = (a,b)\in \QQ^2.$$ Then $\nu\cdot H_m=am+b=0$, so $m = -\frac{b}{a}$.  Geometrically, in order to have $m>0$, the slope of the line between $(0,0)$ and $(a,b)$ must be negative.  Equivalently, $a$ and $b$ must have opposite signs.  

Since $\chi(\cW,\cV)>0$ and $\Delta(\cV)>0$, we must have $P(\nu) >0$ by Riemann-Roch.  We have $$P(x E+y F) = (x+1)\left(y+1-\frac{1}{2}ex\right).$$ The two lines $\ell_1:x= -1$ and $\ell_2:y=-1+\frac{e}{2}x$ meet at $(-1,-1-\frac{e}{2})$ and divide the $(x,y)$-plane into four regions.  The function $P$ is positive on the region $R_1$ right of $\ell_1$ and above $\ell_2$ and the region $R_2$ left of $\ell_1$ and below $\ell_2$.  All the points in $R_2$ have both of their coordinates negative, so we must have $\nu\in R_1$ since $m>0$.  Furthermore, the region in $R_1$ of points with coordinates with opposite signs is contained in the union of two unit width rectangular strips $$R_3=((-1,0)\times (0,\infty)) \cup ((0,\infty)\times (-1,0)).$$ If $e=0$, then $R_3 \cap R_1  =R_3$.  If $e=1$, then $R_3\cap R_1$ contains one unbounded strip $((-1,0)\times (0,\infty))$ and a bounded triangular region with vertices $(0,-1)$, $(0,0)$, and $(2,0)$.  

Since the coordinates of $\nu(\cV)$ are not integers (see Remark \ref{rem-notStableForever}), line bundles $L$ give infinitely many points $\nu(\cV)-\nu(L)$ in each unbounded component of $R_3\cap R_1$.  Riemann-Roch shows that any of these line bundles with $c_1$ sufficiently far away from $\nu(\cV)$ will compute an element $m_{\cV,L}\in S_\cV$.  When $e=0$, this computes both an element $M_1\in S_\cV$ larger than $1$ (coming from the vertical strip) and an element $M_0\in S_\cV$ smaller than $1$ (coming from the horizontal strip).  When $e=1$, this computes an element $M_1\in S_\cV$ larger than $1/2$ (coming from the vertical strip).  We let $M_0=0$ when $e=1$ to streamline the exposition.

Now to compute $S_\cV$ near $1-\frac{e}{2}$ we only need to consider exceptional bundles $\cW$ such that $r(\cW)<r$, $\nu(\cV)-\nu(\cW)\in R_3\cap R_1$, and $m_{\cV,\cW}\in (M_0,M_1)$.  But, the region $R_4$ of points $\nu$ in $R_3\cap R_1$ such that the slope of the line through $\nu$ and $(0,0)$ is between $-M_1$ and $-M_0$ is bounded, and only contains finitely many points of the form $\nu(\cV)-\nu(\cW)$ where $\cW$ is an exceptional bundle with $r(\cW)<r$.  Thus we can compute the set $S_\cV$ in a neighborhood of $1-\frac{e}{2}$, and in particular we see that it is finite near $1-\frac{e}{2}$.  A slightly more detailed analysis would show $S_\cV$ is discrete in $\RR_{>0}$, but we do not need this.

\end{remark}

\begin{example}\label{ex-stabilityIntervals}
Using Remark \ref{rem-stabilityIntervalCompute},
it is straightforward to program a computer to compute the intervals $I_\cV$ for all exceptionals $\cV$ of low rank.  At the same time, we can record the exceptional bundles $\cW$ that compute the endpoints of $I_\cV$, in the sense of Theorem \ref{thm-stabilityInterval}.  (If $e=0$, each endpoint is computed by such an exceptional bundle; if $e=1$, then the right endpoint is computed by an exceptional, and the left endpoint is computed by an exceptional if and only if it is not $0$.)  In Tables \ref{table-stabilityInterval0}, and \ref{table-stabilityInterval1}, we record the rank and first Chern classes of the exceptional bundles $\cV$ of rank up to $20$.  When $e=0$, we use twists, duality, and the symmetry between the two fiber classes to take $c_1(\cV) = (a,b) = a F_1+bF_2$ with $0\leq a < r/2$ and $a\leq b< r$.   When $e=1$, we use twists and duality to take $c_1(\cV) = (a,b) = a E + bF$ with $0\leq a \leq r/2$ and $0 \leq b < r$. We compute the stability interval $I_\cV = (m_0,m_1)$, and we compute the ranks and first Chern classes of exceptional bundles $\cW_i$ which compute the endpoints $m_i$ of $I_\cV$, if they exist.

\begin{table}\caption{Stability intervals and destabilizing bundles for exceptional bundles $\cV$ of rank up to $19$ on $\FF_0$.  See Example \ref{ex-stabilityIntervals}.}
\begin{tabular}{cccc}
$(r(\cV),c_1(\cV))$ & $I_\cV$ & $(r(\cW_0),c_1(\cW_0))$ & $(r(\cW_1),c_1(\cW_1))$\\\hline
$(1,(0,0))$ & $(0,\infty)$ & &\\
$(3,(1,1))$ & $(1/2,2)$  & $(1,(-1,1))$& $(1,(1,-1))$\\
$(5,(1,2))$ & $(1/2,3)$ &   $(1,(-1,1))$ &$(1,(1,-2))$\\
$(7,(1,3))$ & $(1/2,4)$  & $(1,(-1,1))$& $(1,(1,-3))$\\
$(9,(1,4))$ & $(1/2,5)$  & $(1,(-1,1))$& $(1,(1,-4))$\\
$(11,(1,5))$ & $(1/2,6)$  & $(1,(-1,1))$& $(1,(1,-5))$\\
$(11,(4,4))$ & $(4/7,7/4)$  & $(5,(-2,4))$& $(5,(4,-2))$\\
$(13,(1,6))$ & $(1/2,7)$  & $(1,(-1,1))$& $(1,(1,-6))$\\
$(15,(1,7))$ & $(1/2,8)$  & $(1,(-1,1))$& $(1,(1,-7))$\\
$(17,(1,8))$ & $(1/2,9)$ & $(1,(-1,1))$ & $(1,(1,-8))$\\
$(17,(5,5))$ & $(8/9,9/8)$  & $(7,(1,3))$& $(7,(3,1))$\\
$(19,(1,9))$ & $(1/2,10)$  & $(1,(-1,1))$& $(1,(1,-9))$\\
$(19,(4,7))$ & $(8/9,3)$  & $(7,(1,3))$& $(1,(1,-2))$\\
\end{tabular}\label{table-stabilityInterval0}
\end{table}

\begin{table}\caption{Stability intervals and destabilizing bundles for exceptional bundles of rank up to $20$ on $\FF_1$.  See Example \ref{ex-stabilityIntervals}.}
\begin{tabular}{cccc}
$(r(\cV),c_1(\cV))$ & $I_\cV$ & $(r(\cW_0),c_1(\cW_0))$ & $(r(\cW_1),c_1(\cW_1))$\\\hline
$(1,(0,0))$ & $(0,\infty)$ & &\\
$(2,(1,1))$ & $(0,1)$ &  & $(1,(1,0))$\\
$(4,(1,2))$ & $(0,2)$ &  & $(1,(1,-1))$\\
$(5,(2,2))$ & $(0,2/3)$ &  & $(1,(1,0))$\\
$(6,(1,3))$ & $(0,3)$ &  & $(1,(1,-2))$\\
$(8,(1,4))$ & $(0,4)$ &  & $(1,(1,-3))$\\
$(10,(1,5))$ & $(0,5)$ &  & $(1,(1,-4))$\\
$(11,(3,5))$ & $(3/7,2)$ & $(6,(1,3))$ & $(1,(1,-1))$\\
$(12,(1,6))$ & $(0,6)$ &  & $(1,(1,-5))$\\
$(13,(5,5))$ & $(0,5/8)$ &  & $(1,(1,0))$\\
$(14,(1,7))$ & $(0,7)$ &  & $(1,(1,-6))$\\
$(16,(1,8))$ & $(0,8)$ &  & $(1,(1,-7))$\\
$(18,(1,9))$ & $(0,9)$ &  & $(1,(1,-8))$\\
$(19,(5,10))$ & $(1/9,9/5)$ & $(5,(-2,3))$ & $(6,(5,-3))$\\
$(20,(1,10))$ & $(0,10)$ &  & $(1,(1,-9))$\\
\end{tabular}\label{table-stabilityInterval1}
\end{table}
\end{example}

\begin{remark}\label{rem-stabilityIntervalQuotient}
In Theorem \ref{thm-stabilityInterval}, the set $S_\cV$ could instead be defined by requiring $\chi(\cV,\cW) > 0$; in the proof, we just replace the maximal destabilizing subbundle with a minimal destabilizing quotient bundle, and the component of $\RR_{>0}\sm S_\cV$ containing $1-\frac{e}{2}$ won't change.
\end{remark}

Now that we understand the stability of exceptional bundles on $\FF_0$ and $\FF_1$, we are ready to study how the functions $\DLP_H^{<r}(\nu)$ and $\DLP_H(\nu)$ change with the polarization.  We examine some pictures first, since they explain the key phenomenon that occurs when exceptional bundles are destabilized.

\begin{example}\label{ex-variationF0}
Let $e=0$.  In Figure \ref{fig-F0H} we graph the functions $\DLP_{H_{1+\frac{t}{8}}}^{<8}(\epsilon E+\varphi F)$ on the unit square in the $(\epsilon,\varphi)$-plane, where $0\leq t\leq 8$.  There are two times $t$ to focus on:
\begin{itemize}
\item At time $t=0$, we recover the function with the anticanonical polarization (see Example \ref{ex-Kgraphs} and Figure \ref{fig-F0K}).
\item At time $t=8$, the polarization is $H_2$, where the rank $3$ exceptional bundle with $\nu = \frac{1}{3}E + \frac{1}{3}F$ is destabilized by the line sub-bundle $\OO_{\F_0}(E-F)$ or the quotient line bundle $\OO_{\F_0}(F)$.  At this time the branches of the surface controlled by the destabilizing objects meet up and cover the part of the surface controlled by the rank $3$ exceptional bundle.
\end{itemize}
\begin{figure}[t]
\begin{center}
\includegraphics[bb=0 0 5.68in 5.68in]{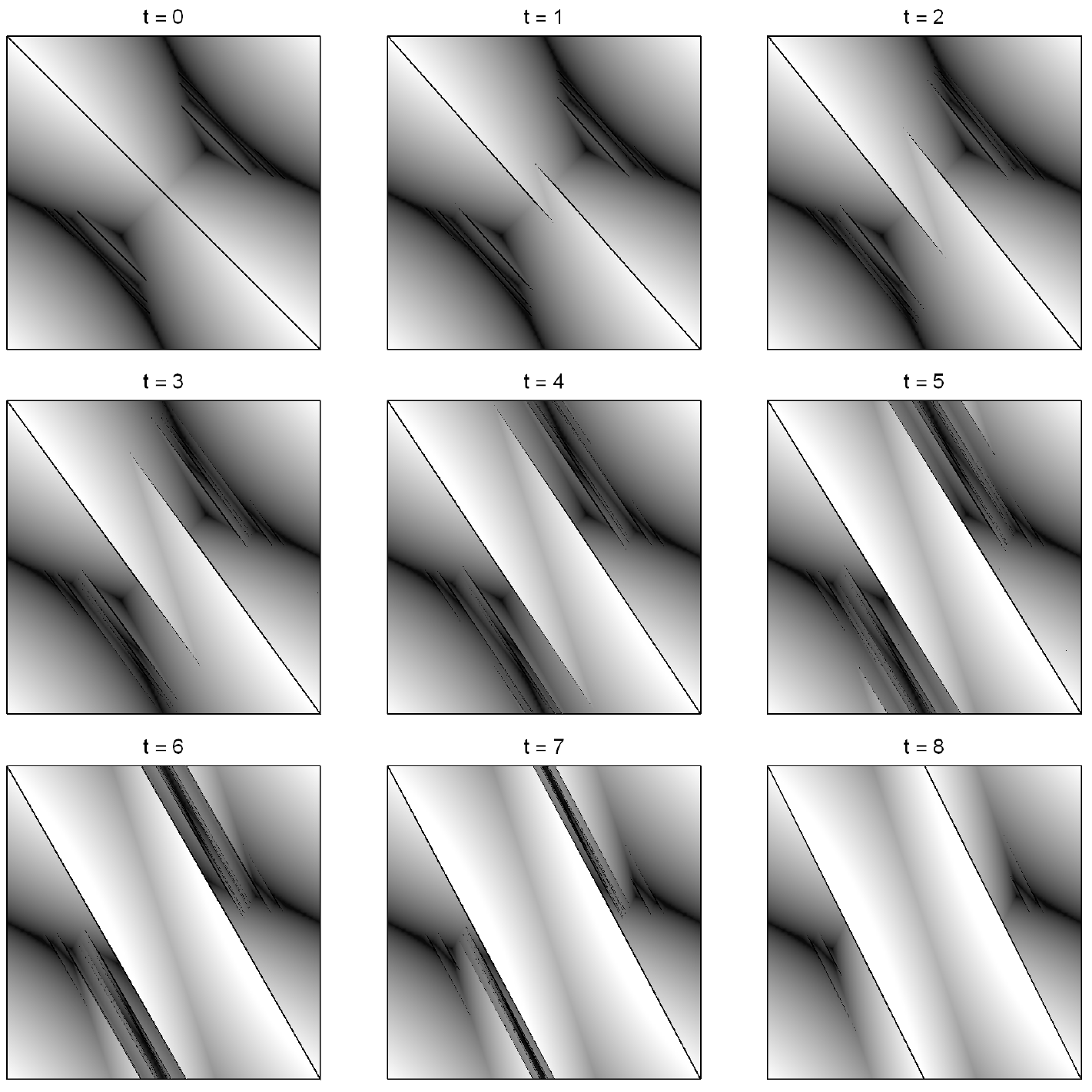}
\end{center}
\caption{For $e=0$ and $0\leq t \leq 8$, we graph the functions $\DLP_{H_{1+\frac{t}{8}}}^{<8}(\epsilon E + \varphi F)$ on the unit square in the $(\epsilon,\varphi)$-plane.  See Example \ref{ex-variationF0}.}\label{fig-F0H}
\end{figure}
\end{example}

\begin{example}\label{ex-variationF1}
Let $e=1$.  In Figure \ref{fig-F1H} we graph the functions $\DLP_{H_{\frac{1}{2}+\frac{t}{12}}}^{<7}(\epsilon E+ \varphi F)$ on the unit square in the $(\epsilon,\varphi)$-plane, where $0\leq t \leq 8$.  There are three times $t$ of particular note:  

\begin{itemize} \item At time $t=0$, we recover the function with the anticanonical polarization (see Example \ref{ex-Kgraphs} and Figure \ref{fig-F1K}).  

\item At time $t=2$ the polarization is $H_{2/3}$, where the rank $5$ exceptional bundle with $\nu = \frac{2}{5}E + \frac{2}{5}F$ is destabilized.  It is destabilized by the sub-line bundle $\OO_{\F_1}(E)$ and the quotient rank $4$ exceptional bundle with $\nu = \frac{1}{4}E + \frac{1}{2}F$.  At this time the portion of the surface controlled by these two destabilizing exceptional bundles meet up and cover the portion of the surface controlled by the rank $5$ exceptional bundle.

\item At time $t=6$ the polarization is $H_1$, and the rank two exceptional bundle with $\nu = \frac{1}{2}E + \frac{1}{2}F$ is destabilized by the sub-line bundle $\OO_{\F_1}(E)$ and the quotient line bundle $\OO_{\F_1}(F)$.  The branches of the surface controlled by these line bundles meet up and cover the portion of the surface controlled by the rank $2$ exceptional bundle.
\end{itemize}

\begin{figure}[t]
\begin{center}
\includegraphics[bb=0 0 5.68in 5.68in]{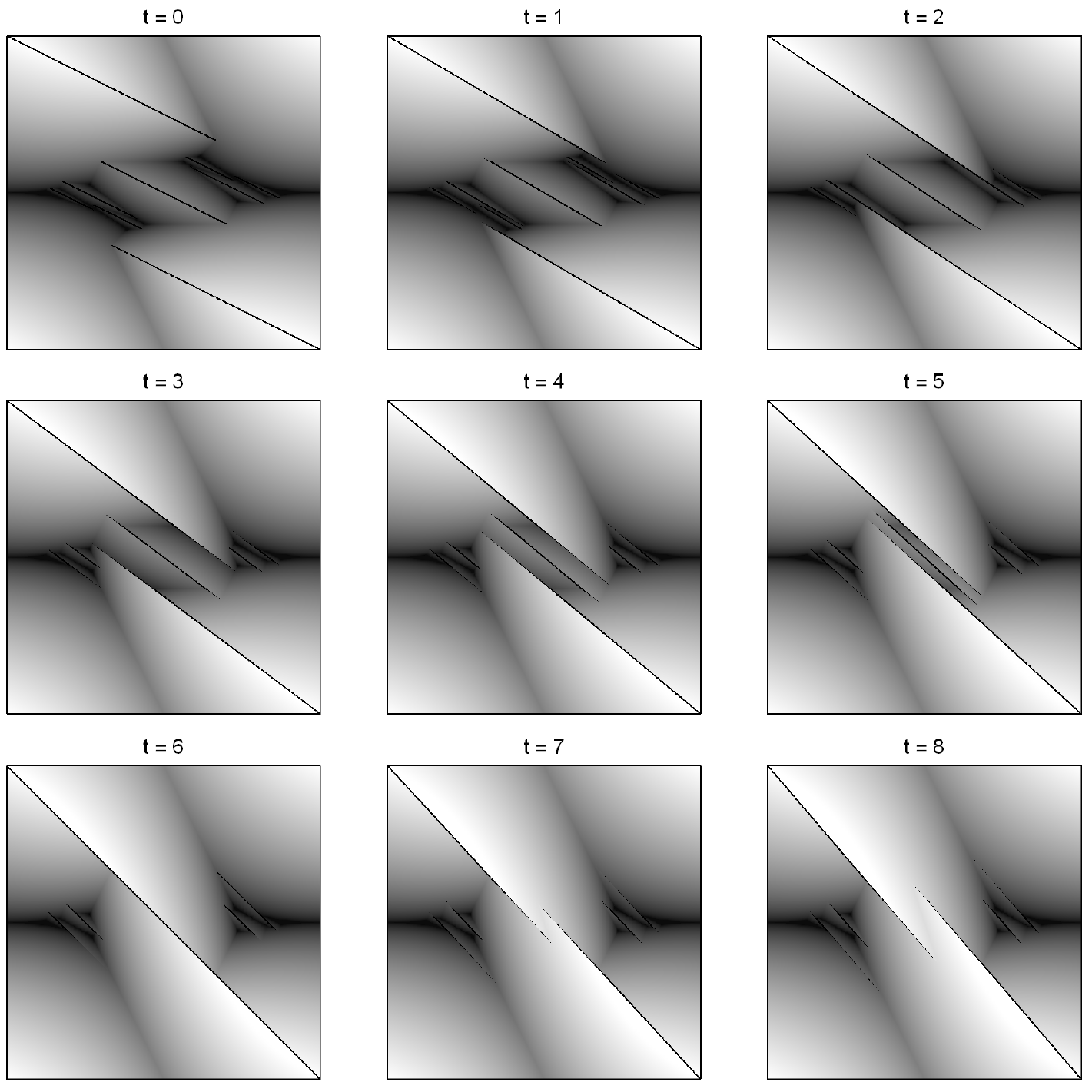}
\end{center}
\caption{For $e=1$ and $0\leq t \leq 8$, we graph the functions $\DLP_{H_{\frac{1}{2}+\frac{t}{12}}}^{<7}(\epsilon E + \varphi F)$ on the unit square in the $(\epsilon,\varphi)$-plane.  See Example \ref{ex-variationF1}.}\label{fig-F1H}
\end{figure}
\end{example}

Now we generalize Lemma \ref{lem-monotonicity2} to higher rank.

\begin{proposition}[Monotonicity in the polarization]\label{prop-DLPmonotone}
Let $e=0$ or $1$, let $r\geq 2$, and let $\nu\in \Pic(\F_e)\te \QQ$ be a total slope.  Consider the polarizations $A_m = -\frac{1}{2}K_{\F_e}+mF$.  If $0 \leq m \leq m'$ or if $\frac{e}{2}-1<m' \leq m \leq 0$, then $$\DLP_{A_m}^{<r}(\nu) \leq \DLP_{A_{m'}}^{<r}(\nu) \qquad \textrm{and} \qquad \DLP_{A_m}(\nu) \leq \DLP_{A_{m'}}(\nu).$$ 
\end{proposition}
\begin{proof}
We consider the function $\DLP_H^{<r}(\nu)$; all the arguments apply identically to $\DLP_{H}(\nu)$ after trivial modifications.  As $H$ moves away from $-K_{\F_e}$, the number $\DLP_{H}^{<r}(\nu)$  can potentially change in three different ways: there must be a $\mu_H$-stable exceptional bundle $\cV$ of rank less than $r$ such that either 
\begin{enumerate}
\item as $H$ changes the branch of $\DLP_{H,\cV}$ used to compute $\DLP_{H,\cV}(\nu)$ changes;
\item as $H$ changes the slope $\nu$ either leaves or enters the domain of definition of $\DLP_{H,\cV}$; or
\item as $H$ moves away from $-K_{\F_e}$ the bundle $\cV$ becomes unstable, so that it no longer contributes to the computation of the supremum in the definition of $\DLP_H^{<r}(\nu)$.
\end{enumerate}
We discuss each case in turn, with cases (1) and (2) being similar to the line bundle case discussed in the proof of Lemma \ref{lem-monotonicity2}.  
 
\emph{Case 1:} This case is handled exactly as in the line bundle case; just replace $L$ with $\cV$ in the proof.

\emph{Case 2:} As in the proof of Lemma \ref{lem-monotonicity2}, there are four cases to consider.  Again consider the case where $(\nu-\nu(\cV))\cdot H = -\frac{1}{2}K_{\F_e}\cdot H$, let $H_-$ and $H_+$ be polarizations slightly closer to and farther from $-K_{\F_e}$, respectively, and suppose $\nu$ is in the domain of definition of $\DLP_{H_-,\cV}$ but not in the domain of definition of $\DLP_{H_+,\cV}$.  Consider the bundle $\cV' = \cV \te (-K_{\F_e})$.  Since $\cV$ is $\mu_H$-stable, it is also $\mu_{H_+}$-stable, and $\cV'$ is $\mu_{H_+}$-stable.  Then as in the proof of Lemma \ref{lem-monotonicity2}, the slope $\nu$ is in the domain of definition of $\DLP_{H_+,\cV'}$ and $\DLP_{H_+,\cV'}(\nu) = \DLP_{H_-,\cV}(\nu)$.  We conclude as before.

\emph{Case 3:} The idea is that as $\cV$ is destabilized, one of its Harder-Narasimhan factors will provide a stronger inequality.  Let $H$ be a polarization such that $\cV$ is strictly $H$-semistable, and let $H_+$ be a polarization slightly farther from $-K_{\F_e}$.  Let $\cV_1,\ldots,\cV_\ell$ be the $H_+$-Harder-Narasimhan factors of $\cV$.  By an induction on $\ell$ and Proposition \ref{prop-mukai} (2), these factors will all be rigid and semistable for the generic polarization $H_+$.  They are direct sums of exceptional bundles by Theorem \ref{thm-rigidSplit}, but since they are $H_+$-semistable they must be semiexceptional.  Let $\cE_1,\ldots,\cE_\ell$ be exceptional bundles with $\cV_i \cong \cE_i^{\oplus a_i}$.  By our choice of $H_+$, we have $\mu_H(\cV) = \mu_H(\cE_i)$ for all $i$, and therefore the functions $\DLP_{H,\cV}$ and $\DLP_{H,\cE_i}$ have the same domain of definition.  Let $\nu$ be a total slope in this domain.  The bundles $\cE_i$ are $H$-stable (and $H_+$-stable) of rank less than $r$, so they contribute to the computation of $\DLP_{H}^{<r}(\nu)$, and they will continue to contribute to the computation of $\DLP_{H_+}^{<r}(\nu)$ so long as $\nu$ remains in the domain of definition; if $\nu$ does not remain in the domain of definition, then as in Case (2) the exceptional bundle can be replaced by a twist without changing the computed value.

To complete the proof, we must show that $$\DLP_{H,\cV}(\nu) \leq \max_i \DLP_{H,\cE_i}(\nu).$$  Since $\cV$ and the $\cE_i$ all have the same $H$-slope, the numbers $(\nu-\nu(\cV))\cdot H$ and $(\nu-\nu(\cE_i))\cdot H$ are all equal and the numbers $\DLP_{H,\cV}(\nu)$ and $\DLP_{H,\cE_i}(\nu)$ are all computed using the same branch.  Without loss of generality assume $(\nu-\nu(\cV)) \cdot H < 0$, as the other two cases are dealt with similarly.  Let $\bv = \ch \cV$, $\bv_i = \ch \cV_i$, and $\be_i = \ch \cE_i$, and consider the (rational) character $\bw = (1,\nu,0)$ of rank $1$ and discriminant $0$.  Then the inequality to be proved is $$\frac{\chi(\bv,\bw)}{r(\bv)} \leq \max_i \frac{\chi(\be_i,\bw)}{r(\be_i)}.$$  But $\bv = \sum_i a_i \be_i$, so the number on the left is the weighted mean of the numbers on the right:\begin{align*}\frac{\chi(\bv,\bw)}{r(\bv)} &= \frac{a_1 \chi(\be_1,\bw) + \cdots + a_\ell \chi(\be_\ell,\bw)}{a_1r(\be_i) +\cdots + a_\ell r(\be_\ell)} \\&= \frac{a_1 r(\be_1) \frac{\chi(\be_1,\bw)}{r(\be_1)}+\cdots + a_\ell r(\be_\ell) \frac{\chi(\be_\ell,\bw)}{r(\be_\ell)}}{a_1r(\be_1)+\cdots + a_\ell r(\be_\ell)}.\end{align*} Thus the required inequality holds.
\end{proof}

Combining Corollary \ref{cor-K1/2} and Proposition \ref{prop-DLPmonotone} immediately proves the following result.

\begin{corollary}\label{cor-DLP1/2}
Let $e=0$ or $1$ and let $\nu\in \Pic(\F_e)\te \QQ$.  Let $H$ be an arbitrary polarization.  Then $$\DLP_{H}(\nu) \geq \frac{1}{2}.$$  
\end{corollary}

Theorem \ref{thm-stabilityInterval} focuses on describing the polarizations such that a given exceptional bundle $\cV$ is stable.  On the other hand, we can fix a polarization and study the stable exceptionals.  We can now show that the Dr\'ezet-Le Potier surface answers this question.  This result generalizes Corollary \ref{cor-delPezzoExceptional} to arbitrary polarizations.

\begin{corollary}\label{cor-DLPexcdelPezzo}
Let $e=0$ or $1$, and let $H$ be an arbitrary polarization.  If $\bv = (r,\nu,\Delta)$ is potentially exceptional, then there is a $\mu_H$-stable exceptional bundle $\cV$ of character $\bv$ if and only if $$ \Delta \geq \DLP_{H}^{<r}(\nu).$$
\end{corollary}
\begin{proof}
($\Rightarrow$) Suppose there is a $\mu_H$-stable exceptional bundle $\cV$.  Then $\cV$ is also $\mu_{H_+}$-stable for a polarization $H_+$ slightly farther from $-K_{\F_e}$, and  Proposition \ref{prop-DLPss} (1) gives $\Delta \geq \DLP_{H_+}^{<r}(\nu)$.  By Proposition \ref{prop-DLPmonotone}, we have $\Delta \geq \DLP_{H}^{<r}(\nu)$.

($\Leftarrow$) Suppose $\Delta \geq \DLP_H^{<r}(\nu)$.  Then by Proposition \ref{prop-DLPmonotone}, we know that $\Delta \geq \DLP_{-K_{\F_e}}^{<r}(\nu)$, and therefore by Corollary \ref{cor-delPezzoExceptional} there is an exceptional bundle $\cV$ of character $\bv$.  It remains to show that $\cV$ is $\mu_H$-stable.  Let $$I = \{m>0 : \Delta \geq \DLP_{H_m}^{<r}(\nu)\}.$$  By Proposition \ref{prop-DLPmonotone}, $I$ is an interval containing $1-\frac{e}{2}$.  As $m$ moves away from $1-\frac{e}{2}$, the number $\DLP_{H_m}^{<r}(\nu)$ is locally constant and only jumps up at special values, so $I$ is open.

We claim that $I = I_\cV$ is the stability interval of $\cV$.  From the first direction of the proof, we know that if $\cV$ is $\mu_{H_m}$-stable, then $m\in I$, and therefore $I_\cV \subset I$.  On the other hand, suppose $\cV$ is strictly $\mu_{H_m}$-semistable, and in the notation of Theorem \ref{thm-stabilityInterval} let $\cW$ be a $\mu_{H_m}$-stable exceptional bundle that shows that $m=m_{\cV,\cW}\in S_\cV$.  Let $H_+$ be a polarization slightly farther away from $-K_{\F_e}$ than $H_m$.  Then $\chi(\cW,\cV) > 0$ and $\mu_{H_+}(\cW) > \mu_{H_+}(\cV)$, which shows that $\Delta < \DLP_{H_+,\cW}(\nu)$ and therefore $\Delta < \DLP_{H_+}^{<r}(\nu)$.  This shows that $m$ must be an endpoint of the interval $I$.  Therefore $I = I_\cV$.
\end{proof}

\subsection{Generic stability intervals} We can also combine the results of the past three sections to study the stability interval of a general bundle on a del Pezzo Hirzebruch surface.  Intuitively, for a general sheaf it is easier to be $\mu_{-K_{\F_e}}$-stable than slope stable for other polarizations.

\begin{corollary}\label{cor-KstabilityEasy}
Let $e=0$ or $1$, and let $\bv = (r,\nu,\Delta)$ be a character such that there is a $\mu_{H_m}$-stable sheaf $\cV$ for some $m$.  Then the general sheaf in the stack $\cP_{F}(\bv)$ is $\mu_{-K_{\F_e}}$-stable.  

In other words, the generic stability interval $I_\bv$ contains $1-\frac{e}{2}$.
\end{corollary}
\begin{proof}
Without loss of generality assume $m > 1-\frac{e}{2}$.  Let $\cV\in \cP_{F}(\bv)$ be a general sheaf. Then it is $\mu_{H_{m+\epsilon}}$-stable for $\epsilon>0$ small.  By Propositions \ref{prop-DLPss} and \ref{prop-DLPmonotone} and Lemma \ref{lem-limitK}, we find (shrinking $\epsilon$ if necessary) $$\Delta \geq \DLP^{<r}_{H_{m+\epsilon}}(\nu)\geq \DLP_{-K_{\F_e}}^{<r}(\nu) = \DLP_{-K_{\F_e}-\epsilon F}^{<r} (\nu).$$ By Theorem \ref{thm-existence}, the general sheaf in $\cP_F(\bv)$ is both $(-K_{\F_e}-\epsilon F)$-semistable and $\mu_{H_{m+\epsilon}}$-stable, and so it is $\mu_{-K_{\F_e}}$-stable.
\end{proof}

\begin{corollary}
Let $e=0$ or $1$, and let $I_{\bv}$ be the generic stability interval of $\bv\in K(\F_e)$.  
\begin{enumerate}
\item If $m\in I_\bv$, then the moduli space $M_{H_m}(\bv)$ is birational to $M_{-K_{\F_e}}(\bv)$.  
\item If $m\notin \overline{I}_{\bv}$, then $M_{H_m}(\bv)$ is empty.
 \end{enumerate}
 \end{corollary}
 \begin{proof}
(1) The spaces $M_{H_m}(\bv)$ and $M_{-K_{\F_e}}(\bv)$ are both irreducible, and by Corollary \ref{cor-KstabilityEasy} they share the open dense subsets of simultaneously $\mu_{H_m}$- and $\mu_{-K_{\F_e}}$-stable sheaves.

(2) If $\cV\in M_{H_m}(\bv)$, then $m\in \overline{I}_{\cV}$, and then $m\in \overline{I}_{\bv}$ by the construction of $\overline{I}_{\cV}$.
 \end{proof}

\begin{example}
Sheaves that are not general in moduli can be slope-stable for some polarization $H$ and fail to be slope-stable for $-K_{\F_e}$.  For example, on $\F_0$ it is easy to show that a general extension of the form $$0\to \OO_{\F_0}\to \cV \to \OO_{\F_0}(2F_1-3F_2)\to 0$$ is strictly $\mu_{H_{\frac{3}{2}}}$-semistable and $\mu_{H_{\frac{3}{2}+\epsilon}}$-stable for all $\epsilon > 0$, so that $I_\cV = (\frac{3}{2},\infty)$.  But, such sheaves $\cV$ are certainly not general in moduli; we have $\Delta(\cV) = \frac{3}{2}$, so $$\dim (M_{H_{\frac{3}{2}+\epsilon}}(\bv)) = \dim (M_{-K_{\F_e}}(\bv)) = r^2(2\Delta -1 )+1=9,$$ while $\dim( \P\Ext^1(\OO_{\F_0}(2F_1-3F_2),\OO_{\F_0}) )= 3$.
\end{example}

\section{Sharp Bogomolov inequalities}\label{sec-sharpBogomolov}

In this section we introduce functions of the slope which provide sharp Bogomolov-type inequalities for various stabilities.  We study their general properties in preparation for the rest of the paper, where we will perform some computations of these functions.

\begin{definition}\label{def-deltass}
Let $\nu\in \Pic(\F_e) \te \QQ$.  We define $$\delta_{m}^{\mus}(\nu) = \inf\left\{\Delta \geq \frac{1}{2}:\textrm{there is a $\mu_{H_m}$-stable sheaf of total slope $\nu$ and discriminant $\Delta$}\right\}.$$ We identically define functions $\delta_m^{s}$, $\delta_m^{ss}$, and $\delta_m^{\muss}$, where the notion of stability is replaced with $H_m$-stability, $H_m$-semistability, and $\mu_{H_m}$-semistability, respectively.  If the Hirzebruch surface $\F_e$ is not evident, we may write notation such as $\delta_{m,\F_e}^{ss}$.
\end{definition}

Analogously, recall that for $m\in \Z$ we defined the function $\delta_m^p(\nu)$ which bounds the discriminant of $H_m$-prioritary sheaves.  See Corollary \ref{cor-prioritaryDelta}.  From the basic relationship between Gieseker and slope stability and Proposition \ref{prop-ssPrior}, it follows that we have inequalities $$\delta_{\lceil m\rceil +1}^p(\nu) \leq \delta_m^{\mu\textrm{-}ss}(\nu)\leq \delta_m^{ss}(\nu)\leq \delta_m^s(\nu) \leq \delta_m^{\mu\textrm{-}s}(\nu)$$ for any $m$ and $\nu$. If $m$ is special, then there are examples showing that any one of the inequalities can be strict.  On the other hand, if $m$ is general, then things are much better behaved, as we now show.

\subsection{Generic polarization}\label{ssec-deltaGeneric} Here we compare the various $\delta$-functions in the case where the polarization $H_m$ is generic.

\begin{theorem}\label{thm-deltaGeneric}
Let $\nu\in \Pic(\F_e)\te \QQ$, and let $m$ be generic.  Then $$\delta_m^{ss}(\nu)=\delta_m^s(\nu)=\delta_m^{\mu\textrm{-}s}(\nu)$$ If furthermore there is no $\mu_{H_m}$-stable exceptional bundle of total slope $\nu$, then these numbers also equal $\delta_m^{\muss}(\nu)$.
\end{theorem}

The result follows from a series of facts about the existence of sheaves exhibiting various stabilities.  Since the polarization is generic, the proofs mostly carry over directly from results in \cite{DLP}, where analogous statements were proved for $\P^2$.

\begin{proposition}
If $m$ is generic and there is no $\mu_{H_m}$-stable exceptional bundle of total slope $\nu$, then $$\delta_m^{\muss}(\nu) = \delta_m^{ss}(\nu)$$
\end{proposition}
\begin{proof}
Let $\cV$ be a $\mu_{H_m}$-semistable sheaf with invariants $(r,\nu,\Delta)$.  Consider a (possibly trivial) $H_m$-Harder-Narasimhan filtration for $\cV$ with factors $\gr_1,\ldots,\gr_\ell$ which have invariants $(r_i,\nu_i,\Delta_i)$.  Since $m$ is generic, each factor $\gr_i$ has the same total slope $\nu$, the discriminants $\Delta_i$ are strictly increasing, and $\Delta$ is a weighted mean $$r \Delta = r_1\Delta_1+\cdots + r_\ell \Delta_\ell.$$  Thus $\gr_1$ is an $H_m$-semistable sheaf with $\Delta_1 \leq \Delta$.  If $\Delta_1 < \frac{1}{2}$ then an $H_m$-Jordan-H\"older factor of $\gr_1$ is a $\mu_{H_m}$-stable exceptional bundle by Lemma \ref{lem-excFacts} (4 \& 5), a contradiction.  Therefore $\Delta_1 \geq \frac{1}{2}$ and $\delta_m^{ss}(\nu) \leq \Delta_1.$
\end{proof}

\begin{remark}
Let $m$ be generic, and suppose there is an $H_m$-stable exceptional bundle $\cV$ of total slope $\nu$ and rank $r$.  Here we show that $\delta_m^{\muss}(\nu) = \frac{1}{2} < \delta_m^{ss}(\nu)$.

The bundle $\cV$ is $\mu_{H_m}$-stable, so an elementary modification $$0\to \cV'\to \cV \to \OO_p\to 0$$ is also $\mu_{H_m}$-stable.  Then $$
\Delta(\cV) = \frac{1}{2}\left(1-\frac{1}{r^2}\right) \qquad \textrm{and} \qquad  
\Delta(\cV') = \frac{1}{2}\left(1-\frac{1}{r^2}\right) + \frac{1}{r},$$ and a straightforward computation shows $$\Delta(\cV^{\oplus(2r-1)} \oplus \cV') = \frac{1}{2},$$
so $\cV^{\oplus(2r-1)}\oplus \cV'$ is $\mu_{H_m}$-semistable of discriminant $\frac{1}{2}$.  Therefore $\delta_m^{\muss}(\nu) = \frac{1}{2}$.  On the other hand, we have $$\delta_m^{ss}(\nu) \geq \DLP_{H_m}(\nu) \geq \DLP_{H_m,\cV}(\nu) = \frac{1}{2}\left(1+\frac{1}{r^2}\right).$$
\end{remark}

\begin{proposition}\label{prop-ssIMPs}
Suppose $m$ is generic and there are $H_m$-semistable sheaves of character $\bv = (r,\nu,\Delta)$.  If $\Delta > \frac{1}{2}$, then there are $H_m$-stable sheaves of character $\bv$.  Therefore, $$\delta_m^{ss}(\nu) = \delta_m^s(\nu).$$
\end{proposition}
\begin{proof}
The argument is essentially the same as the argument given in \cite[Theorem 4.10]{DLP} in the case of $\P^2$, so we will be brief.  One considers the characters $\bgr_1,\ldots,\bgr_\ell$ of a hypothetical length $\ell \geq 2$ Jordan-H\"older filtration of a general sheaf in $\cP_{H_{\lceil m\rceil}}(\bv)$.  Since $m$ is generic, all these characters have the same total slope $\nu$ and discriminant $\Delta$.  Since $\Delta > \frac{1}{2}$, it follows that $\chi(\bgr_i,\bgr_j) = r_ir_j(1-2\Delta) < 0$.  But we can estimate that the codimension of the corresponding Schatz stratum is at least $$-\sum_{i<j} \chi(\bgr_i,\bgr_j) > 0.$$ Therefore no such stratum is dense, and the general sheaf is $H_m$-stable.

If $\delta_m^{ss}(\nu) >\frac{1}{2}$ or if $\delta_m^{ss}(\nu) = \frac{1}{2}$ and the infimum is computed by a sequence of $H_m$-semistable bundles with decreasing discriminants, then this shows that $\delta_m^{ss}(\nu) = \delta_m^{s}(\nu)$.  Finally, it is possible that $\delta_m^{ss}(\nu)=\frac{1}{2}$ is computed by an $H_m$-semistable bundle $\cV$ of discriminant $\frac{1}{2}$.  But then since $m$ is generic any $H_m$-Jordan-H\"older factor of $\cV$ is $H_m$-stable of total slope $\nu$ and discriminant $\frac{1}{2}$,  so $\delta_m^s(\nu) = \frac{1}{2}$.
\end{proof}

The next fact completes the proof of Theorem \ref{thm-deltaGeneric}.

\begin{proposition}\label{prop-sIMPmus}
Suppose $m$ is generic and there are $H_m$-stable sheaves of character $\bv=(r,\nu,\Delta)$.  Then there are $\mu_{H_m}$-stable sheaves of character $\bv$.  Therefore, $$\delta_m^{s}(\nu) = \delta_m^{\mus}(\nu).$$
\end{proposition}
\begin{proof}
The result is clear if $\bv$ is exceptional, so assume $\bv$ is not exceptional. Again the argument closely follows the argument in \cite[Theorem 4.11]{DLP} for $\P^2$.  If the $H_m$-stable sheaf of character $\bv$ is not $\mu_{H_m}$-stable, then we can find characters $\bgr_1,\bgr_2$ such that the general sheaf $\cV$ admits a filtration $$0\to \gr_1\to \cV\to \gr_2\to 0$$ where $\gr_i$ is $H_m$-semistable of character $\bgr_i = (r_i,\nu,\Delta_i)$.  Here we will have $\Delta_1 > \Delta_2$ since $\cV$ is $H_m$-stable.  Estimating the codimension of the corresponding Schatz stratum as in \cite{DLP}, we see that we must have $\chi(\bgr_1,\bgr_2) \geq 0$ in order for the stratum to be dense.  Riemann-Roch then implies $\Delta_1+\Delta_2 \leq 1$, and therefore $\Delta_2 < \frac{1}{2}$.  So, $\gr_2$ is an $H_m$-semistable semiexceptional bundle.  But then $$\chi(\cV,\bgr_2) = \chi(\bgr_1,\bgr_2)+\chi(\bgr_2,\bgr_2) > 0,$$ and this gives $$\Delta < \DLP_{H_m,\gr_2}(\nu).$$ By Proposition \ref{prop-DLPss}, this contradicts the $H_m$-stability of $\cV$ since $\cV$ is not exceptional.  Therefore no such stratum is dense, and the general sheaf is $\mu_{H_m}$-stable.
\end{proof}

\subsection{Existence of sheaves with discriminant above $\delta_m^{\mus}(\nu)$} In this section we show that the function $\delta_m^{\mus}(\nu)$ is reminiscient of the Dr\'ezet-Le Potier curve that appears in the classification of semistable sheaves on $\P^2$.  Specifically, there are always $\mu_{H_m}$-stable sheaves if the discriminant lies above $\delta_m^{\mus}(\nu)$.

\begin{theorem}\label{thm-deltaSurface}
Let $\bv=(r,\nu,\Delta)\in K(\F_e)$ and let $m\in \QQ_{>0}$ be arbitrary.

\begin{enumerate} \item  If $$\Delta > \delta_m^{\mus}(\nu),$$ then there are $\mu_{H_m}$-stable sheaves of character $\bv$.

\item If there is a non-exceptional $\mu_{H_m}$-stable sheaf of character $\bv$, then $$\Delta \geq \delta_m^{\mus}(\nu).$$

\item If there is a $\mu_{H_m}$-stable sheaf of slope $\nu$ and discriminant $\Delta = \delta_m^{\mus}(\nu)>\frac{1}{2}$, then non-exceptional $\mu_{H_m}$-stable sheaves of character $\bv$ exist if and only if $\Delta \geq \delta_m^{\mus}(\nu)$.
\end{enumerate}
\end{theorem}

Since elementary modifications take slope-stable sheaves to slope-stable sheaves, it is clear that there is a function $f_m(r,\nu)$ of the rank and slope such that $H_m$-slope stable sheaves of character $(r,\nu,\Delta)$ exist if and only if $\Delta \geq f_m(r,\nu)$.  So, the interesting part of Theorem \ref{thm-deltaSurface} is that the dependence on the rank is not necessary.  The next result gives the key step.

\begin{proposition}\label{prop-divideSpace}
Let $\bv = (r,\nu,\Delta)\in K(\F_e)$, let $n$ be a positive integer, and suppose there are $H_m$-semistable sheaves of character $n\bv$.  Then there are $H_m$-semistable sheaves of character $\bv$.
\end{proposition}
\begin{proof}
Since there are $H_m$-semistable sheaves of character $n\bv$, there are $H_{\lceil m\rceil +1}$-prioritary sheaves of character $n\bv$.  Therefore $$\Delta \geq \delta_{\lceil m\rceil +1}^p(\nu),$$ and by Corollary \ref{cor-prioritaryDelta} there are $H_{\lceil m\rceil+1}$-prioritary sheaves of character $\bv$.  Suppose there are not $H_m$-semistable sheaves of character $\bv$, and consider the characters $\bgr_1,\ldots,\bgr_\ell$ of the generic $H_m$-Harder-Narasimhan filtration for $\bv$.  Then by Theorem \ref{thm-HNcriterion} the characters $n \bgr_1,\ldots,n \bgr_\ell$ are the characters of the generic $H_m$-Harder-Narasimhan filtration for $n\bv$.  Therefore there are no $H_m$-semistable sheaves of character $n\bv$, a contradiction. 
\end{proof}

It is easy to deduce Theorem \ref{thm-deltaSurface} from the case where $m$ is generic, so we first focus on that case.  As is often the case, sheaves of discriminant $\frac{1}{2}$ require some care.

\begin{lemma}\label{lem-muss12}
Suppose $m$ is generic and there is a $\mu_{H_m}$-stable sheaf $\cV$ of character $\bv = (r,\nu,\frac{1}{2})$.  Then for any $k\geq 1$, there is a $\mu_{H_m}$-stable sheaf of character $\bw = (rk,\nu,\frac{1}{2}+\frac{1}{rk}).$
\end{lemma}
\begin{proof}
Let $\cV'$ be an elementary modification of $\cV$: $$0\to \cV'\to \cV\to \OO_p\to 0.$$  Then $\cV'$ is $\mu_{H_m}$-stable, and $\cW = \cV^{\oplus(k-1)}\oplus \cV'$ is $\mu_{H_m}$-semistable of character $\bw = (rk,\nu,\frac{1}{2}+\frac{1}{rk})$.  Therefore, there are $H_{\lceil m\rceil + 1}$-prioritary sheaves of character $\bw$.  Consider the characters $\bgr_1,\ldots,\bgr_\ell$ of the factors of the generic $H_m$-Harder-Narasimhan filtration of sheaves in $\cP_{H_{\lceil m\rceil +1}}(\bw)$.  Since the generic sheaf is $\mu_{H_m}$-semistable and $m$ is generic, all the factors $\bgr_i$ have the same total slope.  

None of the factors $\bgr_i$  can have discriminant $\Delta_i$ less than $\frac{1}{2}$.  Indeed, if one did then there would be an exceptional bundle $\cE$ of slope $\nu$.  Then $\DLP_{H_m}(\nu) \geq \DLP_{H_m,\cE}(\nu) > \frac{1}{2}$, contradicting the stability of $\cV$ and Proposition \ref{prop-DLPss}.  On the other hand, we must have $\chi(\bgr_i,\bgr_j) = 0$ for all $i<j$, and by Riemann-Roch this requires $\Delta_i+\Delta_j = 1$.  Therefore all $\Delta_i$ must be $\frac{1}{2}$, which contradicts $\Delta(\cW) > \frac{1}{2}$.  Thus the generic Harder-Narasimhan filtration must have length $1$, and there are $H_m$-semistable sheaves of character $\bw$.  By Propositions \ref{prop-ssIMPs} and \ref{prop-sIMPmus}, there are $\mu_{H_m}$-stable sheaves of character $\bw$.
\end{proof}

\begin{lemma}\label{lem-deltaSurface}
If $m$ is generic, then Theorem \ref{thm-deltaSurface} holds.
\end{lemma}
\begin{proof}
(1) Suppose $\Delta > \delta_m^{\mus}(\nu)$.  Pick a $\mu_{H_m}$-stable sheaf $\cW$ with invariants $\bw = (r',\nu,\Delta')$ such that $\Delta > \Delta' > \frac{1}{2}$, using Lemma \ref{lem-muss12} in case $\delta_m^{\mus}(\nu)=\frac{1}{2}$. Then $\cW^{\oplus r}$ is an $H_m$-semistable sheaf of character $r\bw$, and by Propositions \ref{prop-ssIMPs} and \ref{prop-sIMPmus} there are $\mu_{H_m}$-stable sheaves of character $r\bw$; pick one such sheaf $\cU$.  Let $$n = \chi(r\bw) - \chi(r' \bv).$$  Then if we perform $n$ general elementary modifications on $\cU$, we will get a $\mu_{H_m}$-stable sheaf $\cU'$ of character $r'\bv$. Proposition \ref{prop-divideSpace} then provides $H_m$-semistable sheaves of character $\bv$, and Propositions \ref{prop-ssIMPs} and \ref{prop-sIMPmus} give $\mu_{H_m}$-stable sheaves of character $\bv$.

(2) If there is a non-exceptional $\mu_{H_m}$-stable sheaf then $\Delta \geq \frac{1}{2}$, so this is clear. 

(3) Repeat the argument in (1), except choose $\cW$ so that $\Delta' = \delta_m^{\mus}(\nu)$.
\end{proof}

Now we can quickly deduce Theorem \ref{thm-deltaSurface} for arbitrary $m$.

\begin{proof}[Proof of Theorem \ref{thm-deltaSurface}]
(1) Since $\Delta > \delta_m^{\mus}(\nu)$, there is a $\mu_{H_m}$-stable sheaf $\cW$ with invariants $\bw = (r',\nu,\Delta')$ such that $\Delta > \Delta' \geq \frac{1}{2}.$  Then since $\mu_{H_m}$-stability is open in the polarization, if $\epsilon>0$ is small then $\cW$ is both $\mu_{H_{m-\epsilon}}$- and $\mu_{H_{m+\epsilon}}$-stable, and therefore $\delta_{m\pm\epsilon}^{\mus}(\nu) \leq \Delta'$.  By Lemma \ref{lem-deltaSurface}, there are $\mu_{H_{m\pm \epsilon}}$-stable sheaves of character $\bv$, and the irreducibility of the stack of $F$-prioritary sheaves shows that there are sheaves of character $\bv$ that are simultaneously $\mu_{H_{m-\epsilon}}$- and $\mu_{H_{m+\epsilon}}$-stable.  Such sheaves are automatically slope stable stable for the convex combination $H_m$ of $H_{m-\epsilon}$ and $H_{m+\epsilon}$.

(2) As in the proof of Lemma \ref{lem-deltaSurface}.

(3) The argument in (1) is easily adapted to this situation.
\end{proof}

It is worth pointing out that the function $\delta_m^{\mus}(\nu)$ has a monotonicity property in the polarization.  Here we treat the del Pezzo case.  See \S\ref{sec-reduction} when $e\geq 2$.

\begin{corollary}\label{cor-deltaMonotone}
Let $e = 0$ or $1$, let $\nu\in \Pic(\F_e)\te \QQ$, and suppose either $1-\frac{e}{2} \leq m \leq m'$ or $0 < m' \leq m \leq 1-\frac{e}{2}$.  Then $$\delta_m^{\mus}(\nu) \leq \delta_{m'}^{\mus}(\nu).$$
\end{corollary}
\begin{proof}
By Corollary \ref{cor-KstabilityEasy}, if there are $\mu_{H_{m'}}$-stable sheaves of some character $\bv = (r,\nu,\Delta)$, then there are $\mu_{H_m}$-stable sheaves of character $\bv$. 
\end{proof}

We can compare the function $\delta_m^{\mus}(\nu)$ to $\DLP_{H_m}(\nu)$.

\begin{corollary}\label{cor-deltaDLP01}
Let $e=0$ or $1$, let $\nu\in \Pic(\F_e)\te \QQ$, and let $m>0$.  Then $$\delta_m^{\mus}(\nu) \geq \DLP_{H_m}(\nu).$$
\end{corollary}
\begin{proof}
Without loss of generality suppose $m \geq 1-\frac{e}{2}$.  
Suppose $\cV$ is $\mu_{H_m}$-stable with $\Delta(\cV) \geq \frac{1}{2}$.  Then for $\epsilon >0$ small it is also $\mu_{H_{m+\epsilon}}$-stable.  Since it is not semiexceptional, Propositions \ref{prop-DLPss} (2) and \ref{prop-DLPmonotone} give $$\Delta(\cV) \geq \DLP_{H_{m+\epsilon}}(\nu)  \geq \DLP_{H_m}(\nu),$$ and therefore every sheaf used to compute $\delta_{m}^{\mus}(\nu)$ has discriminant at least $\DLP_{-K_{\F_e}}(\nu)$.
\end{proof}

In the anticanonically polarized del Pezzo case, our previous results can be interpreted as follows.

\begin{corollary}\label{cor-deltaDLP}
If $e=0$ or $1$ and $\nu\in \Pic(\F_e)\te \QQ$ then $$\delta_{1-\frac{e}{2}}^{\mus}(\nu) = \DLP_{-K_{\F_e}}(\nu).$$
\end{corollary}
\begin{proof}
By Corollary \ref{cor-deltaDLP01} it remains to show that $\delta_{1-\frac{e}{2}}^{\mus}(\nu) \leq \DLP_{-K_{\F_e}}(\nu)$.  Let $\bv = (r,\nu,\Delta)$ be any integral character with $\Delta > \DLP_{-K_{\F_e}}(\nu)$.  If $\epsilon>0$ is sufficiently small, then by Lemma \ref{lem-limitK} we have $$\Delta \geq \DLP_{-K_{\F_e}}(\nu) \geq \DLP_{-K_{\F_e}}^{< r}(\nu) = \DLP_{-K_{\F_e}+\epsilon F}^{<r}(\nu),$$ so Theorem \ref{thm-existence} shows there are $(-K_{\F_e}+\epsilon F)$-semistable sheaves of character $\bv$.  We have $\Delta>\frac{1}{2}$ by Corollary \ref{cor-K1/2}, so there are $\mu_{-K_{\F_e}+\epsilon F}$-stable sheaves of character $\bv$ by Propositions \ref{prop-ssIMPs} and \ref{prop-sIMPmus}.  Then there are $\mu_{-K_{\F_e}}$-stable sheaves of character $\bv$ by Corollary \ref{cor-KstabilityEasy}.  Therefore $\delta_{1-\frac{e}{2}}^{\mus}(\nu) \leq \Delta$.  We can choose the character $\bv$ such that $\Delta> \DLP_{-K_{\F_e}}(\nu)$ is as close to $\DLP_{-K_{\F_e}}(\nu)$ as we want, so this shows $\delta_{1-\frac{e}{2}}^{\mus}(\nu) \leq \DLP_{-K_{\F_e}}(\nu)$.
\end{proof}

\section{Harder-Narasimhan filtrations from Kronecker modules}\label{sec-HNKronecker}

We let $e=0$ or $1$ throughout this section.  We construct Chern characters $\bv \in K(\F_e)$ and polarizations $H_m$ such that the general sheaf $\cV\in \cP_{H_{\lceil m\rceil}}(\bv)$ has an $H_m$-Harder-Narasimhan filtration of length $2$ and neither factor is semiexceptional.  This is in direct contrast to the case of $\P^2$ (see Example \ref{ex-HNP2}) or an anticanonically polarized Hirzebruch surface (see Corollary \ref{cor-HNShape}). 

 Intuitively, we can construct such characters $\bv$ as follows.  First, take an exceptional collection $$\cE_1,\quad \cE_2,\quad \cE_3,\quad\cE_4$$ of length $4$, so that $\chi(\cE_i,\cE_j) = 0$ for $i>j$.  Let $\cK$ be a bundle constructed from $\cE_3$ and $\cE_4$ (say by taking extensions, kernels, or cokernels of direct sums of copies of $\cE_3$ and $\cE_4$) and let $\cL$ be a bundle  constructed from $\cE_1$ and $\cE_2$.  Such bundles can be viewed as arising from \emph{Kronecker modules}, which are representations of a Kronecker quiver, and have previously been studied by Dr\'ezet \cite{Drezet}, Karpov \cite{Karpov}, and others.  

We observe that $\chi(\cK,\cL) = 0$.  If we can furthermore arrange that $\cK$ and $\cL$ are $H_m$-stable and $\mu_{H_m}(\cK)$ is slightly larger than $\mu_{H_m}(\cL)$, then by Theorem \ref{thm-HNcriterion} the characters $\bk = \ch\cK$ and $\bl = \ch \cL$ will be the characters of the factors of the generic $H_m$-Harder-Narasimhan filtration of sheaves in $\cP_{H_{\lceil m\rceil}} ( \bv)$, where $\bv = \bk + \bl$.  If $\cK$ and $\cL$ both have moduli, we will have constructed the desired character $\bv$.

The main difficulty in the previous analysis is to determine when the bundles $\cK$ and $\cL$ are $H_m$-stable, but our study of stability intervals in the preceding sections makes this tractable.  Karpov previously studied the $-K_{\F_e}$-stability of many such bundles \cite{Karpov}, but for our purposes varying the polarization is crucial.  When combined with general results on stability intervals, our approach also gives a new proof of the $-K_{\F_e}$-stability of these bundles.  We carry out this program in the case where the starting exceptional collection is $$\OO_{\F_e}(-E-\ell F),\quad\OO_{\F_e},\quad \OO_{\F_e}(F),\quad \OO_{\F_e}(E-(\ell-1-e)F) \qquad (\ell \geq 3)$$ and where we construct $\cK$ as an extension and $\cL$ as a cokernel. Many of the arguments undoubtedly generalize to more arbitrary exceptional collections, but as the computations are already considerable we do not pursue this here.

In the final subsection of this section, we study how the parameters of our construction can be varied to both construct stable sheaves of certain characters and see that stable sheaves of certain characters cannot exist.  In other words, we compute the function $\delta_m^{\mus}(\nu)$ for some values of $m$ and $\nu$.  In contrast with the functions $\DLP^{<r}_{H_m}(\nu)$, which are locally constant in $m$ and jump at special values, we will see that it is possible for $\delta_m^{\mus}(\nu)$ to increase continuously as $m$ increases.  

\subsection{Stability of bundles from an inverse pair}\label{ssec-inversePair}

Let $k\geq 3-e$ be an integer, and let $N = 2(k-1)+e$ (so that $N\geq 3$). In this subsection, we consider the stability of bundles on $\F_e$ that arise as extensions \begin{equation}\tag{$\dagger$}\label{eqn-inverse} 0\to \OO_{\F_e} (E -kF)^{\oplus a} \to \cK\to \OO_{\F_e}^{\oplus b} \to 0\end{equation} coming from the inverse exceptional pair $(\OO_{\F_e}(E-kF),\OO_{\F_e})$.  Let $\bk = \ch \cK$.  We have \begin{align*}\hom(\OO_{\F_e},\OO_{\F_e}(E-kF)) &= 0 \\ \ext^1(\OO_{\F_e},\OO_{\F_e}(E-kF)) &= N \\ \ext^2(\OO_{\F_e},\OO_{\F_e}(E-kF)) &= 0\end{align*} so if $\cK'$ is another bundle $$0\to \OO_{\F_e}(E-kF)^{\oplus a'} \to \cK' \to \OO_{\F_e}^{\oplus b'}\to 0,$$ then we have $$\chi(\cK',\cK) = b'b+a'a-Nb'a.$$ In the special case $\cK'=\cK$, we get $$r(\cK)^2(1-2\Delta(\cK))=b^2+a^2-Nab.$$ If we let $\eta = \frac{b}{a}$, then the right hand side will be negative when $$\eta^2-N\eta+1 < 0.$$ Putting $$\psi_N = \frac{N + \sqrt{N^2-4}}{2},$$ the two roots of $\eta^2-N\eta+1=0$ are $\psi_N^{\pm 1}$, and we conclude the following.

\begin{lemma}\label{lem-Kronecker1/2}
We have $\Delta(\cK) > \frac{1}{2}$ if and only if $$\eta := \frac{b}{a} \in (\psi_N^{-1},\psi_N).$$
\end{lemma}

\begin{remark}
The number $\psi_N$ is irrational, since $N\geq 3$.  So, we can never have $\Delta(\cK) = \frac{1}{2}$.
\end{remark}

\begin{remark}
We can use left and right mutations on the exceptional pair $(\OO_{\F_e}(E-kF),\OO_{\F_e})$ to generate an infinite sequence $(\cE_i,\cE_{i+1})$ of exceptional bundles.  These can all be realized as general extensions $$0\to \OO_{\F_e} (E -kF)^{\oplus a} \to \cE_i\to \OO_{\F_e}^{\oplus b} \to 0$$ for appropriate $a,b$, and we will have $\eta \in [0,\psi_N^{-1}) \cup (\psi_N,N]$.    But as we have already studied the stability of exceptional bundles, in this section we are primarily interested in the case where $\cK$ has moduli, so $\Delta(\cK)  > \frac{1}{2}$.
\end{remark}

For polarizations on one side of the anticanonical polarization, the stability of $\cK$ is easy to analyze.

\begin{theorem}\label{thm-inverseInterval}
Suppose $\eta = \frac{b}{a} \in (\psi_N^{-1},\psi_N)$.  Then the stability interval $I_{\cK}$ of a general extension (\ref{eqn-inverse}) contains the interval $[1-\frac{e}{2},k)$, and $\cK$ is strictly $\mu_{H_k}$-semistable.
\end{theorem}
\begin{proof}
The line bundles $\OO_{\F_e}(E-kF)$ and $\OO_{\F_E}$ are both stable of $H_k$-slope $0$, so any bundle $\cK$ fitting as an extension as above is automatically strictly $\mu_{H_k}$-semistable.  They are then also $H_{k+1}$-prioritary by Proposition \ref{prop-ssPrior}.  If $\epsilon>0$ is small, then the characters of $\OO_{\F_e}(E-kF)^{\oplus a}$  and $\OO_{\F_e}^{\oplus b}$ are the characters of the $H_{k+\epsilon}$-Harder-Narasimhan quotients of a general sheaf in $\cP_{F}(\bk)$.  Therefore the general sheaf in $\cP_{F}(\bk)$ fits as an extension of $\OO_{\F_e}^{\oplus b}$ by $\OO_{\F_e}(E-kF)^{\oplus a}$.  To complete the proof, we need to show that the general such extension is $\mu_{H_{k-\epsilon}}$-stable.

Let $\bv = \ch \cK = (r,\nu,\Delta)$, and let $\epsilon>0$ be small.  Suppose we prove that $\Delta \geq \DLP^{<r}_{H_{k-\epsilon}}(\nu)$.  Then by monotonicity of $\DLP_H^{<r}(\nu)$ in the polarization (Proposition \ref{prop-DLPmonotone}) we have $\Delta \geq  \DLP^{<r}_{-K_{\F_e}+\epsilon F}(\nu).$  By Theorem \ref{thm-existence}, there are $(-K_{\F_e}+\epsilon F)$-semistable sheaves of character $\bv$, and by Lemma \ref{lem-Kronecker1/2} and Propositions \ref{prop-ssIMPs} and \ref{prop-sIMPmus} there are $\mu_{-K_{\F_e}+\epsilon F}$-stable sheaves of character $\bv$.  Then the general sheaf $\cK$ in $\cP_{F}(\bk)$ is both $\mu_{H_{k}}$-semistable and $\mu_{-K_{\F_e}+\epsilon F}$-stable.  Then it is $\mu_{H_{k-\epsilon}}$-stable, and Corollary \ref{cor-KstabilityEasy} shows  that $I_\cK$ contains $[1-\frac{e}{2},k)$.

In order to prove $\Delta \geq \DLP^{<r}_{H_{k-\epsilon}}(\nu),$ we let $\cV$ be an exceptional bundle involved in the computation of $\DLP^{<r}_{H_{k-\epsilon}}(\nu)$ and show that $\Delta \geq \DLP_{H_{k-\epsilon},\cV}(\nu)$.  There are three cases to consider, depending on whether $\mu_{H_k}(\cV)$ is positive, negative, or zero.  As $t\in [0,\epsilon)$ varies in a small interval there are only finitely many exceptional bundles of interest in computing $\DLP_{H_{k-t}}^{<r}(\nu)$ by Remark \ref{rem-max}.  Thus we can shrink $\epsilon$ as necessary to accommodate each exceptional bundle $\cV$.

\emph{Case 1:} $\mu_{H_k}(\cV) >0$.  In this case we assume $\epsilon$ is chosen small enough that $$\mu_{H_{k-\epsilon}}(\cV) > \mu_{H_{k-\epsilon}}(\OO_{\F_e}) > \mu_{H_{k-\epsilon}}(\OO_{\F_e}(E-kF))$$ Then since $\OO_{\F_e}$ and $\OO_{\F_e}(E-kF)$ are $\mu_{H_{k-\epsilon}}$-stable, we have $\chi(\cV,\OO_{\F_e})\leq 0$ and $\chi(\cV,\OO_{\F_e}(E-kF))\leq 0$. Therefore also $\chi(\cV,\cK) \leq 0$, and we find $\Delta \geq \DLP_{H_{k-\epsilon},\cV}(\nu)$.

\emph{Case 2:} $\mu_{H_k}(\cV) <0$.  This is handled by a dual argument.

\emph{Case 3:} $\mu_{H_k}(\cV) = 0$.  In this case the total slope of $\cV$ lies on the line spanned by $\nu(\OO_{\F_e})$ and $\nu(\OO_{\F_e}(E-kF))$.  Then there are $a',b'\in \Z$ such that the linear combination $$\bk' = a'\ch(\OO_{\F_e}(E-kF)) + b'\ch \OO_{\F_e}$$ has $r(\bk') = r(\cV)$ and $\nu(\bk') = \nu(\cV)$.  This character satisfies $$\chi(\OO_{\F_e}(E-(k-1)F),\bk') = 0,$$ and we have $$\mu_{H_{k-\epsilon}}(\OO_{\F_e}(E-(k-1)F))=1-\epsilon < -\frac{1}{2}K_{\F_e}\cdot H_{k-\epsilon},$$ so $\OO_{\F_e}(E-(k-1)F))$ is a line bundle involved in the computation of $\DLP_{H_{k-\epsilon}}(\nu(\cV))$ and the character $\bk'$ satisfies $$\Delta(\bk') \leq \DLP_{H_{k-\epsilon}}(\nu(\cV)).$$ Since $\cV$ is $H_{k-\epsilon}$-semistable it must have $\Delta(\cV) \geq \Delta(\bk')$.  Then $\Delta(\bk') < \frac{1}{2}$, so $\eta':= b'/a' \notin (\psi_N^{-1},\psi_N)$.  There are then two further cases to consider:
 
\emph{Case 3a:} $0 < \eta' < \psi_N^{-1},$ or $b'\leq 0$ and $a'>0$.  In this case $\nu(\cV)$ lies on the line spanned by $\nu(\OO_{\F_e})$ and $\nu(\OO_{\F_e}(E-kF))$, but it lies down and to the right of $\nu(\cK)$.  Therefore, $\mu_{H_{k-\epsilon}}(\cV) < \mu_{H_{k-\epsilon}}(\cK),$ and we need to verify $\chi(\cK,\cV) \leq 0$ to show that $\Delta \geq \DLP_{H_{k-\epsilon},\cV}(\nu)$.  But since $\Delta(\bk') \leq \Delta(\cV)$, we can use Riemann-Roch to estimate $$\chi(\cK,\cV)\leq \chi(\cK,\bk') = bb'+aa'-Nba'.$$ If $b'=0$, then we get $$\chi(\cV,\bk') = aa'-Nba' = aa'(1-N\eta),$$ which is negative since $\eta > \psi_{N}^{-1} > \frac{1}{N}.$ Instead suppose $b'\neq 0$.  Then $$\chi(\cK,\bk') = bb'+aa'-Nba' = aa'(\eta\eta'+1-N\eta).$$ Since $a'>0$, we have $\chi(\cK,\cV) \leq 0$ so long as $$\eta \eta' + 1 - N\eta \leq 0,$$ which holds when $$\eta' \leq N- \eta^{-1}.$$ But $\psi_N^{-1} < N-\eta^{-1}$, so $\eta' < \psi_N^{-1}$ gives the required inequality.

\emph{Case 3b:} $\psi_N < \eta',$ or $b'>0$ and $a'\leq 0$.  This time $\nu(\cV)$ lies on the line spanned by $\nu(\OO_{\F_e})$ and $\nu(\OO_{\F_e}(E-kF))$, but it lies up and to the left of $\nu(\cK)$.  Therefore, $\mu_{H_{k-\epsilon}}(\cV) > \mu_{H_{k-\epsilon}}(\cK)$, and we need to verify $\chi(\cV,\cK)\leq 0$ to show that $\Delta \geq \DLP_{H_{k-\epsilon},\cV}(\nu)$.  The argument is similar to Case 3a, so we omit it.
\end{proof}

\subsection{Stability of bundles from a regular pair}\label{ssec-regularPair} Let $\ell \geq 1$,  and let $M = 2(\ell+1) -e$, so that $M\geq 3$.  Here we instead consider the regular exceptional pair $(\OO_{\F_e}(-E-\ell F),\OO_{\F_e})$ and study the stability of general bundles $\cL$ fitting as cokernels
\begin{equation}\tag{$\ddagger$}\label{eqn-regular} 0 \to \OO_{\F_e}(-E-\ell F)^{\oplus c} \to \OO_{\F_e}^{\oplus d} \to \cL \to 0.\end{equation} Write $\bl = \ch \cL$.  We assume $r(\cL) \geq 2$; in this case, since $$\sHom(\OO_{\F_e}(-E-\ell F),\OO_{\F_e})\cong \OO_{\F_e}(E+\ell F)$$ is globally generated we find that $\cL$ is a vector bundle by \cite[Proposition 2.6]{HuizengaJAG}. We have \begin{align*}\hom(\OO_{\F_e}(-E-\ell F),\OO_{\F_e}) &= M \\ \ext^1(\OO_{\F_e}(-E-\ell F),\OO_{\F_e}) &= 0 \\ \ext^2(\OO_{\F_e}(-E-\ell F),\OO_{\F_e}) &= 0,\end{align*} and if $\cL'$ is another bundle given as a cokernel $$0\to \OO_{\F_e}(-E-\ell F)^{\oplus c'} \to \OO_{\F_e}^{\oplus d'} \to \cL' \to 0,$$ then we have $$\chi(\cL',\cL)  = c'c+d'd-Nc'd.$$ As in the previous subsection, we deduce the following fact about discriminants.

\begin{lemma}\label{lem-KroneckerRegular1/2}
We have $\Delta(\cL) > \frac{1}{2}$ if and only if $$\frac{d}{c} \in (\psi_M^{-1},\psi_M).$$
\end{lemma}

\begin{lemma}
Any bundle $\cL$ as above is $H_{\ell+1-e}$-prioritary, and the general sheaf in $\cP_{F}(\bl)$ is a cokernel as in (\ref{eqn-regular}).
\end{lemma}
\begin{proof}
The group $\Ext^2(\cL,\cL(-H_{\ell+1-e}))$ is a quotient of copies of $$\Ext^2(\cL,\OO_{\F_e}(-H_{\ell+1-e}))\cong \Hom(\OO_{\F_e}(-H_{\ell+1-e}),\cL(K_{\F_e}))\cong H^0(\cL(-E+(\ell-e-1)F)),$$ which vanishes since $$H^1(\OO_{\F_e}(-2E-(e+1)F))=0.$$ Let $S = \Hom(\OO_{\F_e}(-E-\ell F)^{\oplus c}, \OO_{\F_e}^{\oplus d}).$  Then the family of cokernels $\cL_s/S$ parameterized by $S$ is a complete family since $(\OO_{\F_e}(-E-\ell F)^{\oplus c},\OO_{\F_e}^{\oplus d})$ is an exceptional pair.  Indeed, the Kodaira-Spencer map $$T_sS \cong \Hom(\OO_{\F_e}(-E-\ell F)^{\oplus c},\OO_{\F_e}^{\oplus d}) \to \Ext^1(\cL_s,\cL_s)$$ factors as the composition of the natural maps $$\Hom(\OO_{\F_e}(-E-\ell F)^{\oplus c},\OO_{\F_e}^{\oplus d}) \to \Ext^1(\cL_s,\OO_{\F_e}^{\oplus d})\to \Ext^1(\cL_s,\cL_s),$$ and both maps are surjective.  
\end{proof}

In contrast with the case of an inverse pair, the stability of a general bundle $\cL$ depends on the exponents $c$ and $d$ (or in particular, on their ratio $d/c$).  However, for ratios $d/c$ in a particular range, if we increase the polarization, then the bundle $\cL$ is destabilized by the quotient line bundle $\OO_{\F_e}(F)$.  The destabilizing subbundle is a bundle coming from an inverse exceptional pair, which is why we studied that phenomenon first in Section \ref{ssec-inversePair}.

\begin{lemma}\label{lem-regularDestab}
Suppose $$\ell+2-\frac{e}{2} < \frac{d}{c} < \psi_M.$$  Then $$\mu_{-K_{\F_e}}(\cL) < \mu_{-K_{\F_e}}(\OO_{\F_e}(F))$$ and $$s:=\chi(\cL, \OO_{\F_e}(F)) > 0.$$ For $$m_\cL := \frac{d}{c}-\ell -1$$ we have $$\mu_{H_{m_\cL}}(\cL) = \mu_{H_{m_\cL}}(\OO_{\F_e}(F)).$$ Let $$\bk = \ch(\cL) - s \cdot\ch(\OO_{\F_e}(F)).$$  Then there are $\mu_{H_{m_\cL}}$-stable sheaves of character $\bk$.  More precisely, a general sheaf $\cK \in M_{H_{m_\cL}}(\bk)$ fits as an extension 
$$ 0 \to \OO_{\F_e}(E-(\ell -e)F)^{\oplus c}\to \cK \to \OO_{\F_e}(2F)^{\oplus (c(2\ell +2-e)-d)}\to 0,$$ and has a stability interval $I_\cK$ containing the interval $[1-\frac{e}{2}, \ell+2-e)$.
\end{lemma}
\begin{proof}
We compute 
$$ 
r(\cL) = d - c
\qquad \textrm{and} \qquad c_1(\cL) = c(E+\ell F),
$$
so $$\mu_{H_m}(\cL) = \frac{c(E+\ell F)(E+(e+m)F)}{d-c} = \frac{c(m+\ell)}{d-c}$$ and $\mu_{H_m}(\OO_{\F_e}(F)) = 1$.  Then $\mu_{H_m} (\cL) \leq \mu_{H_m}(\OO_{\F_e}(F))$ is equivalent to $$c(m+\ell) \leq d-c,$$ or, dividing by $c$, to $$\frac{d}{c} \geq m+\ell+1.$$ Equality holds for $m = m_\cL$, and we have $m_\cL > 1-\frac{e}{2}$ (corresponding to $-K_{\F_e}$) since $\frac{d}{c} > \ell+2-\frac{1}{2}e$.  So, we also have $\mu_{-K_{\F_e}}(\cL) < \mu_{-K_{\F_e}}(\OO_{\F_e}(F))$.

Next we compute $\chi(\cL,\OO_{\F_e}(F))$.  We have \begin{align*} s:=\chi(\cL,\OO_{\F_e}(F)) &= d \chi(\OO_{\F_e},\OO_{\F_e}(F)) - c \chi(\OO_{\F_e}(-E-\ell  F),\OO_{\F_e}(F))\\
&= 2d-c(2\ell+4-e).
\end{align*}
Dividing by $2c$, we find that $s>0$ since $\frac{d}{c} > \ell+2-\frac{1}{2}e.$

The character $\bk$ now satisfies 
\begin{align*}
r(\bk) &= d-c-s=c(2\ell+3-e) - d\\
c_1(\bk) &= c(E+\ell F)-sF = cE +(c(3\ell+4-e)-2d)F\\
\chi(\bk) &= d-2s = c(4\ell+8-2e)-3d.
\end{align*}
On the other hand, if $\cK$ is an extension $$0\to \OO_{\F_e}(E-(\ell-e)F)^{\oplus c} \to \cW \to \OO_{\F_e}(2F)^{\oplus (c(2\ell+2-e)-d)}\to 0,$$ then we easily verify $\ch \cK = \bk$.  Notice that the number $2\ell+2-e$ is both $M$ and the number $N$ of \S \ref{ssec-inversePair} applied to the inverse exceptional pair $(\OO_{\F_e}(E-(\ell+2-e)F),\OO_{\F_e})$.  The transformation $(d,c) \mapsto (c,Mc-d)$ takes pairs $(d,c)$ such that $\frac{d}{c}\in (\psi_M^{-1},\psi_M)$ to pairs with the same property.  Then since $\frac{d}{c}\in (\psi_M^{-1},\psi_M)$, we find that $\Delta(\bk) >\frac{1}{2}$.  By Theorem \ref{thm-inverseInterval}, the general $F$-prioritary sheaf $\cK$ of character $\bk$ is an extension as above, and has a stability interval $I_\cK$ containing $[1-\frac{e}{2},\ell+2-e)$.  But $\frac{d}{c} < \psi_M$ implies $m_\cL < \ell+2-e$, so the general such extension is $\mu_{H_{m_\cL}}$-stable.
\end{proof}

Now we can analyze how the bundles $\cL$ are destabilized as the polarization changes.  First, we recall a general lemma \cite[Lemma 6.1]{CoskunHuizengaNef} that is often useful for studying the stability of extensions.  Recall that a simple object in an abelian category $\cA$ is an object with no proper subobjects, and a semisimple object is a (finite) direct sum of simple objets.

\begin{lemma}[{\cite[Lemma 6.1]{CoskunHuizengaNef}}]\label{lem-abelianExt}
Let $\cA$ be an abelian category, and let $A,B\in \cA$.  Consider an extension $E$ of the form $$0\to A\to E\to B\to 0.$$ \begin{enumerate}
 \item If $A$ is semisimple, $B$ is simple, and $\Hom(E,A) = 0$, then any proper subobject of $E$ is a subobject of $A$.

\item If $A$ is simple, $B$ is semisimple, and $\Hom(B,E) = 0$, then any proper quotient  of $E$ is a quotient of $B$.
 \end{enumerate}
\end{lemma}
\begin{proof}
Part (1) is proved in \cite[Lemma 6.1]{CoskunHuizengaNef}.  Part (2) is dual to part (1).
\end{proof}

\begin{theorem}\label{thm-regularInterval}
Suppose $$\ell+2-\frac{e}{2} < \frac{d}{c} < \psi_M.$$  Then the general cokernel $$0 \to \OO_{\F_e}(-E-\ell F)^{\oplus c} \to \OO_{\F_e}^{\oplus d} \to \cL\to 0$$ fits as an extension $$0\to \cK \to \cL \to \OO_{\F_e}(F)^{\oplus s} \to 0,$$ where $s = \chi(\cL,\OO_{\F_e}(F))$ and $\cK$ is a general bundle that fits as an extension $$0\to \OO_{\F_e}(E-(\ell-e)F)^{\oplus c} \to \cK \to \OO_{\F_e}(2F)^{\oplus( c(2\ell+2-e)-d)}\to 0.$$ The stability interval $I_\cL$ is an open interval that contains $[1-\frac{e}{2},m_\cL)$, where $m_\cL = \frac{d}{c}-\ell-1$, and $\cL$ is strictly $\mu_{H_{m_\cL}}$-semistable.
\end{theorem}
\begin{proof}
If $\epsilon>0$ is small, then by Lemma \ref{lem-regularDestab} and Theorem \ref{thm-HNcriterion} the characters $\bk$ and $s \cdot \ch \OO_{\F_e}(F)$ are the characters of the generic $H_{m_\cL+\epsilon}$-Harder-Narasimhan filtration of sheaves in $\cP_{F}(\bl)$.  Therefore the generic sheaf in $\cP_{F}(\bl)$ fits as an extension $$0\to \cK \to \cL\to \OO_{\F_e}(F)^{\oplus s}\to 0$$ for some $\cK \in M_{H_{m_\cL+\epsilon}}(\bk)$.  The general such $\cK$ also fits as an extension $$0\to \OO_{\F_e}(E-(\ell-e)F)^{\oplus c} \to \cK \to \OO_{\F_e}(2F)^{\oplus( c(2\ell+2-e)-d)}\to 0.$$  Since the general sheaf in $\cP_{F}(\bl)$ is a cokernel as in (\ref{eqn-regular}), we conclude that a general cokernel $\cL$ fits as an extension of the required form.

If $\cK\in M_{H_{m_\cL}}(\bk)$ is general (hence $\mu_{H_{m_\cL}}$-stable), then any bundle $\cL$ fitting as an extension $$0\to \cK\to \cL\to \OO_{\F_e}(F)^{\oplus s}\to 0$$ is automatically $\mu_{H_{m_\cL}}$-semistable.  We will show that $\cL$ is actually $\mu_{H_{m_\cL-\epsilon}}$-stable for any sufficiently small $\epsilon > 0$; then since $\cL$ is a general sheaf in $\cP_F(\bl)$, we find that the stability interval $I_\cL$ contains $[1-\frac{e}{2},m_\cL)$ by Corollary \ref{cor-KstabilityEasy}.

Now let us show that $\cL$ is $\mu_{H_{m_\cL-\epsilon}}$-stable for any sufficiently small $\epsilon > 0$.  Since the walls for slope-semistability are locally finite, we can let $\epsilon>0$ be small enough that if $\cL\onto \cQ$ is a quotient such that the total slopes $\nu(\cL)$ and $\nu(\cQ)$ are not proportional, and $0< \epsilon' < \epsilon$, then $\mu_{H_{m_0-\epsilon'}}(\cL) \neq \mu_{H_{m_0-\epsilon'}}(\cQ)$.  We claim that $\cL$ is $\mu_{H_{m_0-\epsilon'}}$-stable for all such $0 < \epsilon' < \epsilon$.  Let $\cL\onto \cQ$ be a proper quotient and suppose $\mu_{H_{m_0-\epsilon'}}(\cL) \geq \mu_{H_{m_0-\epsilon'}}(\cQ)$. Then it follows from our choice of $\epsilon$ that $\mu_{H_{m_0}}(\cL) = \mu_{H_{m_0}}(\cQ)$.  Let $\cA \subset \Coh(\F_e)$ be the full abelian subcategory of $\mu_{H_{m_0}}$-semistable sheaves with the same $H_{m_0}$-slope as $\cL$. In this category, $\cW$ is a simple object (by Lemma \ref{lem-regularDestab}) and $\OO_{\F_e}(F)^{\oplus s}$ is a semisimple object.  The vanishing $\Hom(\OO_{\F_e}(F),\cL)=0$ follows immediately from the expression of $\cL$ as a cokernel (\ref{eqn-regular}).  Therefore, by Lemma \ref{lem-abelianExt} (2), $\cQ$ is a direct sum of copies of $\OO_{\F_e}(F)$.  But $\mu_{H_{m_0-\epsilon}}(\cL) < \mu_{H_{m_0-\epsilon'}}(\OO_{\F_e}(F)),$ contradicting our assumption that $\mu_{H_{m_0-\epsilon'}}(\cL) \geq \mu_{H_{m_0-\epsilon'}}(\cQ)$.  Therefore $\cL$ is $\mu_{H_{m_0-\epsilon'}}$-stable.
\end{proof}

\subsection{Harder-Narasimhan filtrations from Kronecker modules}\label{ssec-HNKronecker} We can now combine the previous two constructions to find characters $\bv$ such that the generic Harder-Narasimhan filtration has two factors, each of which corresponds to a space of Kronecker modules.  Let \begin{align*} \ell & \geq 3 \\ k &= \ell - e \\ N &= 2(k-1) +e \\ M &= 2(\ell+1)-e\end{align*} and consider bundles $\cK$ and $\cL$ defined by sequences $$0\to \OO_{\F_e}(E-(k-1)F)^{\oplus b} \to \cK \to \OO_{\F_e}(F)^{\oplus a} \to 0$$$$0\to \OO_{\F_e}(-E-\ell F)^{\oplus c} \to \OO_{\F_e}^{\oplus d} \to \cL \to 0.$$ So, $\cK(-F)$ is a bundle coming from an inverse pair as in \S\ref{ssec-inversePair}, and $\cL$ is a bundle coming from a regular pair as in \S\ref{ssec-regularPair}.  Assume the exponents satisfy \begin{align*}\psi_N^{-1} &< \frac{b}{a} < \psi_N \\  2\ell-e+1 &< \frac{d}{c} < \psi_M.
\end{align*} (Note that the lower bound on $\frac{d}{c}$ is stronger than the inequality $\frac{d}{c} > \ell+2-\frac{e}{2}$ in Theorem \ref{thm-regularInterval}, since $\ell \geq 3$.)  The line bundles $$\OO_{\F_e}(-E-\ell F),\quad\OO_{\F_e},\quad \OO_{\F_e}(F),\quad \OO_{\F_e}(E-(\ell-1-e)F)$$ form an exceptional collection, from which we get the orthogonality $\chi(\cK,\cL) = 0$.  By Theorems \ref{thm-inverseInterval} and \ref{thm-regularInterval}, if $m_\cK = k$ and $m_\cL = \frac{d}{c}-\ell-1$, then the stability intervals satisfy \begin{align*} I_\cK &\supset [1-\frac{e}{2},m_\cK) \\ I_{\cL} & \supset [1- \frac{e}{2},m_\cL).\end{align*} The inequality $\frac{d}{c} > 2\ell - e+1$ implies $m_\cK < m_\cL.$
\begin{theorem}\label{thm-intervalKronecker}
With the notation in this subsection, consider a general extension $\cV$ of the form $$0\to \cK \to \cV \to \cL\to 0.$$ There is a unique number $m_\cV\in (1-\frac{e}{2},k)$ such that $\mu_{H_{m_\cV}}(\cK) = \mu_{H_{m_\cV}}(\cL)$.  The stability interval $I_\cV$ of $\cV$ is an open interval containing $[1-\frac{e}{2},m_\cV)$, and $\cV$ is strictly $\mu_{H_{m_\cV}}$-semistable.

Let $\bv  =\ch \cV$, $\bk = \ch \cK$, and $\bl = \ch \cL$.  The stack $\cP_{H_{\lceil m_\cV\rceil+1}}(\bv)$ is nonempty.  If $\epsilon>0$ is small, then for $m\in (m_\cV,m_\cV+\epsilon)$ the general sheaf in $\cP_{H_{\lceil m_\cV\rceil+1}}(\bv)$ has a generic $H_m$-Harder-Narasimhan filtration with quotients of characters $\bk$ and $\bl$.
\end{theorem}
\begin{proof}
To establish the existence of $m_\cV$, we only need to verify the inequalities
\begin{align*}
\mu_{-K_{\F_e}}(\cK) &< \mu_{-K_{\F_e}}(\cL)\\
\mu_{H_k}(\cK) &> \mu_{H_k}(\cL).
\end{align*}
For $m\in \QQ$ we compute
\begin{align*}
\mu_{H_m}(\cK) &= \frac{1-\frac{b}{a}(k-1-m)}{1+\frac{b}{a}}\\
\mu_{H_m}(\cL) &= \frac{m+\ell}{\frac{d}{c}-1}.\\
\end{align*}
For $m =k$, we get \begin{align*} \mu_{H_k}(\cK) &=  1 \\ \mu_{H_k}(\cL) &= \frac{2\ell - e}{\frac{d}{c}-1},\end{align*} and the inequality $\mu_{H_k}(\cK) > \mu_{H_k}(\cL)$ follows from $\frac{d}{c} > 2\ell-e+1$.  On the other hand, taking $m = 1-\frac{e}{2}$, the slope $\mu_{-K_{\F_e}}(\cL)$ will be as small as possible when $\frac{d}{c}$ is as large as possible, and $\mu_{-K_{\F_e}}(\cK)$ will be as large as possible when $\frac{b}{a}$ is as small as possible.  So, we can verify the inequality $\mu_{-K_{\F_e}}(\cK) < \mu_{-K_{\F_e}}(\cL)$ when $\frac{d}{c} = \psi_M$ and $\frac{b}{a} = \psi_N^{-1}$.  In this case, after simplifying we get \begin{align*} \mu_{-K_{\F_e}}(\cL) - \mu_{-K_{\F_e}}(\cK) &=\frac{\ell+1-\frac{e}{2}}{\psi_M-1}-\frac{1 - \psi_N^{-1}(\ell-2-\frac{e}{2})}{1+\psi_N^{-1}} \\&= \frac{\left(\sqrt{2\ell-e}\cdot\left(\sqrt{2 \ell + 4 - e} - \sqrt{2\ell -4 - e}\right)-4\right)^2}{4\sqrt{2\ell-e}\cdot \left(\sqrt{2\ell+4-e}-\sqrt{2\ell-4-e}\right)},
\end{align*}
which is visibly positive.

Since $1-\frac{e}{2} < m_\cV < m_\cK < m_\cL$, the general sheaves of the form of $\cK$ and $\cL$ are both $\mu_{H_{m_\cV}}$-stable, and so $\cV$ is strictly $\mu_{H_{m_\cV}}$-semistable.  We can now show that $\cV$ is $\mu_{H_{m_\cV-\epsilon}}$-stable for any sufficiently small $\epsilon > 0$ by an argument similar to the proof of Theorem \ref{thm-regularInterval}, and then it similarly follows that $\cV$ is $\mu_{H_m}$-stable for any $m\in [1-\frac{e}{2},m_\cV)$; the only difference is in the proof of the vanishing $\Hom(\cL,\cV) = 0$.  

We claim that if $\cV$ is not a split extension then $\Hom(\cL,\cV) = 0$.  Consider the exact sequence $$\Hom(\cL,\cK) \to \Hom(\cL,\cV) \to \Hom(\cL,\cL) \to \Ext^1(\cL,\cK).$$
By assumption $\CC \cong \Hom(\cL,\cL)\to \Ext^1(\cL,\cK)$ is injective.  Also, $\Hom(\cL,\cK) = 0$ since there are polarizations where $\cL$ and $\cK$ are both slope-stable and $\cL$ has greater slope than $\cK$.  Therefore, $\Hom(\cL,\cV) = 0$.  

To guarantee that there are non-split extensions of $\cL$ by $\cK$, we show $\chi(\cL,\cK) < 0$.  We already observed the vanishing $\Hom(\cL,\cK)=0$, and $\Ext^2(\cL,\cK) = 0$ can be proved by a similar argument and Serre duality.  So, $\chi(\cL,\cK) \leq 0$ and we need to prove the inequality is strict.  Suppose $\chi(\cL,\cK) = 0$.  Since $\chi(\cK,\cL) = 0$, Riemann-Roch gives $$P(\nu(\cL)-\nu(\cK)) = P(\nu(\cK)-\nu(\cL)).$$
But $$P(\nu) - P(-\nu) = -K_{\F_e}\cdot \nu,$$ so this would require $$-K_{\F_e} \cdot (\nu(\cL)-\nu(\cK))=0.$$  This contradicts $\mu_{-K_{\F_e}}(\cL) > \mu_{-K_{\F_e}}(\cK)$.  Therefore $\chi(\cL,\cK) < 0$, there are non-split extensions $\cV$ of $\cL$ by $\cK$, we have $\Hom(\cL,\cV) = 0$, and the stability interval $I_\cV$ of $\cV$ contains $(1-\frac{e}{2},m_\cV)$.

Finally, the stack $\cP_{H_{\lceil m_\cV\rceil +1}}(\bv)$ is nonempty by Proposition \ref{prop-ssPrior}.  Let $\epsilon>0$ be small enough that $$\mu_{H_{m_\cV+\epsilon}}(\cK) - \mu_{H_{m_\cV+\epsilon}}(\cL) < 1.$$ Then for $m\in (m_\cV,m_{\cV+\epsilon})$, the characters $\bk$ and $\bl$ satisfy all the criteria of Theorem \ref{thm-HNcriterion}, so they are the characters of the $H_m$-Harder-Narasimhan filtration of the general sheaf in $\cP_{H_{\lceil m_\cV\rceil +1}}(\bv)$.
\end{proof}

\subsection{Sharp Bogomolov inequalities from Kronecker modules}

As in \S\ref{ssec-HNKronecker}, we continue to consider extensions $\cV$ of the form \begin{eqnarray*}&0\to \cK \to \cV\to \cL\to 0\\&0\to \OO_{\F_e}(E-(k-1)F)^{\oplus b} \to \cK \to \OO_{\F_e}(F)^{\oplus a} \to 0\\&0\to \OO_{\F_e}(-E-\ell F)^{\oplus c} \to \OO_{\F_e}^{\oplus d} \to \cL \to 0\end{eqnarray*} where the exponents $a,b,c,d$ satisfy 
\begin{align*}\psi_N^{-1} &< \frac{b}{a} < \psi_N \\  2\ell-e+1 &< \frac{d}{c} < \psi_M.
\end{align*} The ratios $b/a$ and $d/c$ determine and are determined by the total slopes $\nu(\cK)$ and $\nu(\cL)$ respectively.  As $b/a$ increases along the interval $(\psi_N^{-1},\psi_N)$, the total slope $\nu(\cK) = xE+yF = (x,y)$ travels along the open line segment $\overline{P_1P_2}$ of slope $-k$ with endpoints
\begin{align*}P_1&= \left(\frac{\psi_N^{-1}}{1+\psi_N^{-1}},\frac{1-\psi_N^{-1}(k-1)}{1+\psi_N^{-1}}\right)\\
P_2&= \left(\frac{\psi_N}{1+\psi_N},\frac{1-\psi_N(k-1)}{1+\psi_N}\right).
\end{align*}
Rational ratios $b/a$ correspond to rational total slopes on this segment.
As $d/c$ decreases along the interval $(2\ell-e+1,\psi_M)$ the total slope $\nu(\cL)$ travels along the open line segment $\overline{P_3P_4}$ of slope $\ell$ with endpoints
\begin{align*}
P_3 &= \left(\frac{1}{\psi_M-1},\frac{\ell}{\psi_M-1}\right)\\
P_4 &= \left(\frac{1}{2\ell-e},\frac{\ell}{2\ell-e}\right),
\end{align*}
again with rational points corresponding to rational points.
The point $P_1$ lies on the line segment $\overline{P_4P_2}$.  Let $R$ be the triangular region with vertices $P_2$, $P_3$, $P_4$.

\begin{theorem}\label{thm-deltaKronecker}
Let $\nu = x_0 E+y_0 F = (x_0,y_0)$ be a rational total slope in the triangular region $R$.  Suppose $m\in \QQ$ is a number such that the line through $\nu$ of slope $-m$ intersects the open line segments $\overline{P_1P_2}$ and $\overline{P_3P_4}$.  Then \begin{align*}\delta_m^{\mus}(\nu)  &= -\frac{e}{2}x_0^2+x_0y_0+\frac{1}{k+\ell}y_0+\left(\ell-\frac{1}{2}-\frac{e}{2}-\frac{e}{2(k+\ell)}\right)x_0\\
 & \qquad + \frac{(m-k)(y_0-\ell x_0)}{(k+\ell)^2\left(y_0+mx_0-\frac{m+\ell}{k+\ell}\right)}.\end{align*}
\end{theorem}
Note that the geometry of $R$ implies that $1-\frac{e}{2}<m<k$.

\begin{example}\label{ex-triangle}
In Figure \ref{fig-triangle} we sketch the setup of the theorem in one of the simplest cases.  We take $e=1$ and $\ell = 3$, and let $\nu = (\frac{2}{13}, \frac{6}{13})$ and $m = \frac{12}{7}$.  Then the line through $\nu$ of slope $-m$ meets $\overline{P_1P_2}$ at $\nu_1 = (\frac{1}{2},0)$ and meets $\overline{P_3P_4}$ at $\nu_2 = (\frac{2}{11},\frac{6}{11})$.  See Example \ref{ex-KroneckerF1} for more analysis of this example.  For appropriate choices of $m$, the theorem can be applied to any slope $\nu$ lying in the triangle with vertices $P_2,P_3,P_4$.
\begin{figure}[p]
\begin{center}
\includegraphics[scale=0.55,bb=0 0 7.82in 15.39in]{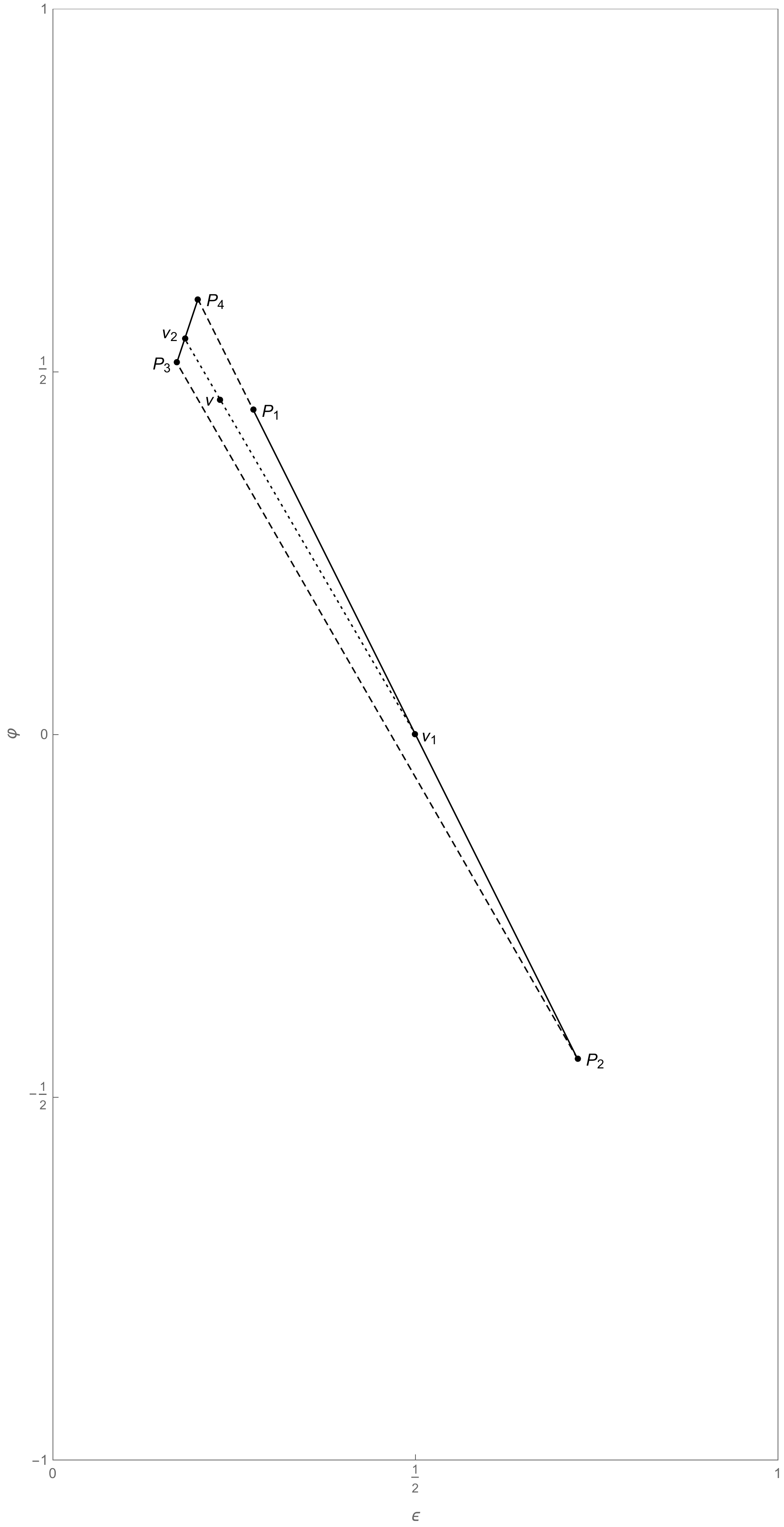}
\end{center}
\caption{The setup of Theorem \ref{thm-deltaKronecker}.  See Example \ref{ex-triangle}.}\label{fig-triangle}
\end{figure}
\end{example}

\begin{remark}
Since $e\in \{0,1\}$, the formula for $\delta_m^{\mus}(\nu)$ can be made more transparent in each case.  If $e = 0$ (so $\ell = k$), we have $$\delta_m^{\mus}(\nu) = x_0y_0+\frac{1}{2\ell}y_0+\left(\ell-\frac{1}{2}\right)x_0+\frac{(m-\ell)(y_0-\ell x_0)}{4\ell^2\left(y_0+mx_0-\frac{m+\ell}{2\ell}\right)}.$$ Since $m<\ell$ and $(x_0,y_0)\in R$, the final term is negative.  So in particular, we have a bound $$\delta_m^{\mus}(\nu) < x_0y_0+\frac{1}{2\ell}y_0+\left(\ell-\frac{1}{2}\right)x_0.$$ A similar analysis can be made when $e=1$.
\end{remark}
\begin{proof}[Proof of Theorem \ref{thm-deltaKronecker}]
Let $L_\cK$ be the line through $P_1$ and $P_2$, let $L_\cL$ be the line through $P_3$ and $P_4$, and let $L_\nu$ be the line through $\nu$ with slope $-m$.  These lines have equations
\begin{align*}
L_\cK &: y=-kx+1\\
L_\cL &: y=\ell x\\
L_\nu &: y-y_0=-m(x-x_0).
\end{align*}
Let $\nu_1=(x_1,y_1)$ be the point of intersection $L_\cK \cap L_\nu$ and let $\nu_2=(x_2,y_2)$ be the point of intersection $L_{\cL} \cap L_\nu$.  We compute
\begin{align*}
x_1 &= \frac{y_0+mx_0-1}{m-k}\\
x_2 &= \frac{y_0+mx_0}{m+\ell},
\end{align*}
and $y_1$ and $y_2$ are readily determined (although we won't need them).
Next we find the ratios $b/a$ and $c/d$ such that the corresponding bundles $\cK$ and $\cL$ have $\nu(\cK) = \nu_1$ and $\nu(\cL) = \nu_2$; it is enough for the $x$-components of the slopes to agree, so comparing the $x$-components we need \begin{align*}
\frac{\frac{b}{a}}{1+\frac{b}{a}} &= \frac{y_0+mx_0-1}{m-k}\\
\frac{1}{\frac{d}{c}-1} &= \frac{y_0+mx_0}{m+\ell},
\end{align*}
from which we get
\begin{align*}
\frac{b}{a} &= - \frac{y_0+mx_0-1}{y_0+m x_0+k-m-1}\\
\frac{d}{c} &= \frac{y_0+mx_0+m+\ell}{y_0+mx_0}.
\end{align*}
Pick arbitrary positive integers $a,b,c,d$ such that the ratios $b/a$ and $d/c$ take the above values, and let $\cK$ and $\cL$ be the corresponding bundles, with characters $\bk$ and $\bl$.  Normalize the ranks to be $1$ by considering the characters $$\bk' = \frac{\ch (\bk)}{r(\bk)} \qquad \bl' = \frac{\ch (\bl)}{r(\bl)}.$$ We compute \begin{align*}
\chi(\bk') &= \frac{\chi(\bk)}{r(\bk)} = \frac{2-\frac{b}{a}(2k-4+e)}{1+\frac{b}{a}}\\
\chi(\bl') &= \frac{\chi(\bl)}{r(\bl)} = \frac{\frac{d}{c}}{\frac{d}{c}-1}
\end{align*}
Let $0<\lambda<1$ be the number such that $\nu = \lambda \nu_1 +(1-\lambda)\nu_2$; again this equality will hold if equality holds for the $x$-components, and we compute 
$$
\lambda = \frac{(k-m)(y_0-\ell x_0)}{(k+\ell)(y_0+mx_0)-m-\ell}.
$$
Now consider the character $\bv' = \lambda \bk'+(1-\lambda) \bl'$ which has total slope $\nu(\bv') = \nu$.  Then substituting for $\lambda$, $\chi(\bk')$, $\chi(\bl')$, $\frac{b}{a}$, $\frac{d}{c}$, and performing considerable simplification we get \begin{align*}\chi(\bv') &= \lambda \chi(\bk') + (1-\lambda)\chi(\bl')\\
&= \left(1-\frac{1}{k+\ell}\right)y_0-\left(\ell-\frac{3}{2}-\frac{e}{2(k+\ell)}\right)x_0+\left(1+\frac{1}{k+\ell}\right)\\&\qquad-\frac{(m+\ell)(y_0+k x_0 - 1)}{(k+\ell)^2\left(y_0+mx_0-\frac{m+\ell}{k+\ell}\right)}.
\end{align*}
By Riemann-Roch,
\begin{align*}
\Delta(\bv') &= P(\nu(\bv')) - \chi(\bv')\\
&= (x_0+1)\left(y_0+1-\frac{e}{2}x_0\right)-\chi(\bv')\\
&= -\frac{e}{2}x_0^2+x_0y_0+\frac{1}{k+\ell}y_0+\left(\ell-\frac{1}{2}-\frac{e}{2}-\frac{e}{2(k+\ell)}\right)x_0-\frac{1}{k+\ell}\\
 & \qquad + \frac{(m+\ell)(y_0+kx_0-1)}{(k+\ell)^2\left(y_0+mx_0-\frac{m+\ell}{k+\ell}\right)}\\
 &= -\frac{e}{2}x_0^2+x_0y_0+\frac{1}{k+\ell}y_0+\left(\ell-\frac{1}{2}-\frac{e}{2}-\frac{e}{2(k+\ell)}\right)x_0\\
 & \qquad + \frac{(m-k)(y_0-\ell x_0)}{(k+\ell)^2\left(y_0+mx_0-\frac{m+\ell}{k+\ell}\right)}.
\end{align*}

We claim that $\delta_m^{\mus}(\nu) = \Delta(\bv')$.  Pick an integer $s$ such that the characters $s\lambda \bk'$ and $s(1-\lambda)\bl'$ are integer multiples of $\bk$ and $\bl$ respectively.  Replace $\bk$ with $s\lambda \bk'$ and $\bl$ with $s(1-\lambda)\bl'$.  Then the character $\bv = \bk + \bl$ is integral and a multiple of $\bv'$, so $\Delta(\bv) = \Delta(\bv')$.  By construction, $m$ is the number $m_\cV$ of Theorem \ref{thm-intervalKronecker}.  By Theorem \ref{thm-intervalKronecker}, it follows that the general sheaf $\cV\in \cP_{H_{\lceil m \rceil+1}}(\bv)$ has a stability interval that contains interval $[1-\frac{e}{2},m)$, is strictly $\mu_{H_m}$-semistable, and is not $\mu_{H_{m+\epsilon}}$-semistable for any $\epsilon>0$.

Furthermore, if we let $t\in \QQ$ be sufficiently close to $m$, we can carry out the above discussion using $t$ in place of $m$.  Let $\Delta_t$ be the discriminant of the corresponding character $\bv'$.  Then by our formula for $\Delta(\bv')$, we see that $\Delta_t$ is a continuous function of $t$ for $t$ close enough to $m$.  By considering $t = m+\epsilon$ slightly larger than $m$, we find that there are $\mu_{H_m}$-stable sheaves of total slope $\nu$ and discriminant $\Delta_{m+\epsilon}$, so $\delta_m^{\mus}(\nu) \leq \Delta_{m+\epsilon}$.  On the other hand, by considering $t = m-\epsilon$ slightly smaller than $m$, we find that there are no $\mu_{H_m}$-stable sheaves of total slope $\nu$ and discriminant $\Delta_{m-\epsilon}$.  Therefore by Theorem \ref{thm-deltaSurface}, we have $\Delta_{m-\epsilon} \leq \delta_m^{\mus}(\nu)$, and $$\Delta_{m-\epsilon} \leq \delta_m^{\mus}(\nu) \leq \Delta_{m+\epsilon}.$$ Letting $\epsilon\to 0$ we get $\delta_m^{\mus}(\nu) = \Delta_m$, which completes the proof.
\end{proof}

\begin{example}\label{ex-KroneckerF0}
Let $e=0$ and take $\ell = 3$.  Consider the total slope $\nu = \frac{1}{5} E + \frac{1}{3} F$, and let $m = \frac{25}{9}$.  Then the point $(\frac{1}{5},\frac{1}{3})$ lies in the triangular region $R$, and the line of slope $-25/9$ through $(\frac{1}{5},\frac{1}{3})$ intersects the open line segments $\overline{P_1P_2}$ and $\overline{P_3P_4}$ at the points $(\frac{1}{2},-\frac{1}{2})$ and $(\frac{2}{13},\frac{6}{13})$, respectively.  (In fact, these rational points on these open line segments have the smallest possible denominators.)  In the notation of the proof, the ratios $\frac{b}{a}$ and $\frac{d}{c}$ satisfy $\frac{b}{a} = 1$ and $\frac{d}{c} = \frac{15}{2}$.  Consider general extensions of the form 
\begin{eqnarray*} &0\to \cK \to \cV\to \cL\to 0\\ &0\to \OO_{\F_0}(E-2F) \to \cK \to \OO_{\F_0}(F) \to 0\\ &0\to \OO_{\F_0}(-E-3 F)^{\oplus 2} \to \OO_{\F_0}^{\oplus 15} \to \cL \to 0,\end{eqnarray*} and let $\bv = \ch \cV$, $\bk = \ch \cK$, $\bl = \ch \cL$. Then writing characters as $(r,\nu,\Delta)$, we have 
\begin{align*}
\bv &= \left(15,\frac{1}{5}{E}+\frac{1}{3}F,\frac{3}{5}\right)\\
\bk &= \left(2,\frac{1}{2}E-\frac{1}{2}F,\frac{3}{4}\right)\\
\bl &= \left(13,\frac{2}{13}E+\frac{6}{13}F,\frac{90}{169}\right)
\end{align*}
The stability interval $I_\cV$ is an open interval containing $[1,\frac{25}{9})$, and $\cV$ is strictly $\mu_{H_{25/9}}$-semistable.  For $\epsilon > 0$ small, the generic sheaf in $\cP_{H_3}(\bv)$ has an $H_{\frac{25}{9}+\epsilon}$-Harder-Narasimhan filtration with quotients of characters $\bk$ and $\bl$.  We have $$\delta_{25/9}^{\mus}\left(\frac{1}{5}E+\frac{1}{3}F\right) = \frac{3}{5}=0.6.$$ On the other hand, considering exceptional bundles of rank smaller than $15$ we can use a computer to compute $$\DLP_{H_{25/9}}^{<15}\left(\frac{1}{5}E+\frac{1}{3}F\right)= \frac{19}{35}\approx 0.543.$$ If we increase the polarization a little bit we will still have $$\DLP_{H_{\frac{25}{9}+\epsilon}}^{<15}\left(\frac{1}{5}E+\frac{1}{3}F\right) < \frac{3}{5} <  \delta_{\frac{25}{9}+\epsilon}^{\mus}\left(\frac{1}{5} E+\frac{1}{3}F\right).$$  There are no $H_{\frac{25}{9}+\epsilon}$-semistable sheaves of character $\bv$, so the exceptional bundles of rank less than $r(\cV)$ are not sufficient to control the existence of $\mu_{H_{\frac{25}{9}+\epsilon}}$-stable sheaves $\cV$ on $\F_0$.
\end{example}

\begin{example}\label{ex-KroneckerF1}
Let $e=1$ and take $\ell = 3$, so $k=2$.  Then $N = 3$ and $M=7$.  Consider the total slope $\nu = \frac{3}{13}E+\frac{6}{13}F$, and let $m = \frac{12}{7}$.  Then the point $(\frac{3}{13},\frac{6}{13})$ lies in the triangular region $R$, and the line of slope $-m$ meets the open line segments $\overline{P_1P_2}$ and $\overline{P_3P_4}$ at the points $(\frac{1}{2},0)$ and $(\frac{2}{11},\frac{6}{11})$, respectively.  Then $\frac{b}{a} = 1$ and $\frac{d}{c} = \frac{13}{2}$.  We consider general extensions of the form
\begin{eqnarray*} &0\to \cK \to \cV\to \cL\to 0\\ &0\to \OO_{\F_1}(E-F) \to \cK \to \OO_{\F_1}(F) \to 0\\ &0\to \OO_{\F_1}(-E-3 F)^{\oplus 2} \to \OO_{\F_1}^{\oplus 13} \to \cL \to 0,\end{eqnarray*}
and let $\bv = \ch \cV$, $\bk = \ch \cK$, $\bl = \ch \cL$.  Writing characters as $(r,\nu,\Delta)$, we get 
\begin{align*}
\bv &= \left(13,\frac{3}{13}E+\frac{6}{13}F,\frac{98}{169}\right)\\
\bk &= \left(2,\frac{1}{2}E,\frac{5}{8}\right)\\
\bl &= \left(11,\frac{2}{11}E+\frac{6}{11}F,\frac{65}{121}\right).
\end{align*}
Then the stability interval $I_\cV$ is an open interval containing $[\frac{1}{2},\frac{12}{7})$, and $\cV$ is strictly $\mu_{H_{12/7}}$-semistable.  For $\epsilon > 0$ small, the generic sheaf in $\cP_{H_2}(\bv)$ has an $H_{\frac{12}{7}+\epsilon}$-Harder-Narasimhan filtration with quotients of characters $\bk$ and $\bl$.  Finally, $$\delta_{12/7}^{\mus}\left(\frac{3}{13}E+\frac{6}{13}F\right)=\frac{98}{169}\approx 0.580,$$ but computer calculations show $$\DLP^{<13}_{H_{12/7}}\left(\frac{3}{13}E+\frac{6}{13}F\right) = \frac{523}{1014}\approx 0.516.$$  So again, $$\DLP_{H_{12/7}}^{<13}\left(\frac{3}{13}E+\frac{6}{13}F\right) < \delta_{12/7}^{\mus}\left(\frac{3}{13}E+\frac{6}{13}F\right),$$ and similar conclusions follow.
\end{example}

We conclude the section with the following more qualitative fact.

\begin{corollary}
If $\nu$ and $m$ are as in Theorem \ref{thm-deltaKronecker} and $m$ is generic, then for any $r\geq 2$ we have $$\DLP_{H_m}^{<r}(\nu) < \delta_m^{\mus}(\nu).$$
\end{corollary}
\begin{proof}
Let $m'\in \QQ$ be close to $m$.  As $m'$ varies, $\DLP_{H_{m'}}^{<r}(\nu)$ only changes values at special values.  Since $m$ is generic, $\DLP_{H_{m'}}^{<r}(\nu)$ remains constant if we vary $m'$ by a little bit.  But for generic $m'$ we have $\delta_{m'}^{\mus}(\nu) \geq \DLP_{H_{m'}}^{<r}(\nu)$, and $\delta_{m'}^{\mus}(\nu)$ is strictly increasing in $m'$.  This is only possible if $\DLP_{H_m}^{<r}(\nu) < \delta_m^{\mus}(\nu)$.
\end{proof}

\section{Reduction to $\FF_0$ and $\FF_1$}\label{sec-reduction}

Finally, we deduce results on an arbitrary Hirzebruch surface $\F_e$ by reducing our problems to the del Pezzo case.

\subsection{The reduction} In this section, consider a Hirzebruch surface $\FF_e$ with $e\geq 2$.  Write $\Pic \F_e = \Z E \oplus \Z F$ and $\Pic \F_{e-2} = \Z E' \oplus \Z F'$ for clarity, and for $m\in \QQ$ let $H_m\in K(\F_e)\te \QQ$ and $H_m'\in K(\F_{e-2})\te \QQ$ be the usual polarizations.  Consider the linear map $$\pi:K(\F_e)\te \QQ\to K(\F_{e-2})\te \QQ$$ given by the formula $$\pi (r,aE+bF,d) = (r,aE'+(b-a)F',d)$$ in $(\ch_0,\ch_1,\ch_2)$-coordinates.  (We will see later that $\pi$ carries integral characters to integral characters.)  Since $\pi$ only affects the $\ch_1$-term, we may also abuse notation and write $$\pi(aE+bF) = aE' + (b-a)F'$$ for a class $aE+bF \in \Pic(\F_e)\te \QQ$.    A different normalization of our polarization is useful.  For $m\in \QQ$, let \begin{align*} A_m &= -\frac{1}{2}K_{\F_e} + mF = E + \left(m+\frac{e}{2}+1\right)F = H_{m-\frac{e}{2}+1} \\ A_{m}' &= -\frac{1}{2}K_{\F_{e-2}}+mF' = E'+\left(m+\frac{e}{2}\right)F' = H_{m-\frac{e}{2}+2}'.\end{align*}   Thus, $A_m \in \Pic \F_e \te \QQ$ and $A_m' \in \Pic \F_{e-2}\te \QQ$ are both ample if and only if $m > \frac{e}{2}-1$. Write $\lceil A_m \rceil$ and $\lceil A_m'\rceil$ for the line bundles \begin{align*}\lceil A_m\rceil &= E + \left\lceil m + \frac{e}{2}+ 1 \right\rceil F\\ \lceil A_m'\rceil &= E'+\left\lceil m+\frac{e}{2}\right\rceil F'\end{align*} Observe that $\pi$ carries the class of $A_m$ (resp. $\lceil A_m\rceil$) to the class of $A_m'$ (resp. $\lceil A_m'\rceil$).  In this section we prove the following theorem.

\begin{theorem}\label{thm-FeInduct}
Let $\bv\in K(\F_e)$ satisfy $\Delta(\bv) \geq 0$, and suppose $m\in \QQ$.   Then there is an $\lceil A_m \rceil$-prioritary sheaf of character $\bv$ on $\F_e$ if and only if there is an $\lceil A_m'\rceil$-prioritary sheaf of character $\pi(\bv)$ on $\F_{e-2}$.

Additionally assume $m > \frac{e}{2}-1$ and that there are $\lceil A_m\rceil$ prioritary sheaves of character $\bv$ on $\F_e$.   Then the $A_m$-Harder-Narasimhan filtration of a general sheaf in $\cP^{\F_e}_{F}(\bv)$ has factors of characters $\bgr_1,\ldots,\bgr_\ell$ if and only if the $A_m'$-Harder-Narasimhan filtration of a general sheaf in $\cP^{\F_{e-2}}_{F}(\pi(\bv))$ has factors of characters $\pi( \bgr_1),\ldots, \pi(\bgr_\ell)$.
\end{theorem}

In the particular case of length 1 Harder-Narasimhan filtrations, we immediately conclude the following result.  Since $A_m'$ is ample whenever $A_m$ is ample, it follows that a solution to the existence problem for $\F_0$ and $\F_1$ gives a solution to the existence problem on all Hirzebruch surfaces $\F_e$.

\begin{corollary}\label{cor-translateModuli}
Let $\bv \in K(\F_e)$ and let $m > \frac{e}{2}-1$.  Then $M_{A_m}^{\F_e}(\bv)$ is nonempty if and only if $M_{A_m'}^{\F_{e-2}}(\pi(\bv))$ is nonempty.
\end{corollary}

We saw that there is a numerical criterion for the existence of prioritary sheaves, and a numerical criterion for a list of characters to be the generic Harder-Narasimhan filtration.  Thus we simply have to check that these criterions transform appropriately under the map $\pi$.

\begin{lemma}\label{lem-piAction}
The map $\pi:K(\F_e)\to K(\F_{e-2})$ has the following properties.
\begin{enumerate}
\item It preserves the intersection pairing: $c_1(\bv_1) \cdot c_1(\bv_2) = c_1(\pi(\bv_1)) \cdot c_1(\pi(\bv_2))$.
\item It preserves discriminants: $\Delta(\bv) = \Delta(\pi( \bv)).$
\item It preserves the class of the canonical bundle and trivial bundle:  we have $\pi (\ch K_{\F_e}) = \ch K_{\F_{e-2}}$ and $\pi(\ch \OO_{\F_e}) = \ch \OO_{\F_{e-2}}$.  
\item It preserves Euler characteristics $\chi(\bv)$ and Euler characteristics of pairs $\chi(\bv,\bw)$.  In particular, $\pi$ carries integral characters to integral characters, and it carries primitive characters to primitive characters.
\item We have $\pi(\ch (A_m)^{\te t}) = \ch (A_m')^{\te t}$.  Therefore $\pi$ transforms $A_m$-slopes to $A_m'$-slopes: $$\mu_{A_m} (\bv) = \mu_{A_m'}(\pi(\bv)).$$ It also transforms $A_m$-Hilbert polynomials to $A_m'$-Hilbert polynomials: $$\chi((A_m^*)^{\te t}, \bv) = \chi((A_{m}'^*)^{\te t},\pi(\bv))$$

\item For $n\in \Z$, it carries the character of the direct sum $$\OO_{\F_e}(-E+(n-1)F)^A \oplus \OO_{\F_e}^B \oplus \OO_{\F_e}(-F)^C$$ to the character of the direct sum $$\OO_{\F_{e-2}}(-E'+nF')^A \oplus \OO_{\F_{e-2}}^B \oplus \OO_{\F_{e-2}}(-F')^C.$$
\end{enumerate}
\end{lemma}
\begin{proof}
(1) Say $c_1(\bv_1) = a_1E + b_1F$ and $c_1(\bv_2) = a_2E + b_2F$.  Then  \begin{align*} c_1(\pi (\bv_1))\cdot c_1(\pi(\bv_2)) &= (a_1 E'+(b_1-a_1)F')\cdot (a_2 E' + (b_2-a_2)F')\\ &= a_1(b_2-a_2)+a_2(b_1-a_1)-(e-2)(a_1a_2)\\
&=a_1b_2+a_2b_1 - ea_1a_2\\
&= (a_1E+b_1F)(a_2E+b_2F)\\
&= c_1(\bv_1)\cdot c_1(\bv_2).
\end{align*}

(2) The map preserves $\ch_0$, $\ch_2$, and the intersection pairing, so since $$\Delta(\bv) = \frac{\ch_1(\bv)^2}{2\ch_0(\bv)} - \frac{\ch_2(\bv)}{\ch_0(\bv)}$$ it follows that it also preserves discriminants.

(3) Clear by a direct computation.

(4) Follows from Riemann-Roch and (1), (2), (3).

(5) Use (1), (2), and (4).

(6) It is clear that $\pi$ acts on the factors in the indicated way, and the result follows by linearity.
\end{proof}

The proof of Theorem \ref{thm-FeInduct} is now easy.

\begin{proof}[Proof of Theorem \ref{thm-FeInduct}]
For the first part, without loss of generality we may assume $m\in \QQ$ is such that $A_m$ and $A_m'$ are integral.  Let \begin{align*}
 n &= m-\frac{e}{2}+1,
\end{align*}
so that $A_m = \lceil A_m\rceil = H_n$ and $A_m' = \lceil A_m'\rceil = H'_{n+1}$. Suppose there is an $H_n$-prioritary sheaf $\cV$ of character $\bv$ on $\F_e$.  By Proposition \ref{prop-triangle}, after taking twists and/or duals of $\bv$ there is an $H_n$-prioritary direct sum of line bundles $$\cW = \OO_{\F_e}(-E+(n-1)F)^A \oplus \OO_{\F_e}^B \oplus \OO_{\F_e}(-F)^C$$ with $r(\cW) = r(\cV)$, $c_1(\cW) = c_1(\cV)$, and $\Delta(\cW) \leq \Delta(\cV)$.  Then $$\cW' = \OO_{\F_{e-2}}(-E'+nF')^A \oplus \OO_{\F_{e-2}}^B \oplus \OO_{\F_{e-2}}(-F')^C$$ is an $H'_{n+1}$-prioritary bundle with $r(\cW') = r(\pi(\bv))$, $c_1(\cW') = c_1(\pi(\bv))$, and $$\Delta(\cW') = \Delta(\cW) \leq \Delta(\cV).$$ By taking elementary modifications of $\cW'$ we get an $H_{n+1}'$-prioritary sheaf of character $\pi(\bv)$, as required.  A similar argument shows that if there is an $H'_{n+1}$-prioritary sheaf of character $\pi(\bv)$ on $\F_{e-2}$, then there is an $H_n$-prioritary sheaf of character $\bv$ on $\F_e$.

We prove the second part by induction on the rank.  The result is clear for rank one characters, so suppose it holds for characters of rank less than $r(\bv)$.  Let $m\in \QQ$ satisfy $m > \frac{e}{2}-1$, and let $\bgr_1,\ldots,\bgr_\ell$ (resp. $\bgr_1',\ldots,\bgr_{\ell'}'$) be the characters of the factors in the $A_m$- (resp. $A_m'$-) Harder-Narasimhan filtration of a general sheaf in $\cP_{\lceil A_m\rceil}^{\F_e}(\bv)$ (resp. $\cP^{\F_{e-2}}_{\lceil A_m'\rceil}(\pi(\bv))$).  We have to show that $\ell = \ell'$ and $\pi(\bgr_i) = \bgr_i'$.  If $\ell = \ell' =1$ then this is immediate, so the remaining possibilities are that $\ell \geq 2$ and/or $\ell' \geq 2$.

First suppose $\ell \geq 2$.  In this case $\bgr_1 + \cdots + \bgr_\ell = \bv$, the corresponding reduced $A_m$-Hilbert polynomials $q_1,\ldots,q_\ell$ satisfy $q_1 > \cdots > q_\ell$, by Lemma \ref{lem-HNclose} we have $\mu_{A_m}(\bgr_1)-\mu_{A_m}(\bgr_\ell) \leq 1$, and $\chi(\bgr_i,\bgr_j) = 0$ for $i < j$ by Lemma \ref{lem-HNorthogonal}.  By Lemma \ref{lem-piAction}, the transformed list $\pi(\bgr_1),\ldots,\pi(\bgr_\ell)$ still satisfies all these properties for the character $\pi(\bv)$ and polarization $A_m'$.  By induction on the rank, we know that the moduli spaces $M_{A_m'}^{\F_{e-2}}(\pi(\bgr_i))$ are nonempty. Therefore by Theorem \ref{thm-HNcriterion}, these are the characters of the factors of the $A_m'$-Harder-Narasimhan filtration of a general sheaf in $\cP_{\lceil A_m'\rceil}^{\F_{e-2}}(\bv)$, and we conclude $\ell = \ell'$ and $\pi (\bgr_i) = \bgr_i'$.

If instead $\ell' \geq 2$, we repeat the above argument but use $\pi^{-1}$ instead of $\pi$.   This shows that $\ell = \ell'$ and $\bgr_i = \pi^{-1}(\bgr_i').$
\end{proof}

It is useful to have the result analogous to \ref{cor-translateModuli} for slope stability.

\begin{corollary}\label{cor-highermus}
Let $\bv \in K(\F_e)$ and let $m > \frac{e}{2}-1$.  Then there is a $\mu_{A_m}$-stable sheaf of character $\bv$ on $\F_e$ if and only if there is a $\mu_{A_m'}$-stable sheaf of character $\pi(\bv)$ on $\F_{e-2}$.

If the generic stability interval $I_{\pi(\bv)}$ of $\pi(\bv)$ is $(m_0,m_1)$, then the generic stability interval  $I_\bv$ of $\bv$ is $(0,m_1-1)$.
\end{corollary}

In particular, if $e\geq 2$ and there are slope-stable sheaves of character $\bv$ then the generic stability interval $I_{\bv}$ is always an interval of the form $(0,m)$ with $m\in (0,\infty]$.  This is an analog of Corollary \ref{cor-KstabilityEasy}, with the anticanonical class $-\frac{1}{2}K_{\F_e} = -H_{1-\frac{e}{2}}$ no longer being ample.

\begin{proof}[Proof of Corollary \ref{cor-highermus}]
We simultaneously prove the second claim by induction on $e$.

$(\Rightarrow)$  Suppose there is a $\mu_{A_m}$-stable sheaf $\cV$ of character $\bv$ on $\F_e$.  Then for $\epsilon>0$ small it is  $\mu_{A_{m+\epsilon}}$-stable, so there are $A_{m+\epsilon}'$-semistable sheaves on $\F_{e-2}$ of character $\pi(\bv)$.  We claim there are $\mu_{A_{m+\epsilon}'}$-stable sheaves of character $\pi(\bv)$.  There are three cases to consider, depending on $\Delta = \Delta(\bv)$.

\emph{Case 1: $\Delta > \frac{1}{2}$.}  Since $A_{m+\epsilon}'$ is generic, Propositions \ref{prop-ssIMPs} and \ref{prop-sIMPmus} show there are $\mu_{A_{m+\epsilon}'}$-stable sheaves of character $\pi(\bv)$.

\emph{Case 2: $\Delta = \frac{1}{2}$.}  Since there is a $\mu_{A_{m+\epsilon}}$-stable sheaf $\cV$ of character $\bv$, the moduli space $M_{A_{m+\epsilon}}(\bv)$ is smooth at $\cV$ and irreducible of dimension $1 = r^2(2\Delta - 1) + 1$.  The character $\bv$ must then be primitive, since if $\bv = n \bv'$ with $\bv'$ primitive, then Proposition \ref{prop-divideSpace} shows there are $\mu_{A_{m+\epsilon}}$-stable sheaves of character $\bv'$ (here we again use that $A_{m+\epsilon}$ is generic).  So, the moduli space $M_{A_{m+\epsilon}}(\bv')$ is smooth of dimension $1$.  Taking direct sums of such sheaves, we see that every sheaf in $M_{A_{m+\epsilon}}(\bv)$ must be strictly semistable, a contradiction.  Therefore $\bv$ is primitive.  But then by Lemma \ref{lem-piAction} (4), we find that $\pi(\bv)$ is also primitive, and it follows that there are $\mu_{A_{m+\epsilon}}$-stable sheaves of character $\pi(\bv)$.

\emph{Case 3: $\Delta < \frac{1}{2}$.}  By Lemma \ref{lem-excFacts} (4), $\cV$ must be an exceptional bundle, $\bv$ is primitive, $\pi(\bv)$ is primitive, and there are $\mu_{A_{m+\epsilon}'}$-stable sheaves of character $\pi(\bv)$.

Now in any case, let $\cV'$ be a general $\mu_{A_{m+\epsilon}'}$-stable sheaf of character $\pi(\bv)$.  By induction (or by Corollary \ref{cor-KstabilityEasy} if $e-2\in \{0,1\}$) we know that $\cV'$ is additionally slope stable for any polarization between $H_1'=A_{\frac{e}{2}-1}'$ and $A_{m+\epsilon}'$.  The polarization $A_m'$ is between these, so $\cV'$ is $\mu_{A_m'}$-stable.  Thus there are $\mu_{A_m'}$-stable sheaves of character $\pi(\bv)$, and we must have $m < m_1-1$.  We conclude that the stability interval of $\bv$ is contained in $(0,m_1-1)$.

($\Leftarrow$)   Let $\cV'$ be a $\mu_{A_m'}$-stable exceptional bundle on $\F_{e-2}$ of character $\pi(\bv)$.  Then for $\epsilon >0$ small it is both $\mu_{A_{m-\epsilon}'}$- and $\mu_{A_{m+\epsilon}'}$-stable, so there are $A_{m-\epsilon}$- and $A_{m+\epsilon}$-semistable bundles on $\F_e$ of character $\bv$.  By the same analysis as in the other direction we see that there are $\mu_{A_{m-\epsilon}}$- and $\mu_{A_{m+\epsilon}}$-stable sheaves of character $\bv$, and therefore there are sheaves that simultaneously $\mu_{A_{m-\epsilon}}$- and $\mu_{A_{m+\epsilon}}$-stable.  They are $\mu_{A_m}$-stable, too.  By induction on $e$ (using Corollary \ref{cor-KstabilityEasy} if $e-2\in \{0,1\}$), the stability interval of $\pi(\bv)$ contains $[1,m_1)$, and therefore there are $\mu_{A_m}$-stable sheaves of character $\bv$ on $\F_e$ for any $m\in (0,m_1-1)$.  So, the stability interval of $\bv$ contains $(0,m_1-1)$.
\end{proof}

Therefore, the function $\delta_{m,\F_e}^{\mus}(\nu)$ only needs to be studied in the del Pezzo case.

\begin{corollary}
If $e \geq 2$, $m > 0$, and $\nu\in \Pic(\F_e)\te \QQ$, then $$\delta_{m,\F_e}^{\mus}(\nu) = \delta_{m+1,\F_{e-2}}^{\mus}(\pi(\nu)).$$
\end{corollary}

Monotonicity in the polarization then follows from the del Pezzo case.  (See Corollary \ref{cor-deltaMonotone}.)

\begin{corollary}\label{cor-deltaMonotoneHigher}
If $e \geq 2$, $0 < m < m'$, and $\nu\in \Pic(\F_e)\te \QQ$, then $$\delta_{m}^{\mus}(\nu) \leq \delta_{m'}^{\mus}(\nu).$$
\end{corollary}

The results of \S \ref{sec-HNKronecker} can all be translated to get similar results on $\F_e$ for $e\geq 2$.  We do not dwell on this here.

\subsection{Exceptional bundles on $\F_e$}

The stable exceptional bundles on $\F_e$ can now be described by passing to the del Pezzo case.  Since Theorem \ref{thm-stabilityInterval} computes the stability intervals of exceptional bundles on $\F_0$ and $\F_1$, the next result computes the stability interval of any exceptional bundle on a Hirzebruch surface.

\begin{corollary}
Let $\bv\in K(\F_e)$ be a potentially exceptional character, where $e\geq 2$, and let $m > \frac{e}{2}-1$.  There is a $\mu_{A_m}$-stable exceptional bundle $\cV$ on $\F_e$ of character $\bv$ if and only if there is an $\mu_{A_m'}$-stable exceptional bundle $\cV'$ on $\F_{e-2}$ of character $\pi(\bv)$.

If $I_{\cV'} = (m_0,m_1)$, then $I_{\cV} = (0,m_1-1)$.
\end{corollary}

\begin{proof}
Both $\bv$ and $\pi(\bv)$ are potentially exceptional, and a $\mu_{A_m}$-stable (resp. $\mu_{A_m'}$-stable) bundle of character $\bv$ (resp. $\pi(\bv)$) is exceptional.  So this follows from Corollary \ref{cor-highermus}.
\end{proof}

Repeatedly applying the map $\pi$ to reduce all the way to $\F_0$ or $\F_1$, we have the following.

\begin{corollary}If $e = 2k+e'$ with $e'\in \{0,1\}$, there is a bijection between the exceptional bundles on $\F_e$ which are slope-stable for some polarization and the exceptional bundles on $\F_{e'}$ which are $\mu_{H_k}$-stable.
\end{corollary}

The sets of characters of $\mu_{A_m}$-stable exceptional bundles are preserved by $\pi$, so it follows that the Dr\'ezet-Le Potier functions also transform appropriately. 

\begin{corollary}
Let $e \geq 2$, let $r\geq 2$, and let $\nu \in \Pic(\F_e) \te \QQ$.  For any $m>\frac{e}{2}-1$ we have $$\DLP_{A_m,\F_e}^{<r} (\nu) = \DLP_{A_m',\F_{e-2}}^{<r}(\pi(\nu)) \qquad \textrm{and} \qquad \DLP_{A_m,\F_e}(\nu) = \DLP_{A_m',\F_{e-2}}(\pi(\nu)).$$
\end{corollary}

Consequently, the monotonicity properties of $\DLP$ also carry over from Proposition \ref{prop-DLPmonotone}.

\begin{corollary}
Let $e \geq 2$, let $r\geq 2$, and let $\nu\in \Pic(\F_e) \te \QQ$.  If $0 < m \leq m'$, then $$\DLP_{H_m}^{<r}(\nu)\leq \DLP_{H_{m'}}^{<r}(\nu) \qquad \textrm{and} \qquad \DLP_{H_m}(\nu) \leq \DLP_{H_{m'}}(\nu).$$
\end{corollary}

Corollary \ref{cor-DLP1/2} has the following final generalization.

\begin{corollary}
Let $\nu\in \Pic(\F_e)\te \QQ$ and  let $H$ be an arbitrary polarization.  Then $$\DLP_{H}(\nu) \geq \frac{1}{2}.$$  
\end{corollary}

The Dr\'ezet-Le Potier functions for $\F_e$ can also be used to inductively compute the $\mu_H$-stable exceptional bundles on $\F_e$ directly, without passing to the del Pezzo case.  This generalizes Corollary \ref{cor-DLPexcdelPezzo}.

\begin{corollary}\label{cor-DLPExceptional}
Let $H$ be an arbitrary polarization.  If $\bv = (r,\nu,\Delta)\in K(\F_e)$ is potentially exceptional, then there is a $\mu_H$-stable exceptional bundle $\cV$ of character $\bv$ if and only if $$\Delta \geq \DLP_{H}^{<r}(\nu).$$
\end{corollary}

Corollary \ref{cor-deltaDLP01} also generalizes immediately.

\begin{corollary}\label{cor-deltaDLPe}
Let $\nu\in \Pic(\F_e)\te \QQ$, and let $m>0$.  Then $$\delta_m^{\mus}(\nu) \geq \DLP_{H_m}(\nu).$$
\end{corollary}

\begin{example}
The correspondence between exceptional bundles on $\F_e$ and $\F_{e-2}$ only holds for the slope-stable exceptional bundles.  

For example, on $\F_4$ consider the potentially exceptional character $\bv = (r,\nu,\Delta) = (3,\frac{1}{3}E+F,\frac{4}{9})$.  It transforms to the character $\bw = \pi^2(\bv) = (3,\frac{1}{3}E+\frac{1}{3}F,\frac{4}{9})$ of an exceptional bundle $\cW$ on $\F_0$, with stability interval $I_{\cW} = (\frac{1}{2},2)$.  Thus there are no slope-stable sheaves on $\F_4$ of character $\bv$, for any polarization $H$.  Furthermore, there is no exceptional bundle on $\F_4$ of character $\bv$.  Indeed, we use Corollary \ref{cor-prioritaryRho} to compute $\rho_\gen(\bv) = 1$, but if there were an exceptional bundle of character $\bv$ we would have $\rho_\gen(\bv) \geq 2$ by Proposition \ref{prop-excPrior} (1).  Although $\pi$ carries potentially exceptional characters to potentially exceptional characters, it can carry non-exceptional characters to exceptional characters.

We make the following conjecture.  If it is true, the exceptional bundles on $\F_e$ can be readily determined from the stability intervals of exceptional bundles on $\F_0$ and $\F_1$.

\begin{conjecture}
Let $e\geq 2$.  If $\cV$ is an exceptional bundle on $\F_e$, then $\cV$ is $\mu_{H_{\epsilon}}$-stable for $\epsilon >0$ sufficiently small.  
\end{conjecture}

Arguments as above and computer computations of stability intervals (as in Example \ref{ex-stabilityIntervals}) show that the conjecture is true for exceptional bundles of rank less than $107$.  The first potential counterexample is given by the character $$(r,\nu,\Delta) = \left(107,\frac{25}{107}E+\frac{76}{107}F,\frac{5724}{11449}\right)$$ on $\F_3$; the stability interval is empty, so to verify the conjecture in this case one needs to show there is no exceptional bundle of this character.
\end{example}

\bibliographystyle{plain}

\end{document}